\documentclass[aap]{imsart}

\RequirePackage{amsmath,amsfonts,amssymb,mathtools}
\RequirePackage{amsthm}
\RequirePackage[numbers]{natbib}
\RequirePackage{dsfont}
\RequirePackage{eufrak}
\RequirePackage{mathrsfs}
\RequirePackage{enumitem}
\RequirePackage{hyperref}
\RequirePackage{ulem}

\startlocaldefs
\newcommand{\dd}{\mathrm{d}}
\newcommand{\ee}{\mathrm{e}}
\newcommand{\prob}{\mathsf{P}}
\newcommand{\probd}{\mathbb{P}}
\newcommand{\Ex}{\mathsf{E}}
\newcommand{\Exd}{\mathbb{E}}
\newcommand{\Rl}{\mathbb{R}}
\newcommand{\Cm}{\mathbb{C}}
\newcommand{\pae}{\mathbb{P}\mbox{-a.e.\ }\omega}
\newcommand{\N}{\mathbb{N}}
\newcommand{\cov}{\mathsf{cov}}
\newcommand{\var}{\mathsf{var}}

\newtheorem{assumption}{Assumption}[section]
\newtheorem{proposition}{Proposition}[section]
\newtheorem{theorem}{Theorem}[section]
\newtheorem{lemma}{Lemma}[section]
\newtheorem{corollary}{Corollary}[section]
\newtheorem{remark}{Remark}[section]

\DeclareMathOperator*{\limsmallspace}{\lim~}

\DeclareMathOperator*{\limspace}{\lim~~}
\DeclareMathOperator*{\maxp}{\max\phantom{p}}

\endlocaldefs

\begin{document}

\begin{frontmatter}
%%%%%%%%%%%%%%%%%%%%%%%%%%%%%%%%%%%%%%%%%%%%%%
%%                                          %%
%% Enter the title of your article here     %%
%%                                          %%
%%%%%%%%%%%%%%%%%%%%%%%%%%%%%%%%%%%%%%%%%%%%%%
\title{Concentration and fluctuation phenomena \\ in the localized phase of the pinning model}
%\title{A sample article title with some additional note\thanksref{T1}}
\runtitle{On the localized phase of the pinning model}
%\thankstext{T1}{A sample of additional note to the title.}

\begin{aug}
%%%%%%%%%%%%%%%%%%%%%%%%%%%%%%%%%%%%%%%%%%%%%%%
%% Only one address is permitted per author. %%
%% Only division, organization and e-mail is %%
%% included in the address.                  %%
%% Additional information can be included in %%
%% the Acknowledgments section if necessary. %%
%% ORCID can be inserted by command:         %%
%% \orcid{0000-0000-0000-0000}               %%
%%%%%%%%%%%%%%%%%%%%%%%%%%%%%%%%%%%%%%%%%%%%%%%
\author[A]{\fnms{Giambattista}~\snm{Giacomin}\ead[label=e1]{giacomin@math.unipd.it}}
\and
\author[B]{\fnms{Marco}~\snm{Zamparo}\ead[label=e2]{marco.zamparo@uniupo.it}}
%%%%%%%%%%%%%%%%%%%%%%%%%%%%%%%%%%%%%%%%%%%%%%
%% Addresses                                %%
%%%%%%%%%%%%%%%%%%%%%%%%%%%%%%%%%%%%%%%%%%%%%%
\address[A]{Universit\`a di Padova, Dipartimento di Matematica Tullio Levi--Civita, 
Via Trieste 63, 35121 Padova, Italy\printead[presep={,\ }]{e1}}

\address[B]{Universit\`a degli Studi del Piemonte Orientale, Dipartimento di Scienze e Innovazione Tecnologica, 
  Viale Teresa Michel 11, 15121 Alessandria, Italy\printead[presep={,\ }]{e2}}
\end{aug}

\begin{abstract}
We focus on the localized phase of pinning models with i.i.d.\ site
disorder on which we assume only that the moment generating function
is bounded in a neighborhood of the origin. We develop quantitative
correlation functions estimates for local observables that entail
quantitative $C^\infty$ estimates on the free energy density, showing
in particular that its regularity class is at least Gevrey-3 in the
whole localized phase.  We then explain how a quenched concentration
bound and the quenched Central Limit Theorem (CLT) on the number of
the pinned sites, i.e., the \textit{contact number}, can be extracted
from the regularity estimates on the free energy: this identifies the
thermal fluctuations of the contact number.  But the centering
sequence in the quenched CLT is random in the sense that it is
disorder dependent: we show that the (disorder induced) fluctuations
of the centering are on the same scale of the thermal fluctuations by
establishing a CLT, with a non degenerate variance, also for the
centering.  For what concerns the correlation and $C^\infty$
estimates, our work substantially generalizes and expands the analysis
in \cite{giacomin2006_1} that dealt with pinning models with
restrictive conditions on the disorder distributions and in which less
explicit, non uniform bounds were obtained.
\end{abstract}



\end{frontmatter}

%%%%%%%%%%%%%%%%%%%%%%%%%%%%%%%%%%%%%%%%%%%%%%
%% Please use \tableofcontents for articles %%
%% with 50 pages and more                   %%
%%%%%%%%%%%%%%%%%%%%%%%%%%%%%%%%%%%%%%%%%%%%%%
\tableofcontents

\section{The pinning model and its localized phase}
\label{sec:1}

\subsection{The pinning model}
\label{sec:intro-model}
Set $\N:=\{1,2,\ldots\}$ and $\N_0:=\N\cup\{0\}$ and let
$T_1,T_2,\cdots$ be i.i.d.\ random variables on a probability
space $(\mathcal{S},\mathfrak{S},\prob)$ valued in
$\N\cup\{\infty\}$. Put $S_0:=0$ and $S_i:=T_1+\cdots+T_i$ for
$i\in\N$. Regarding $S_0,S_1,\ldots$ as renewal times with
inter-arrival times $T_1,T_2,\ldots$, the number of renewals by
the time $n\in\N_0$ is
\begin{equation*}
  L_n:=\sup\big\{i\in \N_0~:~S_i\le n\big\}
\end{equation*}
and we denote by $S$ the renewal set $\{S_i\}_{i\in\N_0}$.

Given a parameter $h\in\Rl$ and a real sequence
$\omega:=\{\omega_a\}_{a\in\N_0}$, that we call \textit{charges} and
that model impurities or disorder present in the system, for every
$n\in\N_0$ we introduce the \textit{pinning model}
$\prob_{n,h,\omega}$ defined by the Gibbs change of measure
\begin{equation*}
  \frac{\dd\prob_{n,h,\omega}}{\dd\prob}:=\frac{1}{Z_{n,h}(\omega)}\ee^{\sum_{i=1}^{L_n}(h+\omega_{S_i})}\mathds{1}_{\{n\in S\}}\,,
\end{equation*}
provided that the \textit{partition function}
$Z_{n,h}(\omega):=\Ex[\ee^{\sum_{i=1}^{L_n}(h+\omega_{S_i})}\mathds{1}_{\{n\in
    S\}}]$ is not 0. Here and in what follows we make the usual
convention that empty sums are 0 and empty products are 1.  We refer
to \cite{giacomin2007,cf:G-SF,cf:dH,cf:Velenik} for an explanation of
the origin and relevance of pinning models in applied sciences,
notably physics and biology. In short, we just point out that the
renewal set is interpreted as the ensemble of points, the
\textit{contacts} or \textit{pinned sites}, at which a polymer of $n$
units, an interface, or another linear structure enters in contact
with a \textit{defect line}, receiving a penalty at the contact point
$a$ if $h+\omega_a<0$ or a reward if $h+\omega_a>0$.  The size of the
contact set $L_n$ is henceforth referred to as \textit{contact
  number}.

\smallskip

The assumptions on the law of inter-arrival times and the charges are
the following.

\begin{assumption}
  \label{assump:p}
 $p(t):=\prob[T_1=t]={\ell(t)}/{t^{\alpha+1}}$ for $t\in\N$ with
  $\alpha \ge 0$ and a slowly varying function at infinity
  $\ell$. Moreover, $p(t)>0$ for all $t\in\N$.
\end{assumption}

We recall that a real measurable function $\ell$, defined on
$[1,+\infty)$ in our case, is slowly varying at infinity if
  $\ell(z)>0$ for all sufficiently large $z$ and $\lim_{z \uparrow
    +\infty}\ell(\lambda z)/\ell(z)=1$ for every $\lambda>0$
  \cite{bingham1989}. Besides this, we require that $p(t)>0$, i.e.,
  $\ell(t)>0$, for all $t\in\N$, which in particular guarantees that
  $Z_{n,h}(\omega)>0$ for every $n$ and every $\omega$.  We believe
  that there is no substantial difficulty in dropping this further
  hypothesis, thus generalizing our results to the case in which
  $p(t)>0$ only for $t$ sufficiently large.  This however would add
  one more step of complexity to arguments that are already rather
  technical. We remark also that we do not require that $\sum_{t\in
    \N} p(t)=1-\prob[T_1=\infty]=1$, so the law $p(\cdot)$ is a
  sub-probability on $\N$, but this is just a presentation choice
  because it is well known that there is no loss of generality in
  assuming $\prob[T_1=\infty]=0$ (see, e.g., \cite[pages
    14-15]{cf:G-SF}).
  
\begin{assumption}
  \label{assump:omega}
  The disorder sequence $\omega:=\{\omega_a\}_{a\in\N_0}$ is sampled
  from a probability space $(\Omega,\mathcal{F},\probd)$ in such a way
  that the canonical projections $\omega\mapsto\omega_a$ form a
  sequence of i.i.d.\ random variables. Furthermore,
  $\int_\Omega\ee^{\eta|\omega_0|}\probd[\dd\omega]<+\infty$ for some
  number $\eta>0$ and $\int_\Omega\omega_0\probd[\dd\omega]=0$.
\end{assumption}

We say that the model is \textit{pure} or \textit{non disordered} if
$\probd[\{\omega=0\}]=1$, which is tantamount to the condition
$\int_\Omega\omega_0^2\,\probd[\dd\omega]=0$. In this case the model
is explicitly solvable (see \cite{giacomin2007,cf:G-SF,cf:dH}).

\smallskip

Under Assumptions \ref{assump:p} and \ref{assump:omega}, a
superadditivity argument shows that the limit defining the \textit{free
  energy density} (we will simply say \textit{free energy} henceforth)
\begin{equation*}
  f(h):=\lim_{n\uparrow\infty}\Exd\bigg[\frac{1}{n}\log Z_{n,h}\bigg]
\end{equation*}
exists and is finite for all $h\in\Rl$ (see \cite[Theorem
  4.6]{giacomin2007}), thus introducing a non-decreasing Lipschitz
convex function $f$ over $\Rl$.  The bounds
$Z_{n,h}(\omega)\ge\ee^{h+\omega_n}p(n)$ and
$Z_{n,h}(\omega)\ge\ee^{\sum_{a=1}^n(h+\omega_a)}p(1)^n$ give $f(h)\ge
\max\{0,h+\log p(1)\}$.  Putting $h_c:=\inf\{h\in\Rl:f(h)>0\}$, we
deduce that $-\infty\le h_c<+\infty$ and, when $h_c>-\infty$, that
$\lim_{h\downarrow h_c}f(h)=0$ and $f(h)=0$ for $h\le h_c$.  The bound
$Z_{n,h}(\omega)\le\ee^{\sum_{a=1}^n\max\{0,h+\omega_a\}}$ yields
$f(h)\le\int_\Omega \max\{0,h+\omega_0\}\probd[\dd\omega]$, which
shows that $\lim_{h\downarrow h_c}f(h)=0$ even when $h_c=-\infty$.  We
say that the model is
\begin{itemize}
\item \textit{delocalized}, or that it is in the \textit{delocalized phase}, if $h<h_c$ (when $h_c>-\infty$);
\item  \textit{critical} if $h=h_c$ (when $h_c>-\infty$);  
\item  \textit{localized}, or that it is in the \textit{localized phase}, if $h>h_c$.
\end{itemize}
This terminology is easily justified by exploiting the convexity
properties of the free energy that, together with the definition of
$h_c$, readily yield that the mean contact number
$\Ex_{n,h,\cdot}[L_n]=\partial_h\log Z_{n,h}$ satisfies
$\lim_{n\uparrow\infty}\Exd[\Ex_{n,h,\cdot}[L_n]]/n=0$ for $h<h_c$
(when $h_c>-\infty$) and
$\liminf_{n\uparrow\infty}\Exd[\Ex_{n,h,\cdot}[L_n]]/n>0$ for $h>h_c$.
Convexity also assures that the inferior limit is a limit except
possibly for a countable set of values of $h$: we will see that a
byproduct of our analysis is that it is a limit in full generality. We
refer to \cite{giacomin2007,cf:G-SF,cf:dH,cf:Velenik} for much more
precise results and extensive discussions, including several open
questions (for the delocalized phase we add \cite{cf:AZ14} and
references therein).

\smallskip

The results we propose here represent a progress in the understanding
of the localized phase of the pinning model.  We underline that much
of the existing literature focuses on the case in which
$\int_\Omega\ee^{z|\omega_0|}\probd[\dd\omega]<+\infty$ for every
$z>0$, and this implies that $h_c>-\infty$.  In fact, it suffices that
$c:=\log\int_\Omega\ee^{\omega_0}\probd[\dd\omega]<+\infty$ to have
$h_c\ge-c>-\infty$ because of the annealed bound, which is Jensen's
inequality:
 \begin{equation}
 \label{eq:annealedbound}
  \Exd\big[\log Z_{n,h}\big]\le\log \Exd\big[Z_{n,h}\big]= 
 \log \Ex \left[ \ee^{(h+c) L_n}\mathds{1}_{\{n\in S\}}\right]\le 0
 \end{equation}
for all $n\in\N_0$ and $h\le-c$. Next is a more general, almost optimal, result for
deciding whether $h_c>-\infty$ or $h_c=-\infty$.  We set
\begin{equation*}
  \varrho :=\sup\bigg\{z\ge 0\,:\int_\Omega\ee^{z\omega_0}\probd[\dd\omega]<+\infty\bigg\}\in [\eta, +\infty]\,,
\end{equation*}
where $\eta>0$ is introduced in Assumption~\ref{assump:omega}.

\begin{proposition}
\label{th:h_c}
 If $(\alpha+1)\varrho>1$, then $h_c>-\infty$. If 
 $(\alpha+1)\varrho<1$, then $h_c=-\infty$.
\end{proposition}

The first part of Proposition \ref{th:h_c} follows from the basic
fractional moment bound \cite{cf:Ton-AAP08}, while the second part is
obtained via a lower bound on the partition function by restricting to
the renewal trajectory that makes contacts at the points $a$ such that
$\omega_a > \lambda$ with a suitably large number $\lambda$ (see,
e.g., \cite[Theorem 3.18]{cf:notesQB}).  One can tackle the case
$(\alpha+1)\varrho=1$ and decide whether $h_c>-\infty$ or not in a
number of cases, but we are not aware of a necessary and sufficient
condition when $(\alpha+1)\varrho=1$.

We remark that in \cite{cf:LS2017} the analysis of a class of pinning
models with heavy-tailed disorder law is considered, addressing the
issue of disorder relevance, i.e., the effect of disorder at
criticality (see, e.g., \cite[Chapter~4]{cf:G-SF} for the notion of
disorder relevance). Sharper conditions of the tail behaviour of the
charges are imposed with respect to Assumption~\ref{assump:omega}, but
in the language of Proposition~\ref{th:h_c} we see that in
\cite{cf:LS2017} it is assumed that $\varrho>1$. Therefore, in
particular, \cite{cf:LS2017} is restricted to the case in which
$h_c>-\infty$, and this follows directly from \eqref{eq:annealedbound}
without the need of Proposition~\ref{th:h_c}.

\subsection{The main results}
\label{sec:mainresults}

Most of our results rely, directly or indirectly, on correlation
estimates that generalize analogous results proved in
\cite{giacomin2006_1}.  In spite of the fact that they are central to
our analysis and that ultimately they can be resumed by saying that
\textit{in the localized phase correlations decay exponentially fast},
they are rather technical and we refer the reader directly to
Section~\ref{sec:corr}.  Here we focus on the following consequences,
where of course requiring a result to hold for every $H\subset
(h_c,+\infty)$ closed (respectively, compact) is equivalent to
requiring it for every $H\subset (h_c,+\infty)$ with $\inf H>h_c$
(respectively, with $\inf H>h_c$ and $\sup H< +\infty$).

\begin{theorem}
  \label{th:Cinfty}
  The free energy $f$ is strictly convex and infinitely
  differentiable on $(h_c,+\infty)$, and the following property holds
  for $\pae$: for every compact set $H\subset (h_c,+\infty)$ and
  $r\in\N_0$
  \begin{equation*}
    \adjustlimits\lim_{n\uparrow\infty}\sup_{h\in H}\bigg|\frac{1}{n}\partial^r_h\log Z_{n,h}(\omega)-\partial_h^r f(h)\bigg|=
    \adjustlimits\lim_{n\uparrow\infty}\sup_{h\in H}\bigg|\Exd\bigg[\frac{1}{n}\partial^r_h\log Z_{n,h}\bigg]-\partial_h^r f(h)\bigg|=0\,.
\end{equation*}
Moreover, for every closed set $H\subset (h_c,+\infty)$ there exists a
constant $c>0$ such that for all $r\in\N$
 \begin{equation*}
 \sup_{h \in H} \Big| \partial_h^r f(h)\Big| \le c^r (r!)^3\,.
 \end{equation*}
\end{theorem}

Theorem~\ref{th:Cinfty} provides quantitative bounds to the
derivatives of the free energy $f$ for the first time, showing in
particular that it is of class Gevrey-3 on $(h_c,+\infty)$, which
means that for every $H\subset (h_c,+\infty)$ compact there exists a
constant $c>0$ such that $\sup_{h \in H}|\partial_h^r f(h)| \le c^r
(r!)^3$ for all $r\in\N$.  This theorem generalizes \cite[Theorem
  2.1]{giacomin2006_1} because of the weaker assumptions on the
charges, in fact we do not require that they satisfy some
concentration inequality, and expands it in two respects: a Gevrey
class to which the free energy belongs is identified and convergence
of the derivatives of the finite-volume free energy for typical
realizations of the charges is demonstrated.

Besides, we stress that Theorem~\ref{th:Cinfty}, like all the results
that follow, is uniform with respect to the free parameter $h$: we
identify a set of charges of full probability for which the results
hold for every $h>h_c$.

\begin{remark}\label{rem:Griffiths}
 {\rm In \cite{cf:KM2003} it is claimed that the pinning model
   exhibits a Griffiths singularity in the localized phase,
  i.e., a singularity originally discovered in dilute
     Ising models where the free energy is $C^\infty$ but
     non-analytic. Griffiths singularities are due to rare but
     arbitrarily large atypical region in the environment: imagine a
     pinning model in which the charges $\omega_a$ take only the
     values $-1$ and $1$ with equal probability. Then it is
     straightforward to see that the model is localized for $h$ equal
     to the critical value $h_-:=1-\log\prob[T_1<\infty]$ of the homogeneous model in which
     $\omega_a=-1$ for every $a$. But in infinite volume one can find
     arbitrarily large regions in which the charges are all equal to
     $-1$: it is therefore believable that these critical regions,
     albeit finite, may impact the analyticity of the free energy at
     $h=h_-$.  Mathematically, this is an open problem for the
     pinning model and for most disordered models, because it turns
     out to be very difficult, outside of dilute systems, to decouple
     these rare regions from the rest of the system. In
   \cite{cf:GG-MPAG2022}, motivated by the arguments in
   \cite{cf:KM2003}, a toy pinning model is introduced and studied:
   notably for $\alpha=1/2$ this model exhibits a Griffiths
   singularity whose precise behaviour is established, in particular
   the free energy is in the Gevrey-(3/2) class. In spite of not being
   completely implausible that the toy model result holds for the true
   model, establishing this appears to be very challenging and how the
   Gevrey class depends on the exponent $\alpha$ remains elusive even
   in the toy model.}
\end{remark}

We will then show that from Theorem~\ref{th:Cinfty} one can extract a
quenched concentration bound and a quenched Central Limit Theorem
(CLT) for the contact number in the localized phase, i.e., a
concentration bound and a CLT for the contact number in the localized
phase conditional on a typical realization of the charges. The first
two derivatives of the free energy in the localized phase, which are
the functions $\rho:=\partial_hf$ and $v:=\partial_h^2 f$ on
$(h_c,+\infty)$, play a special role in what follows. We shall refer
to $\rho$ as the \textit{contact density} since, recalling that
$\Ex_{n,h,\omega}[L_n]=\partial_h\log Z_{n,h}(\omega)$,
Theorem~\ref{th:Cinfty} shows that the following holds for $\pae$: for
every $h>h_c$
 \begin{equation}
   \lim_{n\uparrow\infty}\Ex_{n,h,\omega}\bigg[\frac{L_n}{n}\bigg]=\lim_{n\uparrow\infty}\Exd\Bigg[\Ex_{n,h,\cdot}\bigg[\frac{L_n}{n}\bigg]\Bigg]=\partial_hf(h)=:\rho(h)\,,
   \label{eq:contact_fraction}
 \end{equation}
confirming in addition that the mean contact fraction
$\Exd[\Ex_{n,h,\cdot}[L_n/n]]$ has a limit when $n$ goes to
infinity. Theorem~\ref{th:Cinfty} also states that $v$ is the scaled
limiting variance of $L_n$ in the localized phase. In fact, as
$\Ex_{n,h,\cdot}[(L_n-\Ex_{n,h,\cdot}[L_n])^2]=\partial_h^2\log
Z_{n,h}$, Theorem~\ref{th:Cinfty} implies that also the following is
valid for $\pae$: for every $h>h_c$
\begin{eqnarray}
  \nonumber
  &&\lim_{n\uparrow\infty}\Ex_{n,h,\omega}\bigg[\bigg(\frac{L_n-\Ex_{n,h,\omega}[L_n]}{\sqrt{n}}\bigg)^{\!\!2}\bigg]\\
  &&\qquad=\lim_{n\uparrow\infty}\Exd\Bigg[\Ex_{n,h,\cdot}\bigg[\bigg(\frac{L_n-\Ex_{n,h,\cdot}[L_n]}{\sqrt{n}}\bigg)^{\!\!2}\bigg]\Bigg]=\partial_h^2f(h)=:v(h)\,.
\label{eq:lim_vh}
\end{eqnarray}
We stress that $\rho(h)>0$ and $v(h)>0$ for all $h>h_c$, the former
being the hallmark of the localized phase, as seen in the previous
section, and the latter being an expression of the strict convexity
property of $f$.

\begin{theorem}
  \label{th:CLT+concentration}
  The following properties hold for $\pae$:
  \begin{enumerate}[label=({\roman*})]
%\begin{enumerate}[leftmargin=0.7 cm, itemsep=1ex,label=({\roman*})]
\item for every closed set $H\subset (h_c,+\infty)$ there exists a
  constant $\kappa_\omega>0$ such that for all $u\ge 0$ and $n\in\N$
 \begin{equation*}
   \sup_{h\in H}\prob_{n,h,\omega}\Big[\big|L_n-\Ex_{n,h,\omega}[L_n]\big|>u\Big] \le 2\exp\left\{-\kappa_\omega\frac{u^2}{n+u^{5/3}}\right\};
 \end{equation*}
    
 \item for every compact set $H\subset (h_c,+\infty)$
 \begin{equation*}
 \adjustlimits\limsmallspace_{n\uparrow\infty}\sup_{h\in H}\,\sup_{u\in\Rl}\,\Bigg|\prob_{n,h,\omega}\bigg[\frac{L_n-\Ex_{n,h,\omega}[L_n]}{\sqrt{n v(h)}}\le u\bigg]-
     \frac{1}{\sqrt{2\pi}}\int_{-\infty}^u\ee^{-\frac{1}{2}z^2}\dd z\Bigg|=0
\end{equation*}
with $v(h):=\partial_h^2 f(h)>0$.

 \end{enumerate}
\end{theorem}

Theorem~\ref{th:CLT+concentration} involves the random centering
$\Ex_{n,h,\cdot}[L_n]$. It is therefore natural to address the size of
this centering: as a matter of fact, without control on the centering
the CLT turns out to be void of interest.  To this aim we propose for
the centering variable a concentration bound, a control on the
expectation, existence of a limiting variance, and the CLT.

\begin{theorem}
  \label{th:centering}
  The following conclusions hold:
  \begin{enumerate}[label=({\roman*})]
%\begin{enumerate}[leftmargin=0.8 cm,itemsep=1ex,label=({\roman*})]
\item for every closed set $H\subset (h_c,+\infty)$ there exists a
  constant $\kappa>0$ such that for all $n\in\N$ and $u\ge 0$
 \begin{equation*}
\sup_{h\in H}\probd\bigg[\Big|\Ex_{n,h,\cdot}[L_n]-\Exd\big[\Ex_{n,h,\cdot}[L_n]\big]\Big|>u\bigg]\le 2\exp\left\{-\kappa \frac{ u^2}{n+u^{5/3}}\right\}\,;
 \end{equation*}
    
  \item for every $H\subset (h_c,+\infty)$ closed there exists a
    constant $c>0$ such that for all $n\in\N$
\begin{equation*}
  \sup_{h\in H}\Big|\Exd\big[\Ex_{n,h,\cdot}[L_n]\big]-\rho(h)n\Big|\le c
\end{equation*}
with $\rho(h):=\partial_hf(h)$;

\item  there exists a locally Lipschitz continuous function $w$
  on $(h_c,+\infty)$ such that for every compact set $H\subset (h_c,+\infty)$
   \begin{equation*}
     \adjustlimits\limsmallspace_{n\uparrow\infty}\sup_{h\in H}
     \Bigg|\Exd\Bigg[\bigg(\frac{\Ex_{n,h,\cdot}[L_n]-\Exd[\Ex_{n,h,\cdot}[L_n]]}{\sqrt{n}}\bigg)^{\!\!2}\Bigg]-w(h)\Bigg|=0\,;
   \end{equation*}
   
\item if $\int_\Omega\omega_0^2\, \probd[\dd\omega]>0$, then $w(h)>0$
  for all $h>h_c$ and for every $H\subset (h_c,+\infty)$ compact
\begin{equation*}
\adjustlimits\limsmallspace_{n\uparrow\infty}\sup_{h\in H}\,\sup_{u\in\Rl}\,\Bigg|\probd\bigg[\frac{\Ex_{n,h,\cdot}[L_n]-\Exd[\Ex_{n,h,\cdot}[L_n]]}{\sqrt{n w(h)}}\le u\bigg]
 -\frac{1}{\sqrt{2\pi}}\int_{-\infty}^u\ee^{-\frac{1}{2}z^2}\dd z\Bigg|=0\,.
\end{equation*}

\end{enumerate}
\end{theorem}

We highlight the very different nature of the CLTs in Theorems
\ref{th:CLT+concentration} and \ref{th:centering}: the fluctuations of
the contact number for a typical environment are thermal, whereas the
fluctuations of the centering are due to the disorder.  In the absence
of disorder the random centering $\Ex_{n,h,\cdot}[L_n]$ is of course
constant, so parts (i) and (iii) of Theorem \ref{th:centering} are
trivial, as well as Proposition \ref{prop:fluctuations_EL} below. On
the other hand, as soon as the model is disordered, Theorem
\ref{th:centering} assures among others that the limiting distribution
of the random centering is non degenerate, i.e., $w(h)>0$ for $h>h_c$.
The analysis of the variance of $\Ex_{n,h,\cdot}[L_n]$ and
establishing the positivity of $w(h)$ in the presence of disorder,
which are important preliminary results in view of formulating the
CLT, require a specific and demanding study.  In particular, the
positivity of $w(h)$ can be established in a rather straightforward
way when the charges are Gaussian variables via integration by parts,
but leaving the Gaussian framework is not at all straightforward.

\smallskip

Resorting to the concentration bound stated by Theorem
\ref{th:centering}, we are also able to provide an upper bound for the
magnitude of the fluctuations of the centering for typical
realizations of the charges, in the spirit of Hardy--Littlewood
estimates for random walks \cite{feller1943}.

\begin{proposition}
  \label{prop:fluctuations_EL}
  The following property holds for $\pae$: for every compact set
  $H\subset (h_c,+\infty)$ there exists a constant $c>0$ such that for
  all $n\in\N$
  \begin{equation*}
    \sup_{h\in H}\frac{|\Ex_{n,h,\omega}[L_n]-\rho(h) n|}{\sqrt{n\log n}}\le c\,.
  \end{equation*}
\end{proposition}

We complete the presentation of our results by an aspect that may
appear at first more specific and less central than the previous ones:
the definition of an \textit{alternative free energy} $\mu(h)$. This
impression, which is due to the fact that we have decided not to
present the decay of correlation estimates in the introduction, is
however false because they are the crucial building block of our
analysis and, in turn, their exponential decay relies on the strict
positivity of $\mu(h)$. As already explained in
\cite{cf:AZ96,giacomin2006_1}, $\mu(h)$ is another natural free energy
associated to our model, and we let the following statement introduce
it, where we use
$Z_{n,h}^{\textbf{-}}(\omega):=Z_{n,h}(\omega)\ee^{-h-\omega_n}$.

\begin{proposition}
  \label{prop:mu}
 For every $h\in \Rl$ the limit
 \begin{equation}
 \label{eq:mu-def}
 -\lim_{n \uparrow \infty} \frac 1n \log \Exd \bigg[\frac {1}{Z_{n,h}^{\textbf{-}}}\bigg] =: \mu(h)
 \end{equation}
exists  and defines a Lipschitz function $\mu$ with Lipschitz
  constant equal to $1$. Moreover, there exists a constant $c>0$ such
that $c\min\{f(h),f(h)^2\}\le \mu(h) \le f(h)$ for all $h$.  Hence,
$\mu(h)=0$ for $h \le h_c$ (when $h_c>-\infty$) and $\mu(h)>0$ for
$h>h_c$.
\end{proposition}

A direct consequence of the bounds relating $\mu(h)$ and $f(h)$ in
Proposition~\ref{prop:mu} is that the function $\mu$ is equivalent to
$f$ as far as detecting the phase transition is concerned, but whether
their behaviour approaching criticality is the same or not is an open
problem when  disorder is {relevant} (see \cite{cf:GTirrel2009},
notably Theorem~2.4).  This
point is further discussed in \cite{cf:CGT12}, and in
\cite{giacomin2006_1,cf:GTirrel2009,cf:CGT12} it is also explained why
both $1/\mu(h)$ and $1/f(h)$ are two natural correlation lengths for
the system: in fact, $\mu(h)$ directly enters the correlation bounds
we give and use starting from the next section. One of the important
properties of $\mu(h)$ is that it determines the maximal polymer
excursion $M_n:=\max\{T_1,\ldots,T_{L_n}\}$ in the localized phase as
follows.

\begin{proposition}
   \label{prop:maximal_excursion}
    The following property holds for $\pae$: for every $h>h_c$ the sequence
  \begin{equation*}
 \bigg\{\frac{M_n}{\log n}\bigg\}_{n\in\N}
  \end{equation*}
  converges in probability, with respect to the law
  $\prob_{n,h,\omega}$, to $1/\mu(h)$.
\end{proposition}

Propositions \ref{prop:mu} and \ref{prop:maximal_excursion} generalize
and improve part (1) and part (2), respectively, of \cite[Theorem
  2.5]{giacomin2006_1}.  While Proposition~\ref{prop:mu} generalizes
part (1) of this theorem only for the weaker assumptions on the
charges (see however Remark~\ref{rem:mu} below),
Proposition~\ref{prop:maximal_excursion} improves part (2) also
because it establishes convergence in probability for almost all the
realizations of the disorder: in fact, the upper bound on the size of
the maximal excursion in \cite{giacomin2006_1} is obtained only in
probability with respect to the disorder.  Finally, we note that the
proof of Proposition~\ref{prop:maximal_excursion} in Section
\ref{sec:maximal_excursion} yields a result that is uniform in $h$
belonging to any compact set $H\subset (h_c, +\infty)$, but we have
chosen to present here only the weaker result for readability.

\begin{remark}
 \label{rem:mu}
{\rm The definition of $\mu(h)$ in \cite{cf:AZ96,giacomin2006_1} is
  rather $-\lim_{n \uparrow \infty} (1/n) \log \Exd[ 1/Z_{n,h}]$. The
  two definitions are equivalent if and only if $\int_\Omega
  \ee^{-\omega_0} \probd [\dd \omega]< +\infty$, otherwise $\Exd
  \left[ 1/{Z_{n,h}}\right]=+\infty$. For several reasons, see in
  particular Corollary~\ref{cor:PT1}, considering $\Exd
  [1/Z_{n,h}^{\textbf{-}}]$ is more natural as well as
    necessary if $\int_\Omega \ee^{-\omega_0} \probd[\dd \omega]=
    +\infty$. We point out that the superadditivity argument in
    \cite{cf:AZ96,giacomin2006_1} for the existence of $\mu(h)$ does
    not work when $\int_\Omega \ee^{-\omega_0} \probd [\dd \omega]=
    +\infty$ and needs to be modified.}
 \end{remark}

\subsection{Further comments on the results and comparisons with  \cite{giacomin2006_1}}
 
Theorem \ref{th:Cinfty} generalizes and expands \cite[Theorem
  2.1]{giacomin2006_1}, while Propositions \ref{prop:mu} and
\ref{prop:maximal_excursion} generalize and improve analogous results
stated in \cite[Theorem 2.5]{giacomin2006_1} as already pointed
out. On the other hand, Theorems \ref{th:CLT+concentration} and
\ref{th:centering} are new, as well as Proposition
\ref{prop:fluctuations_EL}.  We list here a number of comments about
our findings:

\begin{enumerate}

\item 
In comparison to \cite{giacomin2006_1}, our hypotheses on disorder are
substantially weaker and completely explicit. In fact, in
\cite{giacomin2006_1} it is asked that disorder has the property that
Lipschitz convex functions of the charges satisfy a suitable
concentration inequality.  This property is used to show that
$\mu(h)>0$ for $h>h_c$ by means of a concentration bound for the
finite-volume free energy.  In particular, the concentration
inequality assumed in \cite{giacomin2006_1} requires (super-)Gaussian
tail decay of the charge distribution. Here we just rely on Assumption
\ref{assump:omega}: the boundedness of some exponential moments of the
disorder variable.  We achieve a concentration bound for the
finite-volume free energy via a version of the McDiarmid's inequality
for functions of independent random variables \cite{maurer2021}. In
\cite{cf:watbled} concentration for the finite-volume free energy has
been achieved with a generalization of Hoeffding's inequality and
requiring exponential boundedness for $\eta=1$ in the language of
Assumption \ref{assump:omega}.

\smallskip
   
\item The correlation estimates in Section~\ref{sec:corr} are the
  technical basis of our results. They build on the ideas in
  \cite{cf:taming95,giacomin2006_1}: \cite{cf:taming95} deals with
  disordered Ising models in high temperature or strong external field
  regime and exploits cluster expansion estimates, while
  \cite{giacomin2006_1} deals with pinning and copolymer models till
  the critical point, i.e., in the whole localized regime. We
  generalize and make more explicit the estimates in
  \cite{giacomin2006_1} to accommodate the more general set up and to
  achieve new results. In particular, with respect to
  \cite{giacomin2006_1}, the estimates are improved so that they are
  amenable to be used in a Birkhoff--sum approach, to which we resort
  here to prove the almost sure results with respect to disorder.
  Furthermore, they are also improved in the direction of stronger
  mixing estimates, like those in Lemma~\ref{lem:mixing} and Corollary
  \ref{cor:mixing}, that are needed for establishing the CLT for the
  centering variable.

  The mentioned Birkhoff--sum approach, which we use starting from
  Section \ref{sec:maximal_excursion}, turns out to be a powerful
  tool: in particular, it is crucial in establishing almost sure
  convergence of the free-energy derivatives in Theorem \ref{th:Cinfty},
  which in turn is the basis of the quenched CLT for the contact
  number, and in upgrading \cite[Theorem 2.5]{giacomin2006_1} to a
  full almost sure result in Proposition~\ref{prop:maximal_excursion}.
  We believe that the applications of such approach are not restricted
  to the results of this paper and that it could provide further
  insight into the localized phase of pinning models in future
  works. Moreover, the Birkhoff--sum approach might be of interest
  also for the study of pinning models subjected to correlated
  disorder with good mixing properties (see \cite{cf:CCP2019} and
  references therein).

\smallskip  

\item We stress once more that we have taken care of providing almost
  sure results that are uniform in the parameter $h$: by this we mean
  that we have identified a set of charges of full probability that
  works for every $h>h_c$. This is another byproduct of the
  Birkhoff--sum approach. The pinning model is usually presented as a
  model with two free parameters
  \cite{giacomin2007,cf:G-SF,cf:dH,cf:Velenik}: $h$ and an amplitude
  $\beta\ge 0$ of the charges, so $\omega_a$ enters the model via
  $\beta \omega_a$, and then one fixes $\int_\Omega \omega_0^2
  \,\probd[\dd \omega] =1$.  In this framework, the localized phase is
  a region of the parameter plane where the free energy is strictly
  positive. For the sake of simplifying the formulas, we have decided
  to set $\beta=1$, but we have not fixed the variance of the
  charges. There is therefore no loss of generality except that one
  may wonder about the uniformity of the results with respect to both
  $h$ and $\beta$ together: our approach generalizes in a
  straightforward way to obtain such a double uniformity.

\smallskip
  
\item Compared to \cite[Theorem 2.1]{giacomin2006_1} and to the
  results in \cite{cf:taming95}, our estimates in Theorem
  \ref{th:Cinfty} control the dependence on the bounds on (all) the
  derivatives of the free energy in such a way that a Gevrey class to
  which the free energy belongs is identified. The Gevrey class we
  find is the same for all the models we consider. Obtaining the
  optimal Gevrey class, probably dependent on details of the model
  such as the regular variation class of the inter-arrival law and the
  regularity properties of the law of the disorder, is a particular
  case of the very challenging issue of understanding Griffiths
  singularities.

\smallskip  
  
\item We prove a quenched CLT for the key observable of pinning
  models, i.e., the contact number.  Disorder induces a number of
  technical difficulties, but also a substantial one: since we want to
  prove the results for typical realizations of the environment, the
  centering constant $\Ex_{n,h,\omega}[L_n]$ depends on the
  realization $\omega$ of the environment. A full result includes
  therefore a control on the behaviour of $\Ex_{n,h,\omega}[L_n]$ and,
  a posteriori, one realizes that thermal and disorder fluctuations
  are of the same order, so neither one can be neglected. For this
  reason we establish a CLT also for this variable, as well as a
  Hardy--Littlewood type-estimate, and we show that the limiting
  distribution is degenerate only in the absence of disorder. Non
  degeneracy can be demonstrated in a rather straightforward way when
  disorder is Gaussian (via integration by parts), but our argument
  controls general disorder laws.

\smallskip
  
\item In the forthcoming work \cite{giacominzamparo} we are going to
  exploit the estimates we develop here, notably the quenched CLT for
  the contact number, to study the pinning model conditioned on the
  number of pinned sites. The aim is to obtain sharp results on the
  effect of disorder on the so called \textit{big-jump phenomenon} that
  has been widely studied in the absence of disorder (see
  \cite{giacomin2020,cf:VBB,zamparo2021} and references therein). A
  central role in \cite{giacominzamparo} is played by the fact that
  disorder smooths the pinning transition and first order transitions
  do not exist anymore in disorder pinning models
  \cite{caravenna2013,giacomin2006_2}: this is in strict analogy with
  the smoothing phenomenon in the Ising model \cite{cf:AW,cf:DingXia}
  with the notable difference that in pinning models one does know
  that the model still has a phase transition. With respect to these
  aspects we point out that a pinning model can be rewritten in binary
  Ising like variables, with many body interactions of all orders, as
  it is taken up and discussed in detail in \cite{giacominzamparo}.
\end{enumerate}

\subsection{Organization of the rest of the paper}

The rest of the paper is organized as follows. In Section
\ref{sec:basics} we discuss some basic facts of the model, such as the
fundamental property of conditional independence between consecutive
stretches of the polymer and a concentration bound for the
finite-volume free energy. The latter is then used to prove
Proposition \ref{prop:mu}. Section \ref{sec:corr} presents some
correlation estimates, a first application of which is the proof of
Proposition \ref{prop:maximal_excursion}. In Section
\ref{sec:regularity_estimates} we develop some regularity estimates
for the free energy, which allows us to prove Theorem \ref{th:Cinfty}
and, on the basis of general results about cumulants, also Theorem
\ref{th:CLT+concentration}.  Finally, in Section \ref{sec:CLTEL} we
study the cumulants of the centering variable involved in
Theorem~\ref{th:CLT+concentration}, with particular attention to the
mean and variance, proving Theorem \ref{th:centering} and Proposition
\ref{prop:fluctuations_EL}.

\section{Some basic facts}
\label{sec:basics}

\subsection{Conditional independence in the pinning model}

The renewal nature of the pinning model implies a conditional
independence between consecutive stretches of the polymer. We state
this fundamental property in terms of the random variables
$X_a:=\mathds{1}_{\{a\in S\}}$ for $a\in\N_0$, which
take value 1 at the contact sites and value 0 at the other
sites. Recall that $S:=\{S_i\}_{i\in\N_0}$ and note that $X_0=1$ as
$S_0:=0$, which can be interpreted by saying that there is a contact
on the left of any site $a\in\N$. Also note that the model can be
recast as
\begin{equation*}
  \frac{\dd\prob_{n,h,\omega}}{\dd\prob}=\frac{1}{Z_{n,h}(\omega)}\ee^{\sum_{a=1}^n(h+\omega_a)X_a}X_n
\end{equation*}
with partition functions
$Z_{n,h}(\omega)=\Ex[\ee^{\sum_{a=1}^n(h+\omega_a)X_a}X_n]$ and
$Z_{n,h}^{\textbf{-}}(\omega):=Z_{n,h}(\omega)\ee^{-h-\omega_n}=\Ex[\ee^{\sum_{a=1}^{n-1}(h+\omega_a)X_a}X_n]$. In
this setting the contact number reads $L_n=\sum_{a=1}^nX_a$.

The following lemma introduces the mentioned conditional independence,
whose standard proof is omitted (see, e.g., \cite[Proposition
  3.1]{zamparo2022}). Let $\mathcal{M}_n$ be the set of functions
$\phi:\{0,1\}^{n+1}\to\Cm$ such that $|\phi(x_0,\ldots,x_n)|\le 1$ for
all $x_0,\ldots,x_n\in\{0,1\}$ and denote by $\vartheta$ the left
shift that maps the sequence
$\omega:=\{\omega_a\}_{a\in\N_0}\in\Omega$ to
$\vartheta\omega:=\{\omega_{a+1}\}_{a\in\N_0}$.

\begin{lemma}
  \label{lem:fact}
For every $h\in\Rl$, $\omega\in\Omega$, integers $0\le a\le n$, and
functions $\phi\in\mathcal{M}_a$ and $\psi\in\mathcal{M}_{n-a}$
\begin{eqnarray}
  \nonumber
  &&\Ex_{n,h,\omega}\Big[\phi(X_0,\ldots,X_a)\psi(X_a,\ldots,X_n)\Big|X_a=1\Big]\\
  \nonumber
  &&\qquad=\Ex_{a,h,\omega}\big[\phi(X_0,\ldots,X_a)\big]\Ex_{n-a,h,\vartheta^a\omega}\big[\psi(X_0,\ldots,X_{n-a})\big]\,.
\end{eqnarray}
In particular, it follows that
\begin{eqnarray}
  \nonumber
  &&\Ex_{n,h,\omega}\Big[\phi(X_0,\ldots,X_a)\psi(X_a,\ldots,X_n)\Big|X_a=1\Big]\\
  \nonumber
  &&\qquad=\Ex_{n,h,\omega}\big[\phi(X_0,\ldots,X_a)\big|X_a=1\big]\Ex_{n,h,\omega}\big[\psi(X_a,\ldots,X_n)\big|X_a=1\big]\,.
\end{eqnarray}
\end{lemma}

\begin{remark}
         {\rm We will often refer to Lemma \ref{lem:fact} using the
           following equivalent unconditional versions:
\begin{eqnarray}
  \nonumber
  &&\Ex_{n,h,\omega}\Big[\phi(X_0,\ldots,X_a)X_a\psi(X_a,\ldots,X_n)\Big]\\
   \nonumber
   &&\qquad=\Ex_{n,h,\omega}\big[\phi(X_0,\ldots,X_a)X_a\big]\Ex_{n-a,h,\vartheta^a\omega}\big[\psi(X_0,\ldots,X_{n-a})\big]\\
   \nonumber
  &&\qquad=\Ex_{a,h,\omega}\big[\phi(X_0,\ldots,X_a)\big]\Ex_{n,h,\omega}\big[X_a\psi(X_a,\ldots,X_n)\big]\,.
\end{eqnarray}}
\end{remark}

Lemma \ref{lem:fact} is at the basis of a technical estimate that we
use all over the paper to compare configurations where a site is not a
contact with configurations where that site becomes a contact.  In
order to present such estimate in the next lemma, we observe that the
properties of slowly varying functions (see \cite[Theorem 1.2.1 and
  Proposition 1.3.6]{bingham1989}) imply that there exists a constant
$\xi>0$ such that for all $t,\tau\in\N$
\begin{equation}
  \frac{p(t+\tau)}{p(t)p(\tau)}\le\xi \min\{t^\xi,\tau^\xi\}\,.
  \label{eq:xi_def}
\end{equation}
Recall our convention that empty sums are equal to 0 and empty
products are equal to 1.

\begin{lemma}
  \label{lem:utile_per_tutto}
  For every $h\in\Rl$, $\omega:=\{\omega_b\}_{b\in\N_0}\in\Omega$,
  integers $0<a<n$, and non-negative real functions
  $\phi\in\mathcal{M}_{a-1}$ and $\psi\in\mathcal{M}_{n-a-1}$
  \begin{eqnarray}
    \nonumber
  &&\Ex_{n,h,\omega}\Big[\phi(X_0,\ldots,X_{a-1})(1-X_a)\psi(X_{a+1},\ldots,X_n)\Big]\\
  \nonumber
  &&\qquad\le \xi\ee^{-h-\omega_a}\min\big\{a^\xi,(n-a)^\xi\big\}\Ex_{n,h,\omega}\Big[\phi(X_0,\ldots,X_{a-1})X_a\psi(X_{a+1},\ldots,X_n)\Big]\,.
  \end{eqnarray}
  Furthermore, for each $\lambda>0$
  \begin{eqnarray}
    \nonumber
  &&\Ex_{n,h,\omega}\Big[\phi(X_0,\ldots,X_{a-1})\mathds{1}_{\{T_{L_a+1}\le \lambda\}}(1-X_a)\psi(X_{a+1},\ldots,X_n)\Big]\\
  \nonumber
  &&\qquad\le \xi\ee^{-h-\omega_a}\min\big\{\lambda^\xi,a^\xi,(n-a)^\xi\big\}\\
  \nonumber
  &&\qquad\quad\times\Ex_{n,h,\omega}\Big[\phi(X_0,\ldots,X_{a-1})\mathds{1}_{\{T_{L_a}+T_{L_a+1}\le \lambda\}}X_a\psi(X_{a+1},\ldots,X_n)\Big]\,.
\end{eqnarray}
\end{lemma}

\begin{proof}[Proof of Lemma~\ref{lem:utile_per_tutto}]
  It suffices to prove the second part of the lemma. In fact, the
  first part follows from the second one with a number $\lambda\ge n$
  since for $0<a<n$ we have $T_{L_a+1}\le n$ and $T_{L_a}+T_{L_a+1}\le
  n$ almost surely with respect to the polymer measure
  $\prob_{n,h,\omega}$.

\smallskip
  
Pick $h\in\Rl$, $\omega:=\{\omega_b\}_{b\in\N_0}\in\Omega$, integers
$0< a< n$, and $\lambda>0$. The lemma uses a standard trick to
introduce the indicator function that there is a contact at site $a$
at the expense of a multiplicative constant that depends on the charge
at $a$ and on a power of $a$ or $n-a$. In detail, we are going to
exploit the two identities
\begin{equation}
  \mathds{1}_{\{T_{L_a+1}\le \lambda\}}(1-X_a)=\sum_{i=0}^{a-1}\sum_{j=a+1}^\infty \mathds{1}_{\{j-i\le \lambda\}}X_i\prod_{k=i+1}^{j-1}
  (1-X_k)X_j
\label{eq:identity_a1}
\end{equation}
and
\begin{equation}
  \mathds{1}_{\{T_{L_a}+T_{L_a+1}\le \lambda\}}X_a=\sum_{i=0}^{a-1}\sum_{j=a+1}^\infty\mathds{1}_{\{j-i\le \lambda\}}X_i\prod_{k=i+1}^{a-1}(1-X_k)X_a\prod_{k=a+1}^{j-1}(1-X_k)X_j\,.
  \label{eq:identity_a2}
\end{equation}
The identity \eqref{eq:identity_a1} decomposes the event that the site
$a$ is in a polymer excursion of size at most $\lambda$, but there is
no contact in $a$, as a disjoint union of the events that there is a
contact at $i<a$, a contact at $j>a$ with $j-i\le \lambda$, and no contact
between $i$ and $j$.  Similarly, \eqref{eq:identity_a2} decomposes the
event that the site $a$ is a contact separating two excursions of
total size at most $\lambda$ according to the positions $i$ and $j$ of
the first contacts on the left and on the right of $a$, respectively.
Also in this case $j-i \le \lambda$.

Fix non-negative real functions $\phi\in\mathcal{M}_{a-1}$ and
$\psi\in\mathcal{M}_{n-a-1}$, and put for brevity
$\Phi:=\phi(X_0,\ldots,X_{a-1})$ and
$\Psi:=\psi(X_{a+1},\ldots,X_n)$. For $i<a<j$ also define the random
quantities $\Phi_i:=\phi(X_0,\ldots,X_i,0,\ldots,0)$ and
$\Psi_j:=\psi(0,\ldots,0,X_j,\ldots,X_n)$. Since $X_n=1$ almost surely
under the polymer measure, the identity (\ref{eq:identity_a1}) implies
\begin{eqnarray}
  \nonumber
  &&\Ex_{n,h,\omega}\Big[\Phi\mathds{1}_{\{T_{L_a+1}\le \lambda\}}(1-X_a)\Psi\Big]\\
  \nonumber
  &&\qquad=
  \sum_{i=0}^{a-1}\sum_{j=a+1}^n\mathds{1}_{\{j-i\le \lambda\}}\Ex_{n,h,\omega}\bigg[\Phi X_i\prod_{k=i+1}^{j-1}(1-X_k)X_j\Psi\bigg]\\
  \nonumber
  &&\qquad=\sum_{i=0}^{a-1}\sum_{j=a+1}^n\mathds{1}_{\{j-i\le \lambda\}}\Ex_{n,h,\omega}\bigg[\Phi_i X_i\prod_{k=i+1}^{j-1}(1-X_k)X_j\Psi_j\bigg]\,.
\end{eqnarray}
We can invoke Lemma \ref{lem:fact} to obtain
\begin{eqnarray}
  \nonumber
  &&\Ex_{n,h,\omega}\Big[\Phi\mathds{1}_{\{T_{L_a+1}\le \lambda\}}(1-X_a)\Psi\Big]\\
  \nonumber
  &&\qquad=\sum_{i=0}^{a-1}\sum_{j=a+1}^n\mathds{1}_{\{j-i\le \lambda\}}
  \Ex_{j,h,\omega}\big[\Phi_iX_i\big]\Ex_{j-i,h,\vartheta^i\omega}\bigg[\prod_{k=1}^{j-i-1}(1-X_k)\bigg]\Ex_{n,h,\omega}\big[X_j\Psi_j\big]\\
  &&\qquad=\sum_{i=0}^{a-1}\sum_{j=a+1}^n\mathds{1}_{\{j-i\le \lambda\}}
  \Ex_{j,h,\omega}\big[\Phi_iX_i\big]\,\frac{p(j-i)\ee^{h+\omega_j}}{Z_{j-i,h}(\vartheta^i\omega)}\,\Ex_{n,h,\omega}\big[X_j\Psi_j\big]\,.
  \label{eq:lower_one_1}
\end{eqnarray}
Similarly, (\ref{eq:identity_a2}) and Lemma \ref{lem:fact} show that
\begin{eqnarray}
  \nonumber
  &&\Ex_{n,h,\omega}\Big[\Phi \mathds{1}_{\{T_{L_a}+T_{L_a+1}\le \lambda\}}X_a\Psi\Big]\\
  \nonumber
  &&\qquad=\sum_{i=0}^{a-1}\sum_{j=a+1}^n\mathds{1}_{\{j-i\le \lambda\}}\Ex_{j,h,\omega}\big[\Phi_iX_i\big]\\
  \nonumber
  &&\qquad\qquad\qquad\qquad\quad\times\Ex_{j-i,h,\vartheta^i\omega}\bigg[\prod_{k=1}^{a-i-1}(1-X_k)X_{a-i}\prod_{k=a-i+1}^{j-i-1}(1-X_k)\bigg]\Ex_{n,h,\omega}\big[X_j\Psi_j\big]\\
  &&\qquad=\sum_{i=0}^{a-1}\sum_{j=a+1}^n\mathds{1}_{\{j-i\le \lambda\}}\Ex_{j,h,\omega}\big[\Phi_iX_i\big]\,\frac{p(a-i)p(j-a)\ee^{2h+\omega_a+\omega_j}}{Z_{j-i,h}(\vartheta^i\omega)}\,
  \Ex_{n,h,\omega}\big[X_j\Psi_j\big]\,.
   \label{eq:lower_one_2}
\end{eqnarray}
Now we observe that if $i<a<j$ with $j-i\le\lambda$, then the inequality
(\ref{eq:xi_def}) gives
\begin{eqnarray}
  \nonumber
  p(j-i)&\le &\xi \min\big\{(a-i)^\xi,(j-a)^\xi\big\}p(a-i)p(j-a)\\
  &\le &\xi \min\big\{\lambda^\xi,a^\xi,(n-a)^\xi\big\}p(a-i)p(j-a)\,.
  \label{eq:lower_one_3}
\end{eqnarray}
Therefore, plugging (\ref{eq:lower_one_3}) in (\ref{eq:lower_one_1}),
and then comparing with (\ref{eq:lower_one_2}), we get
\begin{align*}
  ~~~~~~~~~~~&\Ex_{n,h,\omega}\Big[\Phi\mathds{1}_{\{T_{L_a+1}\le \lambda\}}(1-X_a)\Psi\Big]\\
  &\qquad\le \xi\ee^{-h-\omega_a}\min\big\{\lambda^\xi,a^\xi,(n-a)^\xi\big\}\Ex_{n,h,\omega}\Big[\Phi \mathds{1}_{\{T_{L_a}+T_{L_a+1}\le \lambda\}}X_a\Psi\Big]\,.
~~~~~~~~~~~~~~~~\qedhere
\end{align*}
\end{proof}

\begin{remark}
   \label{remark:enforce_contact}
  {\rm The first part of Lemma \ref{lem:utile_per_tutto} is equivalent to 
\begin{eqnarray}
    \nonumber
  &&\Ex_{n,h,\omega}\Big[\phi(X_0,\ldots,X_{a-1})\psi(X_{a+1},\ldots,X_n)\Big]\\
  \nonumber
  &&\qquad\le 
  \Big[1+\xi\ee^{-h-\omega_a}\min\big\{a^\xi,(n-a)^\xi\big\}\Big]\Ex_{n,h,\omega}\Big[\phi(X_0,\ldots,X_{a-1})X_a\psi(X_{a+1},\ldots,X_n)\Big]\,.
\end{eqnarray}
This bound with $\phi$ and $\psi$ identically equal to 1 reads
\begin{equation*}
    \Ex_{n,h,\omega}[X_a]\ge \frac{1}{1+\xi\ee^{-h-\omega_a}\min\{a^\xi,(n-a)^\xi\}}\,.
\end{equation*}
The latter, which trivially holds even if $a=0$ or $a=n$, is a
refinement of \cite[Lemma A.1]{giacomin2006_1}.}
\end{remark}

\subsection{A concentration bound under subexponential conditions}
A version of the McDiarmid's inequality for functions of independent
random variables under subexponential conditions is provided in
\cite{maurer2021}. We use this result to obtain a concentration bound
for the finite-volume free energies $\log
Z_{n,h}=\log\Ex[\ee^{\sum_{a=1}^n(h+\omega_a)X_a}X_n]$ and $\log
Z_{n,h}^{\textbf{-}}=\log\Ex[\ee^{\sum_{a=1}^{n-1}(h+\omega_a)X_a}X_n]$.
Note that
\begin{equation}
  \Big|\log Z_{n,h}(\omega)-\log Z_{n,h'}(\omega')\Big|\le n|h-h'|+\sum_{a=1}^n|\omega_a-\omega_a'|
  \label{Lipschitz}
\end{equation}
for all $n\in\N$, $h$ and $h'$ in $\Rl$, and
$\omega:=\{\omega_a\}_{a\in\N_0}$ and
$\omega':=\{\omega_a'\}_{a\in\N_0}$ in $\Omega$, as one can easily
verify. The same holds if $Z_{n,h}$ is replaced by $
Z^{\textbf{-}}_{n,h}$. In particular, bound (\ref{Lipschitz}) gives
$|f(h)-f(h')|\le|h-h'|$.

\begin{theorem}
  \label{th:concentration}
  There exists a constant $\kappa>0$ such that for every $n\in\N$, $h\in\Rl$,
  and $u\ge 0$
 \begin{equation*}
\probd\bigg[\Big|\log Z_{n,h}-\Exd\big[\log Z_{n,h}\big]\Big|>u\bigg]\le 2\ee^{-\frac{\kappa u^2}{n+u}}\, ,
  \end{equation*}
 and the very same bound holds if $Z_{n,h}$ is replaced by $ Z^{\textbf{-}}_{n,h}$.
\end{theorem}

\begin{proof}[Proof of Theorem \ref{th:concentration}]
We develop the argument for $Z_{n,h}$, as the argument for
$Z^{\textbf{-}}_{n,h}$ is identical.  Fix $n\in\N$, $h\in\Rl$, and
$u\ge 0$. For $\varpi:=\{\varpi_b\}_{b\in\N_0}\in\Omega$ and
$a\in\{1,\ldots,n\}$, let $\Lambda_{\varpi,a}$ be the random variable on
$\Omega$ that maps $\omega:=\{\omega_b\}_{b\in\N_0}$ to
\begin{equation*}
\Lambda_{\varpi,a}(\omega):=\log Z_{n,h}(\{\varpi_0,\ldots,\varpi_{a-1},\omega_a,\varpi_{a+1},\ldots\})\,,
\end{equation*}
namely, $\Lambda_{\varpi,a}(\omega)$ is the finite-volume free energy
for a system that at site $b$ has the charge $\varpi_b$ if $b\ne a$
and $\omega_a$ if $b=a$.  An analog of the McDiarmid's inequality for
functions of independent random variables under subexponential
conditions (see \cite[Theorem 4]{maurer2021}) states that if there
exists a constant $c>0$ such that for every $\varpi\in\Omega$,
$a\in\{1,\ldots,n\}$, and $q\ge 1$ we have
  \begin{equation}
    \Exd\Big[\big|\Lambda_{\varpi,a}-\Exd[\Lambda_{\varpi,a}]\big|^q\Big]\le (cq)^q\,,
\label{th:concentration_1}
  \end{equation}
  then
  \begin{equation*}
    \probd\bigg[\Big|\log Z_{n,h}-\Exd\big[\log Z_{n,h}\big]\Big|>u\bigg]\le 2\ee^{-\frac{u^2}{(2\ee c)^2n+2\ee cu}}\,.
  \end{equation*}
Let us show that (\ref{th:concentration_1}) is satisfied with
$c:=(2/\eta\ee)\int_\Omega\ee^{\eta|\omega_0|}\probd[\dd\omega]$ and
$\eta$ as in Assumption \ref{assump:omega}. This proves the theorem
with $1/\kappa:=\max\{2ec,(2\ee c)^2\}$.
  
Bound (\ref{Lipschitz}) yields
$|\Lambda_{\varpi,a}(\omega)-\Lambda_{\varpi,a}(\omega')|\le
|\omega_a-\omega_a'|$ for $\omega:=\{\omega_b\}_{b\in\N_0}$ and
$\omega':=\{\omega_b'\}_{b\in\N_0}$ in $\Omega$, which thanks to
H\"older's inequality gives for $q\ge 1$
\begin{align}
  \nonumber
  \Big|\Lambda_{\varpi,a}(\omega)-\Exd[\Lambda_{\varpi,a}]\Big|^q &\le \int_\Omega\Big|\Lambda_{\varpi,a}(\omega)-\Lambda_{\varpi,a}(\omega')\Big|^q\,\probd[\dd\omega']\\
  \nonumber
  &\le\int_\Omega|\omega_a-\omega'_a|^q\,\probd[\dd\omega']\le 2^{q-1}|\omega_a|^q+2^{q-1}\int_\Omega|\omega'_0|^q\,\probd[\dd\omega']\,.
\end{align}
Using the inequality $\zeta^q\le (q/\eta \ee)^q\ee^{\eta\zeta}$ for
$\zeta\ge 0$ we finally get
\begin{equation*}
  \Exd\Big[\big|\Lambda_{\varpi,a}-\Exd[\Lambda_{\varpi,a}]\big|^q\Big]\le 2^q\int_\Omega|\omega_0|^q\,\probd[\dd\omega]
  \le \bigg(\frac{2q}{\eta\ee}\bigg)^q\int_\Omega\ee^{\eta|\omega_0|}\probd[\dd\omega]\le(cq)^q\,.
  \qedhere
\end{equation*}
  \end{proof}

\subsection{First consequences of concentration}

The first consequence of the concentration bounds for the free
energies is the lower bound on $\mu(h)$ in Proposition~\ref{prop:mu},
which is proved in this section.  The superadditivity property
exploited in \cite{cf:AZ96,giacomin2006_1} to establish the existence
of $\mu(h)$ is no longer available because now we work with
$Z^{\textbf{-}}_{n,h}$. However, one can show that the sequence
$\{\log\Exd[1/Z^{\textbf{-}}_{n,h}]\}_{n\in\N_0}$ is nearly
superadditive and nearly subadditive: we choose to show the second
property since it follows directly from Lemma
\ref{lem:utile_per_tutto}.

\begin{proof}[Proof of Proposition~\ref{prop:mu}]
Let us show that the limit defining $\mu(h)$ exists for all $h\in\Rl$.
We remark that $q:=\int_\Omega\mathds{1}_{\{\omega_0\le
  0\}}\probd[\dd\omega]>0$ because
$\int_\Omega\omega_0\probd[\dd\omega]=0$. Fix $h\in\Rl$ and consider
the sequence whose $n^{\mathrm{th}}$ term is $\log [2^\xi(\ee^h+\xi
  n^\xi)/q]-\log\Exd[{1}/{Z_{n,h}^{\textbf{-}}}]$, where the number
$\xi$ has been defined in (\ref{eq:xi_def}). We claim that such
sequence is subadditive, so that
$\lim_{n\uparrow\infty}-({1}/{n})\log\Exd\big[{1}/{Z_{n,h}^{\textbf{-}}}\big]=:\mu(h)$
exists as a consequence. To prove this, we appeal to Lemma
\ref{lem:utile_per_tutto}, through Remark
\ref{remark:enforce_contact}, to state that for all $\omega\in\Omega$
and $m,n\in\N$
\begin{equation*}
  \Ex_{m+n,h,\omega}[X_m]\ge\frac{1}{1+\xi\ee^{-h-\omega_m}\min\{m^\xi,n^\xi\}}\,.
\end{equation*}
Then we observe that
\begin{align}
  \nonumber
  \Ex_{m+n,h,\omega}[X_m]=\frac{\Ex[X_m\ee^{\sum_{a=1}^{m+n}(h+\omega_a)X_a}X_{m+n}]}{Z_{m+n,h}(\omega)}
  &=\frac{Z_{m,h}(\omega)Z_{n,h}(\vartheta^m\omega)}{Z_{m+n,h}(\omega)}\\
  \nonumber
  &=\ee^{h+\omega_m}\frac{Z^{\textbf{-}}_{m,h}(\omega)Z^{\textbf{-}}_{n,h}(\vartheta^m\omega)}{Z^{\textbf{-}}_{m+n,h}(\omega)}\,,
\end{align}
where we have used the renewal property in the second equality.
Therefore
\begin{equation*}
  \ee^{h+\omega_m}\frac{Z^{\textbf{-}}_{m,h}(\omega)Z^{\textbf{-}}_{n,h}(\vartheta^m\omega)}{Z^{\textbf{-}}_{m+n,h}(\omega)}\ge\frac{1}{1+\xi\ee^{-h-\omega_m}\min\{m^\xi,n^\xi\}}\,,
\end{equation*}
namely
\begin{equation*}
  \frac{1}{Z^{\textbf{-}}_{m+n,h}(\omega)}\ge \frac{1}{\ee^{h+\omega_m}+\xi\min\{m^\xi,n^\xi\}}\,
  \frac{1}{Z^{\textbf{-}}_{m,h}(\omega)}\frac{1}{Z^{\textbf{-}}_{n,h}(\vartheta^m\omega)}\,.
\end{equation*}
At this point, to obtain the desired subadditivity property it
suffices to note that the three factors in the right-hand side of the
last inequality are statistically independent, so by integrating with
respect to $\probd[\dd\omega]$ and exploiting that
\begin{align}
  \nonumber
\int_\Omega\frac{1}{\ee^{h+\omega_0}+\xi\min\{m^\xi,n^\xi\}}\,\probd[\dd\omega]&\ge 
\int_\Omega\frac{\mathds{1}_{\{\omega_0\le0\}}}{\ee^h+\xi\min\{m^\xi,n^\xi\}}\,\probd[\dd\omega]\\
\nonumber
&= q\,\frac{1}{\ee^h+\xi \min\{m^\xi,n^\xi\}}\ge  \frac{q}{2^\xi}\frac{\ee^h+\xi(m+n)^\xi}{(\ee^h+\xi m^\xi)(\ee^h+\xi n^\xi)}
\end{align}
we complete the task. 

\smallskip

\noindent\textit{The Lipschitz property of $\mu$.}  The function $\mu$ that
maps $h\in\Rl$ to $\mu(h)$ is Lipschitz continuous with Lipschitz
constant equal to $1$. In fact, since $0\le \partial_h \log
Z_{n,h}^{\textbf{-}}(\omega)\le n$, for every $n\in\N$ and $h\in\Rl$
we have
\begin{equation*}
  \Bigg|\partial_h  \frac{1}{n}\log \Exd\bigg[\frac{1}{Z_{n,h}^{\textbf{-}}}\bigg]\Bigg|
  =\left|\frac{1}{n}\frac{\Exd\Big[\frac{\partial_h\log Z_{n,h}^{\textbf{-}}}{Z_{n,h}^{\textbf{-}}}\Big]}{\Exd\Big[\frac{1}{Z_{n,h}^{\textbf{-}}}\Big]}\right|\le 1\,.
\end{equation*}
From here we deduce that for all $n\in\N$ and $h,h'\in\Rl$
\begin{equation*}
\Bigg|\frac{1}{n}\log \Exd\bigg[\frac{1}{Z_{n,h}^{\textbf{-}}}\bigg]-\frac{1}{n}\log \Exd\bigg[\frac{1}{Z_{n,h'}^{\textbf{-}}}\bigg]\Bigg|\le |h-h'|\,,
\end{equation*}
which in turn gives $|\mu(h)-\mu(h')|\le|h-h'|$ by sending $n$ to
infinity.

\smallskip

\noindent\textit{The bounds on $\mu(h)$.}  The inequality $\mu(h) \le f(h)$
follows from Jensen's inequality, which gives for all $n\in\N$ and
$h\in\Rl$
\begin{equation*}
  -\log \Exd\bigg[\frac{1}{Z_{n,h}^{\textbf{-}}}\bigg]\le\Exd\big[\log Z_{n,h}^{\textbf{-}}\big]=\Exd\big[\log Z_{n,h}\big]-h\,.
\end{equation*}
For a lower bound on $\mu(h)$ note that
$p(n)/Z^{\textbf{-}}_{n,h}(\omega)\le 1$, because the left-hand side
is the probability
$\prob_{n,h,\omega}[T_1=n]=\ee^{h+\omega_n}p(n)/Z_{n,h}(\omega)$, so for
every $n\in\N$, $h\in\Rl$, and $\omega\in\Omega$ we have
\begin{equation}
  \frac{p(n)}{Z_{n,h}^{\textbf{-}}(\omega)}\le \mathds{1}_{\big\{\log Z^{\textbf{-}}_{n,h}( \omega)
    - \Exd[ \log Z^{\textbf{-}}_{n,h}] < -\frac{1}{2}f(h)n\big\}}+p(n)\,\ee^{\frac{1}{2} f(h)n-\Exd[ \log Z^{\textbf{-}}_{n,h}]}\,.
  \label{eq:lower_bound_mu}
\end{equation}
At the same time, using that $\Exd[\log Z^{\textbf{-}}_{n,h}] =
\Exd[\log Z_{n,h}]-h$, we can state that $\Exd[\log
  Z^{\textbf{-}}_{n,h}]\ge f(h)n-\epsilon n$ for any given
$\epsilon>0$ and all sufficiently large $n$. Thus, integrating
(\ref{eq:lower_bound_mu}) with respect to $\probd[\dd\omega]$ and
invoking Theorem \ref{th:concentration}, we find for all sufficiently
large $n$
\begin{align}
  \nonumber
  p(n)\,\Exd\bigg[ \frac 1 {Z^{\textbf{-}}_{n,h}}\bigg] &\le 2 \ee^{- \frac{\kappa f(h)^2}{4+2f(h)}n}+p(n)\,\ee^{-\frac{1}{2}f(h)n+\epsilon n}\\
  \nonumber
  &\le 2 \ee^{-\frac{\kappa}{6}\min\{f(h),f(h)^2\}n}+p(n)\,\ee^{-\frac{1}{2}f(h)n+\epsilon n}\,.
\end{align}
Taking the logarithm, dividing by $n$, and then sending $n$ to
infinity, this bound allows us to conclude that $\mu(h)\ge
c\min\{f(h),f(h)^2\}-\epsilon$ with $c:=\min\{\kappa/6,1/2\}>0$, which
implies $\mu(h)\ge c\min\{f(h),f(h)^2\}$ because $\epsilon$ is
arbitrary.
\end{proof}

In the proof of Proposition~\ref{prop:mu} we have observed that
$\prob_{n,h,\omega}[T_1=n]=p(n)/Z_{n,h}^{\textbf{-}}(\omega)$. As a
direct consequence, Proposition~\ref{prop:mu} has the following
corollary.

\begin{corollary}
  \label{cor:PT1}
For every $h\in\Rl$
\begin{equation*}
    \lim_{n \uparrow \infty}  \frac{1}{n} \log \Exd\Big[\prob_{n,h,\cdot}[T_1=n]\Big]=-\mu(h)\,.
  \end{equation*}
\end{corollary}

Corollary \ref{cor:PT1} and the fact that $\mu(h)\ge
c\min\{f(h),f(h)^2\}>0$ for $h>h_c$ imply that the contact fraction in
the localized phase is small with exponentially small probability. We
let the following lemma to state a precise result.

\begin{lemma}
  \label{lem:density}
  For every closed set $H\subset(h_c,+\infty)$ there exist constants
  $\delta>0$ and $\gamma>0$ such that for all $n\in\N$
  \begin{equation*}
    \Exd\bigg[\sup_{h\in H}\prob_{n,h,\cdot}\big[L_n<\delta n\big]\bigg]\le \ee^{-\gamma n}\,.
    \end{equation*}
\end{lemma}

Measurability of the variable of which we are taking the expectation
in Lemma~\ref{lem:density} holds because
$\prob_{n,h,\omega}[L_n<\delta n]$ is continuous with respect to $h$
for any $n$ and $\omega$, so that $\sup_{h\in
  H}\prob_{n,h,\cdot}[L_n<\delta n]=\sup_{h\in
  D}\prob_{n,h,\cdot}[L_n<\delta n]$, $D$ being a countable dense
subset of $H$.  The same argument applies to also the variables of
which we are taking the expectation in Lemmas \ref{lem:decay},
\ref{lem:decay_aux}, and \ref{lem:mixing}.

\begin{proof}[Proof of Lemma \ref{lem:density}]
Fix a closed set $H\subset(h_c,+\infty)$ and define $h_o:=\inf
H>h_c$. Since $\mu(h_o)>0$ by Proposition~\ref{prop:mu}, Corollary
\ref{cor:PT1} shows that there exist constants $\gamma_o>0$ and
$G_o>0$ such that
\begin{equation}
  \Exd\Big[\prob_{n,h_o,\cdot}[T_1=n]\Big]\le G_o\ee^{-\gamma_o n}
  \label{cor:density_0}
\end{equation}
for all $n\in\N$. Put $\gamma:=\gamma_o/3$ and
$\delta:=\gamma/\log(1+G_o/\gamma)$.  We claim that the lemma
holds with such $\delta$ and $\gamma$. In fact, for every
$n\in\N$, $h\in\Rl$, and $\omega\in\Omega$ we have
\begin{align}
  \nonumber
  &\prob_{n,h,\omega}\big[L_n<\delta n\big]\\
  \nonumber
  &\qquad=\sum_{l\in\N}\mathds{1}_{\{l<\delta n\}}\sum_{0=a_0<a_1<\cdots <a_l=n}
    \Ex_{n,h,\omega}\bigg[\prod_{i=1}^lX_{a_{i-1}}\prod_{k=a_{i-1}+1}^{a_i-1}(1-X_k)X_{a_i}\bigg]\,.
  \end{align}
On the other hand, repeated applications of Lemma \ref{lem:fact} show
that for all integers $0=a_0<a_1<\cdots <a_l=n$
\begin{align}
  \nonumber
  &\Ex_{n,h,\omega}\bigg[\prod_{i=1}^lX_{a_{i-1}}\prod_{k=a_{i-1}+1}^{a_i-1}(1-X_k)X_{a_i}\bigg]\\
  \nonumber
  &\qquad=\prod_{i=1}^l\Ex_{a_i-a_{i-1},h,\vartheta^{a_{i-1}}\omega}\bigg[\prod_{k=1}^{a_i-a_{i-1}-1}(1-X_k)\bigg]\Ex_{a_i,h,\omega}[X_{a_{i-1}}]\\
  \nonumber
  &\qquad\le\prod_{i=1}^l\Ex_{a_i-a_{i-1},h,\vartheta^{a_{i-1}}\omega}\bigg[\prod_{k=1}^{a_i-a_{i-1}-1}(1-X_k)\bigg]\\
  \nonumber
  &\qquad=\prod_{i=1}^l\prob_{a_i-a_{i-1},h,\vartheta^{a_{i-1}}\omega}\big[T_1=a_i-a_{i-1}\big]\,.
\end{align}
Since $\prob_{m,h,\omega}[T_1=m]=p(m)/Z_{m,h}^{\textbf{-}}(\omega)$
with
$Z_{m,h}^{\textbf{-}}(\omega)=\Ex[\ee^{\sum_{a=1}^{m-1}(h+\omega_a)X_a}X_m]$
is decreasing with respect to $h$ for any $m$ and $\omega$, we deduce
that
\begin{align}
  \nonumber
  &\sup_{h\in H}\prob_{n,h,\omega}\big[L_n<\delta n\big]\\
    \nonumber
    &\qquad\le\sum_{l\in\N}\mathds{1}_{\{l<\delta n\}}\sum_{0=a_0<a_1<\cdots <a_l=n}
    \prod_{i=1}^l\prob_{a_i-a_{i-1},h_o,\vartheta^{a_{i-1}}\omega}\big[T_1=a_i-a_{i-1}\big]\,.
  \end{align}
In this way, integrating with respect to $\probd[\dd\omega]$, the statistical
independence of the factors
$\prob_{a_i-a_{i-1},h_o,\vartheta^{a_{i-1}}\cdot}[T_1=a_i-a_{i-1}]$
combined with bound (\ref{cor:density_0}) leads us to the inequality
  \begin{align}
    \nonumber
    \Exd\bigg[\sup_{h\in H}\prob_{n,h,\omega}\big[L_n<\delta n\big]\bigg]&\le \sum_{l\in\N}\mathds{1}_{\{l<\delta n\}}{n-1 \choose l-1}G_o^l\ee^{-\gamma_o n}\,.
  \end{align}
As $(1+G_o/\gamma)^\delta=\ee^\gamma$ by construction, we can
finally appeal to the bound $\mathds{1}_{\{l<\delta n\}}\le
(1+G_o/\gamma)^{\delta n-l}\le(\gamma/G_o)^l\ee^{\gamma n}$ and to the
relationship $\gamma_o=3\gamma$ to get
  \begin{eqnarray}
    \nonumber
    ~~~~~~~~~~~~~~~~~~~~~~~~\Exd\bigg[\sup_{h\in H}\prob_{n,h,\omega}\big[L_n<\delta n\big]\bigg]&\le & \sum_{l=1}^n (\gamma/G_o)^l\ee^{\gamma n}{n-1 \choose l-1}G_o^l\ee^{-3\gamma n}\\
      \nonumber
      &=&\gamma(1+\gamma)^{n-1}\ee^{-2\gamma n}\le\ee^{-\gamma n}\,.~~~~~~~~~~~~~~~~~~~~~~~~~~~~\qedhere
  \end{eqnarray}
\end{proof}

\section{Correlation estimates}
\label{sec:corr}

This section introduces the correlation estimates that we use for the
main proofs of the paper. The first lemma we present is concerned with
the covariance of two observables under the polymer measure
$\prob_{n,h,\omega}$: given two measurable complex functions $\Phi$
and $\Psi$ over $(\mathcal{S},\mathfrak{S})$ we define their
covariance as
\begin{equation*}
  \cov_{n,h,\omega}[\Phi,\Psi]:=\Ex_{n,h,\omega}\big[\Phi\overline{\Psi}\,\big]-\Ex_{n,h,\omega}\big[\Phi\big]\Ex_{n,h,\omega}\big[\,\overline{\Psi}\,\big]\,,
\end{equation*}
provided that all the expectations exist. As usual, $\bar{\zeta}$ is
the complex conjugate of $\zeta\in\Cm$.  The variance
$\cov_{n,h,\omega}[\Phi,\Phi]$ is denoted by
$\var_{n,h,\omega}[\Phi]$.  The lemma involves two independent copies
of the polymer, so let $\{S_i\}_{i\in\N_0}$ and
$\{S_i'\}_{i\in\N_0}$ be two independent copies of the renewal
process, defined on the product probability space
$(\mathcal{S}^2,\mathfrak{S}^{\otimes 2},\prob^{\otimes 2})$. For
$n\in\N_0$, $h\in\Rl$, and $\omega\in\Omega$, consider on
$(\mathcal{S}^2,\mathfrak{S}^{\otimes 2})$ also the product measure
$\prob_{n,h,\omega}^{\otimes 2}$ and denote by
$\Ex_{n,h,\omega}^{\otimes 2}$ the corresponding expectation. Finally,
let $\xi$ be the number introduced in \eqref{eq:xi_def}.

\begin{lemma}
  \label{lem:corr}
  For every $h\in\Rl$, $\omega:=\{\omega_c\}_{c\in\N_0}\in\Omega$,
  integers $0\le a\le b\le n$, and functions $\phi\in\mathcal{M}_a$
  and $\psi\in\mathcal{M}_{n-b}$
\begin{equation*}
\bigg|\cov_{n,h,\omega}\Big[\phi(X_0,\ldots,X_a),\psi(X_b,\ldots,X_n)\Big]\bigg|
\le 2\sum_{i=0}^{a-1}\sum_{j=b+1}^\infty\Ex_{j-i,h,\vartheta^i\omega}^{\otimes 2}\bigg[\prod_{k=1}^{j-i-1}(1-X_kX_k')\bigg]\,.
\end{equation*}
Moreover
\begin{align}
  \nonumber
  &\frac{|\cov_{n,h,\omega}[X_a,X_b]|}{\min\{\Ex_{n,h,\omega}[X_a],\Ex_{n,h,\omega}[X_b]\}}\\
  \nonumber
  &\qquad\le
  \sum_{i=0}^a\sum_{j=b}^\infty 2\xi (j-i)^\xi\ee^{-h-\min\{\omega_a,\omega_b\}}\Ex_{j-i,h,\vartheta^i\omega}^{\otimes 2}\bigg[\prod_{k=1}^{j-i-1}(1-X_kX_k')\bigg]\,.
\end{align}
\end{lemma}

\begin{proof}[Proof of Lemma \ref{lem:corr}]
  Fix $h\in\Rl$, $\omega:=\{\omega_c\}_{c\in\N_0}\in\Omega$, and
  integers $0\le a\le b\le n$.  To begin with, we prove
  that for every functions $\phi\in\mathcal{M}_a$ and
  $\psi\in\mathcal{M}_{n-b}$
   \begin{equation}
    \cov_{n,h,\omega}[\Phi,\Psi]=\Ex_{n,h,\omega}^{\otimes 2}\big[\Phi\overline{\Psi}-\Phi\overline{\Psi'}\,\big]
    =\Ex_{n,h,\omega}^{\otimes 2}\bigg[\Phi\big(\,\overline{\Psi}-\overline{\Psi'}\,\big)\prod_{k=a}^b(1-X_kX_k')\bigg]\,,
  \label{lem:corr_0}
  \end{equation}
where we have set $\Phi:=\phi(X_0,\ldots,X_a)$,
$\Psi:=\psi(X_b,\ldots,X_n)$, and $\Psi':=\psi(X_b',\ldots,X_n')$ for
brevity.  The first equality is immediate. Regarding the second
equality, we use the identity
\begin{equation*}
  \prod_{k=a}^b(1-X_kX_k')-\sum_{\substack{\mathcal{A}\subseteq\{a,\ldots,b\}\\\mathcal{A}\ne\emptyset}}(-1)^{|\mathcal{A}|}\prod_{k\in\mathcal{A}}X_kX_k'=1
\end{equation*}
to write down
\begin{align}
  \nonumber
  \Ex_{n,h,\omega}[\Phi\overline{\Psi}-\Phi\overline{\Psi'}]&=\Ex_{n,h,\omega}^{\otimes 2}\bigg[\Phi(\overline{\Psi}-\overline{\Psi'})\prod_{k=a}^b(1-X_kX_k')\bigg]\\
  \nonumber
  &\quad-\sum_{\substack{\mathcal{A}\subseteq\{a,\ldots,b\}\\\mathcal{A}\ne\emptyset}}(-1)^{|\mathcal{A}|}\Ex_{n,h,\omega}^{\otimes 2}\bigg[\Phi(\overline{\Psi}-\overline{\Psi'})\prod_{k\in\mathcal{A}}X_kX_k'\bigg]\,.
\end{align}
The desired result is due to the fact that $\Ex_{n,h,\omega}^{\otimes
  2}[\Phi(\overline{\Psi}-\overline{\Psi'})\prod_{k\in\mathcal{A}}X_kX_k']=0$ for any given
non-empty set $\mathcal{A}\subseteq\{a,\ldots,b\}$. To prove this fact we
invoke Lemma \ref{lem:fact} to state that
\begin{equation*}
  \Ex_{n,h,\omega}\bigg[\Phi\prod_{k\in\mathcal{A}}X_k\overline{\Psi}\bigg]
  =\frac{\Ex_{n,h,\omega}\big[\Phi\prod_{k\in\mathcal{A}}X_k\big]\Ex_{n,h,\omega}\big[\prod_{k\in\mathcal{A}}X_k\overline{\Psi}\,\big]}
      {\Ex_{n,h,\omega}\big[\prod_{k\in\mathcal{A}}X_k\big]}\,,
\end{equation*}
which entails
\begin{align}
  \nonumber
  &\Ex_{n,h,\omega}^{\otimes 2}\bigg[\Phi(\overline{\Psi}-\overline{\Psi'})\prod_{k\in\mathcal{A}}X_kX_k'\bigg]\\
  \nonumber
  &\qquad=\Ex_{n,h,\omega}\bigg[\Phi\prod_{k\in\mathcal{A}}X_k\overline{\Psi}\bigg]\Ex_{n,h,\omega}\bigg[\prod_{k\in\mathcal{A}}X_k\bigg]
  -\Ex_{n,h,\omega}\bigg[\Phi\prod_{k\in\mathcal{A}}X_k\bigg]\Ex_{n,h,\omega}\bigg[\prod_{k\in\mathcal{A}}X_k\overline{\Psi}\bigg]=0\,.
\end{align}

\smallskip

\noindent\textit{The bound on
  $\cov_{n,h,\omega}[\phi(X_0,\ldots,X_a),\psi(X_b,\ldots,X_n)]$.}
Given functions $\phi\in\mathcal{M}_a$ and $\psi\in\mathcal{M}_{n-b}$,
formula (\ref{lem:corr_0}) shows that
\begin{equation*}
\Big|\cov_{n,h,\omega}\big[\Phi,\Psi\big]\Big|\le 2\Ex_{n,h,\omega}^{\otimes 2}\bigg[\prod_{k=a}^b(1-X_kX_k')\bigg]\, .
\end{equation*}
The identity
\begin{align}
  \nonumber
  \prod_{k=a}^b(1-X_kX_k')&=\sum_{i=0}^{a-1}\sum_{j=b+1}^nX_iX_i'\prod_{k=i+1}^{j-1}(1-X_kX_k')X_jX_j'+\prod_{k=a}^n(1-X_kX_k')\,,
\end{align}
where we have expanded the left-hand side according to the last
contacts before $a$ and the first contacts (if any) after $b$ that not
exceed $n$, together with the fact that $X_n=1$ and $X_n'=1$ almost
surely under the two-polymers measure $\prob_{n,h,\omega}^{\otimes
  2}$, yields
\begin{equation}
  \Big|\cov_{n,h,\omega}\big[\Phi,\Psi\big]\Big|
  \le 2\sum_{i=0}^{a-1}\sum_{j=b+1}^n\Ex_{n,h,\omega}^{\otimes 2}\bigg[X_iX_i'\prod_{k=i+1}^{j-1}(1-X_kX_k')X_jX_j'\bigg]\,.
  \label{eq:corr_1000} 
\end{equation}
Finally, we appeal once again to Lemma \ref{lem:fact} to get the bound
\begin{eqnarray}
  \nonumber
  &&\Ex_{n,h,\omega}^{\otimes 2}\bigg[X_iX_i'\prod_{k=i+1}^{j-1}(1-X_kX_k')X_jX_j'\bigg]\\
  \nonumber
  &&\qquad=\sum_{x_i\in\{0,1\}}\cdots\sum_{x_j\in\{0,1\}}\Ex_{n,h,\omega}\bigg[X_i\prod_{k=i+1}^{j-1}(1-x_kX_k)X_j\bigg]
  \Ex_{n,h,\omega}\bigg[X_iX_j\prod_{k=i}^j\mathds{1}_{\{X_k=x_k\}}\bigg]\\
  \nonumber
  &&\qquad\le\sum_{x_i\in\{0,1\}}\cdots\sum_{x_j\in\{0,1\}}\Ex_{j-i,h,\vartheta^i\omega}\bigg[\prod_{k=1}^{j-i-1}(1-x_{k+i}X_k)\bigg]
  \Ex_{j-i,h,\vartheta^i\omega}\bigg[\prod_{k=0}^{j-i}\mathds{1}_{\{X_k=x_{k+i}\}}\bigg]\\
  &&\qquad=\Ex_{j-i,h,\vartheta^i\omega}^{\otimes 2}\bigg[\prod_{k=1}^{j-i-1}(1-X_kX_k')\bigg]\,.
  \label{eq:corr_1001}
\end{eqnarray}
Inequalities \eqref{eq:corr_1000} and \eqref{eq:corr_1001} give the
bound on
$\cov_{n,h,\omega}[\phi(X_0,\ldots,X_a),\psi(X_b,\ldots,X_n)]$ stated
by the lemma.

\smallskip

\noindent\textit{The bounds on $\cov_{n,h,\omega}[X_a,X_b]$.}  Formula
(\ref{lem:corr_0}) with the choice $\phi(x_0,\ldots,x_a):=x_a$ and
$\psi(x_b,\ldots,x_n):=x_b$ shows that both
 \begin{equation*}
   \Big|\cov_{n,h,\omega}[X_a,X_b]\Big|\le 2\Ex_{n,h,\omega}^{\otimes 2}\bigg[X_a\prod_{k=a}^b(1-X_kX_k')\bigg]
 \end{equation*}
 and
  \begin{equation*}
   \Big|\cov_{n,h,\omega}[X_a,X_b]\Big|\le 2\Ex_{n,h,\omega}^{\otimes 2}\bigg[\prod_{k=a}^b(1-X_kX_k')X_b\bigg]\,.
  \end{equation*}
  Then, the identities
\begin{equation*}
  \prod_{k=a}^b(1-X_kX_k')=\sum_{j=b+1}^n\prod_{k=a}^{j-1}(1-X_kX_k')X_jX_j'+\prod_{k=a}^n(1-X_nX_n')
  \end{equation*}
and
\begin{equation*}
  \prod_{k=a}^b(1-X_kX_k')=\sum_{i=0}^{a-1}X_iX_i'\prod_{k=i+1}^b(1-X_kX_k')\,,
  \end{equation*}
together with the fact that $X_n=1$ and $X_n'=1$ almost surely under
the two-polymers measure, give
 \begin{equation}
   \Big|\cov_{n,h,\omega}[X_a,X_b]\Big|\le 2\sum_{j=b+1}^n\Ex_{n,h,\omega}^{\otimes 2}\bigg[X_a\prod_{k=a}^{j-1}(1-X_kX_k')X_jX_j'\bigg]
   \label{covab_bound_a}
 \end{equation}
and
\begin{equation}
  \Big|\cov_{n,h,\omega}[X_a,X_b]\Big|\le 2\sum_{i=0}^{a-1}\Ex_{n,h,\omega}^{\otimes 2}\bigg[X_iX_i'\prod_{k=i+1}^b(1-X_kX_k')X_b\bigg]\,.
  \label{covab_bound_b}
\end{equation}
We shall verify that (\ref{covab_bound_a}) implies the inequality
\begin{align}
  \nonumber
  &\big|\cov_{n,h,\omega}[X_a,X_b]\big|\\
  \nonumber
  &\qquad\le 
  \Ex_{n,h,\omega}[X_a]\sum_{j=b+1}^n 2\xi (j-a)^\xi\ee^{-h-\omega_a}\Ex_{j-a,h,\vartheta^a\omega}^{\otimes 2}\bigg[\prod_{k=1}^{j-a-1}(1-X_kX_k')\bigg]\,,
\end{align}
whereas (\ref{covab_bound_b}) provides
\begin{align}
  \nonumber
  &\big|\cov_{n,h,\omega}[X_a,X_b]\big|\\
  \nonumber
  &\qquad\le 
  \Ex_{n,h,\omega}[X_b]
  \sum_{i=0}^{a-1} 2\xi (b-i)^\xi\ee^{-h-\omega_b}\Ex_{b-i,h,\vartheta^i\omega}^{\otimes 2}\bigg[\prod_{k=1}^{b-i-1}(1-X_kX_k')\bigg]\,.
\end{align}
The latter inequalities demonstrate the loose bound on
$\cov_{n,h,\omega}[X_a,X_b]$ stated by the lemma.

\smallskip

We develop the argument for \eqref{covab_bound_a}, as the argument
for \eqref{covab_bound_b} is identical. We write the terms in the
right-hand side of \eqref{covab_bound_a} as
\begin{align}
   \nonumber
   &\Ex_{n,h,\omega}^{\otimes 2}\bigg[X_a\prod_{k=a}^{j-1}(1-X_kX_k')X_jX_j'\bigg]\\
   \nonumber
   &\qquad=\Ex_{n,h,\omega}^{\otimes 2}\bigg[X_a(1-X_a')\prod_{k=a+1}^{j-1}(1-X_kX_k')X_jX_j'\bigg]\\
   \nonumber
   &\qquad =\sum_{x_{a+1}\in\{0,1\}}\cdots\sum_{x_j\in\{0,1\}}
   \Ex_{n,h,\omega}\bigg[X_a\prod_{k=a+1}^{j-1}(1-x_kX_k)X_j\bigg]\\
   \nonumber
   &\qquad\qquad\qquad\qquad\qquad\qquad\quad\times\Ex_{n,h,\omega}\bigg[(1-X_a)\prod_{k=a+1}^j\mathds{1}_{\{X_k=x_k\}}X_j\bigg]\,.
\end{align}
Repeated applications of Lemma \ref{lem:fact} give
\begin{equation*}
   \Ex_{n,h,\omega}\bigg[X_a\prod_{k=a+1}^{j-1}(1-x_kX_k)X_j\bigg]\le\Ex_{n,h,\omega}[X_a]\,\Ex_{j-a,h,\vartheta^a\omega}\bigg[\prod_{k=1}^{j-a-1}(1-x_{k+a}X_k)\bigg]
\end{equation*}
and
\begin{equation}
  \Ex_{n,h,\omega}\bigg[(1-X_a)\prod_{k=a+1}^j\mathds{1}_{\{X_k=x_k\}}X_j\bigg]\le \Ex_{j,h,\omega}\bigg[(1-X_a)\prod_{k=a+1}^j\mathds{1}_{\{X_k=x_k\}}\bigg]\,.
  \label{eq:covab_bound_a_1}
\end{equation}
At this point, we invoke Lemma \ref{lem:utile_per_tutto}, through
Remark \ref{remark:enforce_contact}, to enforce in the right-hand side
of (\ref{eq:covab_bound_a_1}) the constraint that $a$ is a contact
site, which in turn allows us a further application of Lemma
\ref{lem:fact}:
\begin{align}
  \nonumber
  \Ex_{j,h,\omega}\bigg[(1-X_a)\prod_{k=a+1}^j\mathds{1}_{\{X_k=x_k\}}\bigg]&\le \xi(j-a)^\xi\ee^{-h-\omega_a}\Ex_{j,h,\omega}\bigg[X_a\prod_{k=a+1}^j\mathds{1}_{\{X_k=x_k\}}\bigg] \\
  \nonumber
  &\le\xi(j-a)^\xi\ee^{-h-\omega_a}\Ex_{j-a,h,\vartheta^a\omega}\bigg[\prod_{k=1}^{j-a}\mathds{1}_{\{X_k=x_{k+a}\}}\bigg]\,.
\end{align}
Putting the pieces together, for the terms in the right-hand side of
(\ref{covab_bound_a}) we find the estimate
\begin{align}
  \nonumber
  ~~~~~~~~~~~~~~~~&\Ex_{n,h,\omega}^{\otimes 2}\bigg[X_a\prod_{k=a}^{j-1}(1-X_kX_k')X_jX_j'\bigg]\\
  \nonumber
  &\qquad\le \Ex_{n,h,\omega}[X_a]\,\xi(j-a)^\xi\ee^{-h-\omega_a}\Ex_{j-a,h,\vartheta^a\omega}^{\otimes 2}\bigg[\prod_{k=1}^{j-a-1}(1-X_kX_k')\bigg]\,.
  ~~~~~~~~~~~~~~~~~\qedhere
\end{align}
\end{proof}

In the light of Lemma \ref{lem:corr}, the next lemma entails that
correlations decay exponentially fast in the localized phase.

\begin{lemma}
  \label{lem:decay}
  For every closed set $H\subset(h_c,+\infty)$ there exist constants $\gamma>0$ and $G>0$ such
  that for all $n\in\N$
\begin{equation*}
  \Exd\Bigg[\sup_{h\in H}\Ex_{n,h,\cdot}^{\otimes 2}\bigg[\prod_{k=1}^{n-1}(1-X_kX_k')\bigg]\Bigg]\le G\,\ee^{-\gamma n}\,.
\end{equation*}
\end{lemma}

The proof of Lemma \ref{lem:decay} follows the proof of \cite[Lemma
  3.1]{giacomin2006_1} and relies on the following technical result,
which states the possibility to cover most of the sites with
inter-arrival times of finite length.

\begin{lemma}
  \label{lem:decay_aux}
  For every $H\subset(h_c,+\infty)$ closed and $\varkappa>0$ there
  exist constants $\tau\in\N$ and $\gamma>0$ such that for all
  $n\in\N$
  \begin{equation*}
    \Exd\Bigg[\sup_{h\in H}\prob_{n,h,\cdot}\bigg[\sum_{i=1}^{L_n}T_i\mathds{1}_{\{T_i>\tau\}}\ge\varkappa n\bigg]\Bigg]\le \ee^{-\gamma n}\,.
    \end{equation*}
\end{lemma}

\begin{proof}[Proof of Lemma \ref{lem:decay_aux}] 
Fix $H\subset(h_c,+\infty)$ closed and $\varkappa>0$. Define
$h_o:=\inf H>h_c$. According to Corollary \ref{cor:PT1}, there exist
constants $\gamma_o>0$ and $G_o>0$ such that
\begin{equation}
  \Exd\Big[\prob_{n,h_o,\cdot}[T_1=n]\Big]\le G_o\ee^{-\gamma_o n}
  \label{lem:decay_aux_0}
\end{equation}
for all $n\in\N$. Put $\gamma:=\varkappa\gamma_o/4$ and take
$\tau\in\N$ so large that $\tau\ge
(1/\gamma)\log(1+G_o\ee^\gamma/\gamma^2)$.  We claim that the lemma
holds with such $\tau$ and $\gamma$. In fact, since
$\tau\sum_{i=1}^{L_n}\mathds{1}_{\{T_i>\tau\}}<\sum_{i=1}^{L_n}T_i\mathds{1}_{\{T_i>\tau\}}\le
S_{L_n}\le n$, we can write for all $n\in\N$, $h\in\Rl$, and
$\omega\in\Omega$
\begin{align}
    \nonumber
    &\prob_{n,h,\omega}\bigg[\sum_{i=1}^{L_n}T_i\mathds{1}_{\{T_i>\tau\}}\ge\varkappa n\bigg]\\
    \nonumber
    &\qquad =\sum_{r\in\N}\mathds{1}_{\{r<\frac{n}{\tau}\}}\prob_{n,h,\omega}\bigg[\sum_{i=1}^{L_n}T_i\mathds{1}_{\{T_i>\tau\}}\ge\varkappa n,\,
      \sum_{i=1}^{L_n}\mathds{1}_{\{T_i>\tau\}}=r\bigg]\,.
\end{align}
On the other hand, the conditions
$\sum_{i=1}^{L_n}T_i\mathds{1}_{\{T_i>\tau\}}\ge\varkappa n$ and
$\sum_{i=1}^{L_n}\mathds{1}_{\{T_i>\tau\}}=r$ imply that there are
$r$ inter-arrival times that cover more than $\varkappa n$ sites, namely,
there are integers $0\le a_1< b_1\le \cdots \le a_r<b_r\le n$ with
$\sum_{s=1}^r(b_s-a_s)\ge\varkappa n$ such that $X_{a_s}=X_{b_s}=1$ for all
$s$ and, when $a_s+1<b_s$, $X_{a_s+1}=\cdots=X_{b_s-1}=0$. It follows
that
\begin{align}
    \nonumber
    &\prob_{n,h,\omega}\bigg[\sum_{i=1}^{L_n}T_i\mathds{1}_{\{T_i>\tau\}}\ge\varkappa n,\,\sum_{i=1}^{L_n}\mathds{1}_{\{T_i>\tau\}}=r\bigg]\\
    \nonumber
    &\qquad \le\sum_{0\le a_1< b_1\le \cdots \le a_r<b_r\le n}\mathds{1}_{\big\{\sum_{s=1}^r(b_s-a_s)\ge\varkappa n\big\}}
    \Ex_{n,h,\omega}\bigg[\prod_{s=1}^rX_{a_s}\prod_{k=a_s+1}^{b_s-1}(1-X_k)X_{b_s}\bigg]\,,
\end{align}
and Lemma \ref{lem:fact} gives for every $0\le a_1< b_1\le \cdots \le
a_r<b_r\le n$
\begin{align}
  \nonumber
  \Ex_{n,h,\omega}\bigg[\prod_{s=1}^rX_{a_s}\prod_{k=a_s+1}^{b_s-1}(1-X_k)X_{b_s}\bigg]
  &\le\prod_{s=1}^r\Ex_{b_s-a_s,h,\vartheta^{a_s}\omega}\bigg[\prod_{k=1}^{b_s-a_s-1}(1-X_k)\bigg]\\
  \nonumber
  &=\prod_{s=1}^r\prob_{b_s-a_s,h,\vartheta^{a_s}\omega}\big[T_1=b_s-a_s\big]\,.
\end{align}
Putting the pieces together and taking the supremum with respect to $h$
bearing in mind that
$\prob_{m,h,\omega}[T_1=m]=p(m)/Z_{m,h}^{\textbf{-}}(\omega)$ with
$Z_{m,h}^{\textbf{-}}(\omega)=\Ex[\ee^{\sum_{a=1}^{m-1}(h+\omega_a)X_a}X_m]$
is decreasing with respect to $h$ for any $m$ and $\omega$, we deduce
that
\begin{align}
    \nonumber
    &\sup_{h\in H}\prob_{n,h,\omega}\bigg[\sum_{i=1}^{L_n}T_i\mathds{1}_{\{T_i>\tau\}}\ge\varkappa n\bigg]\\
    \nonumber
    &~~\le\sum_{r\in\N}\mathds{1}_{\{r<\frac{n}{\tau}\}}\sum_{0\le a_1< b_1\le \cdots \le a_r<b_r\le n}\mathds{1}_{\big\{\sum_{s=1}^r(b_s-a_s)\ge\varkappa n\big\}}
   \prod_{s=1}^r\prob_{b_s-a_s,h_o,\vartheta^{a_s}\omega}\big[T_1=b_s-a_s\big]\,.
\end{align}
In this way, integrating with respect to $\probd[\dd\omega]$ and using
the statistical independence of the factors
$\prob_{b_s-a_s,h_o,\vartheta^{a_s}\cdot}[T_1=b_s-a_s]$,
(\ref{lem:decay_aux_0}) and an estimate of the number of possible ways
one can place $r$ stretches in $\{1,\ldots,n\}$ give
\begin{align}
    \nonumber
    &\Exd\Bigg[\sup_{h\in H}\prob_{n,h,\cdot}\bigg[\sum_{i=1}^{L_n}T_i\mathds{1}_{\{T_i>\tau\}}>\varkappa n\bigg]\Bigg]\\
    \nonumber
    &\qquad\le\sum_{r\in\N}\mathds{1}_{\{r<\frac{n}{\tau}\}}\sum_{0\le a_1< b_1\le \cdots \le a_r<b_r\le n}\mathds{1}_{\big\{\sum_{s=1}^r(b_s-a_s)\ge\varkappa n\big\}}
    G_o^r\ee^{-\gamma_o\sum_{s=1}^r(b_s-a_s)}\\
    \nonumber
    &\qquad\le \sum_{r\in\N}\mathds{1}_{\{r<\frac{n}{\tau}\}}{n+r \choose 2r}G_o^r\ee^{-\gamma_o\varkappa n}\,.
\end{align}
As ${n+r\choose 2r}\le [(n+r)^r/r!]{n\choose
  r}=\gamma^{-r}[(\gamma n+\gamma r)^r/r!]{n\choose
  r}\le\gamma^{-r}\ee^{\gamma n+\gamma r}{n\choose r}$ and
$\varkappa\gamma_o=4\gamma$, we can even state that
\begin{equation*}
    \Exd\Bigg[\sup_{h\in H}\prob_{n,h,\cdot}\bigg[\sum_{i=1}^{L_n}T_i\mathds{1}_{\{T_i>\tau\}}>\varkappa n\bigg]\Bigg]
\le \sum_{r\in\N}\mathds{1}_{\{r<\frac{n}{\tau}\}}{n \choose r}\bigg(\frac{G_o\ee^\gamma}{\gamma}\bigg)^{\!\!r}\ee^{-3\gamma n}\,.
\end{equation*}
To conclude, set $\zeta:=1+G_o\ee^\gamma/\gamma^2$ for brevity and
note that by construction $\zeta^{1/\tau}\le\ee^\gamma$ and
$G_o\ee^\gamma/\gamma\le \zeta\gamma$. Also note that
$\mathds{1}_{\{r<n/\tau\}}\le\zeta^{n/\tau-r}$ as $\zeta\ge 1$, so
that $\mathds{1}_{\{r<\frac{n}{\tau}\}}\le
\ee^{n\gamma}\zeta^{-r}$. Then
\begin{align}
  \nonumber
    ~~~~~~~~~~~~~~~~~~\Exd\Bigg[\sup_{h\in H}\prob_{n,h,\cdot}\bigg[\sum_{i=1}^{L_n}T_i\mathds{1}_{\{T_i>\tau\}}>\varkappa n\bigg]\Bigg]
      &\le \sum_{r=0}^n\ee^{n\gamma}\zeta^{-r}{n \choose r}(\zeta\gamma)^r\ee^{-3\gamma n}\\
      \nonumber
      &=(1+\gamma)^n\ee^{-2\gamma n}\le \ee^{-\gamma n}\,.
      ~~~~~~~~~~~~~~~~~~~~~~\qedhere
\end{align}
\end{proof}

\begin{proof}[Proof of Lemma \ref{lem:decay}]
Fix a closed set $H\subset(h_c,+\infty)$.  Lemmas \ref{lem:density}
and \ref{lem:decay_aux} ensure us that there exist constants
$\delta>0$, $\tau\in\N$, and $\gamma_o>0$ such that for all
$n\in\N$
\begin{equation}
  \Exd\bigg[\sup_{h\in H}\prob_{n,h,\cdot}[L_n<\delta n]\bigg]\le \ee^{-\gamma_o n}
  \label{lem:decay_1}
\end{equation}
and
  \begin{equation}
    \Exd\Bigg[\sup_{h\in H}\prob_{n,h,\cdot}\bigg[\sum_{i=1}^{L_n}T_i\mathds{1}_{\{T_i>\tau\}}\ge\frac{\delta}{2}n\bigg]\Bigg]\le \ee^{-\gamma_o n}\,.
    \label{lem:decay_2}
  \end{equation}
Moreover, if we pick a number $\lambda<0$ so large in modulus that
$\int_\Omega \mathds{1}_{\{\omega_0<\lambda\}}\probd[\dd\omega]\le
\delta/32$, then a Chernoff--type bound yields
\begin{align}
  \nonumber
  \int_\Omega \mathds{1}_{\big\{\sum_{a=1}^n\mathds{1}_{\{\omega_a<\lambda\}}>\frac{\delta}{4}n\big\}}\probd[\dd\omega]&\le
  \ee^{-\frac{\delta}{2} n}\int_\Omega \ee^{2\sum_{a=1}^n\mathds{1}_{\{\omega_a<\lambda\}}}\probd[\dd\omega]\\
  \nonumber
  &=\ee^{-\frac{\delta}{2}n}\bigg\{1+(\ee^2-1)\int_\Omega \mathds{1}_{\{\omega_0<\lambda\}}\probd[\dd\omega]\bigg\}^n\\
  &\le \ee^{-\frac{\delta}{2}n}\bigg(1+\frac{\delta}{4}\bigg)^n\le \ee^{-\frac{\delta}{4} n}\,.
  \label{lem:decay_3}
 \end{align}
We will show below that for all $n\in\N$, $h\in\Rl$, and
$\omega:=\{\omega_a\}_{a\in\N_0}\in\Omega$
\begin{equation}
\Ex_{n,h,\omega}^{\otimes 2}\bigg[\mathds{1}_{\big\{\sum_{a=1}^{n-1}X_a'\mathds{1}_{\{\omega_a\ge\lambda,\,T_{L_a+1}\le\tau\}}\ge\frac{\delta}{4}n\big\}}
    \prod_{k=1}^{n-1}(1-X_kX_k')\bigg]\le 2\tau\bigg(\frac{\xi\tau^\xi}{\xi\tau^\xi+\ee^{h+\lambda}}\bigg)^{\frac{\delta n}{8\tau}},
\label{lem:decay_4}
\end{equation}
where $\xi$ is the number introduced in (\ref{eq:xi_def}).  Bounds
(\ref{lem:decay_1}), (\ref{lem:decay_2}), (\ref{lem:decay_3}), and
(\ref{lem:decay_4}) imply the lemma as follows.  The conditions
$\sum_{a=1}^n\mathds{1}_{\{\omega_a<\lambda\}}\le n{\delta}/4$,
$L_n'=\sum_{a=1}^nX_a'\ge\delta n$, and
$\sum_{a=1}^{n-1}X_a'\mathds{1}_{\{\omega_a\ge\lambda,\,T_{L_a+1}\le\tau\}}<
n \delta/4$ give $\sum_{a=0}^{n-1}\mathds{1}_{\{T_{L_a+1}>\tau\}}\ge
n{\delta}/{2}$. In turn, the latter yields for
$\{S_i\}_{i\in\N_0}\ni n$
\begin{equation*}
  \frac{\delta}{2}n\le\sum_{a=0}^{n-1}\mathds{1}_{\{T_{L_a+1}>\tau\}}=\sum_{i=1}^{L_n}\sum_{a=S_{i-1}}^{S_i-1}\mathds{1}_{\{T_{L_a+1}>\tau\}}
  =\sum_{i=1}^{L_n}T_i\mathds{1}_{\{T_i>\tau\}}\,.
\end{equation*}
Thus, for every $n\in\N$, $h\in\Rl$, and
$\omega:=\{\omega_a\}_{a\in\N_0}\in\Omega$ we find
\begin{align}
  \nonumber
  &\Ex_{n,h,\omega}^{\otimes 2}\bigg[\prod_{k=1}^{n-1}(1-X_kX_k')\bigg]\\
  \nonumber
  &\qquad\le  \mathds{1}_{\big\{\sum_{a=1}^n\mathds{1}_{\{\omega_a<\lambda\}}>\frac{\delta}{4}n\big\}}+
  \prob_{n,h,\omega}[L_n<\delta n]+\prob_{n,h,\omega}\bigg[\sum_{i=1}^{L_n}T_i\mathds{1}_{\{T_i>\tau\}}\ge\frac{\delta}{2}n\bigg]\\
  \nonumber
  &\qquad\quad +\Ex_{n,h,\omega}^{\otimes 2}\bigg[\mathds{1}_{\big\{\sum_{a=1}^{n-1}X_a'\mathds{1}_{\{\omega_a\ge\lambda,\,T_{L_a+1}\le\tau\}}\ge\frac{\delta}{4}n\big\}}
    \prod_{k=1}^{n-1}(1-X_kX_k')\bigg]\\
  \nonumber
  &\qquad\le  \mathds{1}_{\big\{\sum_{a=1}^n\mathds{1}_{\{\omega_a<\lambda\}}>\frac{\delta}{4}n\big\}}+
  \prob_{n,h,\omega}[L_n<\delta n]+\prob_{n,h,\omega}\bigg[\sum_{i=1}^{L_n}T_i\mathds{1}_{\{T_i>\tau\}}\ge\frac{\delta}{2}n\bigg]\\
  \nonumber
  &\qquad\quad+2\tau\bigg(\frac{\xi\tau^\xi}{\xi\tau^\xi+\ee^{h+\lambda}}\bigg)^{\frac{\delta n}{8\tau}}.
\end{align}
Taking the supremum with respect to $h$ in this inequality, defining
$h_o:=\inf H$, and integrating with respect to $\probd[\dd\omega]$, we
get for all $n\in\N$
\begin{equation*}
  \Exd\Bigg[\sup_{h\in H}\Ex_{n,h,\cdot}^{\otimes 2}\bigg[\prod_{k=1}^{n-1}(1-X_kX_k')\bigg]\Bigg]\le \ee^{-\frac{\delta}{4} n}+2\ee^{-\gamma_o n}
  +2\tau\bigg(\frac{\xi\tau^\xi}{\xi\tau^\xi+\ee^{h_o+\lambda}}\bigg)^{\frac{\delta n}{8\tau}}\,,
\end{equation*}
which allows us to verify the lemma with the constants $G:=3+2\tau$ and
\begin{equation*}
  \gamma:=\min\left\{\frac{\delta}{4},\,\gamma_o,\,\frac{\delta}{8\tau}\log\frac{\xi\tau^\xi+\ee^{h_o+\lambda}}{\xi\tau^\xi}\right\}\,.
\end{equation*}

\smallskip

\noindent\textit{The bound \eqref{lem:decay_4}.}  In order to demonstrate
\eqref{lem:decay_4}, we introduce for $a\in\N$
the variable
\begin{align}
  \nonumber
  Y_a&:=X_a\mathds{1}_{\{T_{L_a}+T_{L_a+1}\le\tau\}}+(1-X_a)\mathds{1}_{\{T_{L_a+1}\le\tau\}}\\
  &=\sum_{i=\max\{0,a+1-\tau\}}^{a-1}\,\sum_{j=a+1}^{i+\tau}X_i\prod_{\substack{k=i+1\\k\ne a}}^{j-1}(1-X_k)X_j\,,
  \label{eq:def_Ya}
\end{align}
namely, $Y_a$ is the indicator function that the distance between the
first contact before $a$ and the first contact after $a$ is at most
$\tau$. The variable $Y_a$ has the following nice property, which is a
direct consequence of Lemma \ref{lem:utile_per_tutto}: for every
non-negative real functions $\phi\in\mathcal{M}_{a-1}$ and
$\psi\in\mathcal{M}_{n-a-1}$
\begin{equation}
  \Ex_{n,h,\omega}\big[\Phi (1-X_a)Y_a\Psi\big]\le \frac{\xi\tau^\xi}{\xi\tau^\xi+\ee^{h+\omega_a}}\,\Ex_{n,h,\omega}\big[\Phi Y_a\Psi\big]\,,
\label{eq:prop_Ya}
\end{equation}
where we have set $\Phi:=\phi(X_0,\ldots,X_{a-1})$ and
$\Psi:=\psi(X_{a+1},\ldots,X_n)$ for brevity. This bound shows how the
constraint $X_a=0$ can be relaxed when $Y_a$ is involved. For its
proof we invoke Lemma \ref{lem:utile_per_tutto} to obtain
\begin{align}
  \nonumber
  \Ex_{n,h,\omega}\big[\Phi (1-X_a)Y_a\Psi\big]&=\Ex_{n,h,\omega}\Big[\Phi (1-X_a)\mathds{1}_{\{T_{L_a+1}\le\tau\}}\Psi\Big]\\
  \nonumber
  &\le\xi\tau^\xi\ee^{-h-\omega_a}\,\Ex_{n,h,\omega}\Big[\Phi X_a\mathds{1}_{\{T_{L_a}+T_{L_a+1}\le\tau\}}\Psi\Big]\\
  \nonumber
  &=\xi\tau^\xi\ee^{-h-\omega_a}\,\Ex_{n,h,\omega}\big[\Phi X_aY_a\Psi\big]\,,
\end{align}
which is nothing but a different way to write (\ref{eq:prop_Ya}).

Let us put the variables $Y_a$ into context. Since
$X_a'Y_a=X_a'\mathds{1}_{\{T_{L_a+1}\le\tau\}}$ under the condition
that $X_a=0$ if $X_a'=1$, for every $n\in\N$ and
$\omega:=\{\omega_a\}_{a\in\N_0}\in\Omega$ we have
\begin{align}
  \nonumber
  &\Ex_{n,h,\omega}^{\otimes 2}\bigg[\mathds{1}_{\big\{\sum_{a=1}^{n-1}X_a'\mathds{1}_{\{\omega_a\ge\lambda,\,T_{L_a+1}\le\tau\}}\ge\frac{\delta}{4}n\big\}}
    \prod_{k=1}^{n-1}(1-X_kX_k')\bigg]\\
  \nonumber
  &\qquad =\Ex_{n,h,\omega}^{\otimes 2}\bigg[\mathds{1}_{\big\{\sum_{a=1}^{n-1}X_a'Y_a\mathds{1}_{\{\omega_a\ge\lambda\}}\ge\frac{\delta}{4}n\big\}}
    \prod_{k=1}^{n-1}(1-X_kX_k')\bigg]\,.
\end{align}
For a reason that will be clear in a moment, we prefer that the index
$a$ in the last expression runs over sublattices with spacing
$2\tau$. Thus, for $r\in\{1,\ldots,2\tau\}$ put
$\mathbb{A}_r:=\{r+2\tau i: i\in\N_0\mbox{ and }r+2\tau
i<n\}$. Since $\cup_{r=1}^{2\tau}\mathbb{A}_r=\{1,\ldots,n-1\}$, the
condition
$\sum_{a=1}^{n-1}X_a'Y_a\mathds{1}_{\{\omega_a\ge\lambda\}}\ge
n{\delta}/{4}$ entails that there exists $r$ such that
$\sum_{a\in\mathbb{A}_r}X_a'Y_a\mathds{1}_{\{\omega_a\ge\lambda\}}\ge
n{\delta}/8\tau$. In this way, we find
\begin{align}
  \nonumber
  &\Ex_{n,h,\omega}^{\otimes 2}\bigg[\mathds{1}_{\big\{\sum_{a=1}^{n-1}X_a'\mathds{1}_{\{\omega_a\ge\lambda,\,T_{L_a+1}\le\tau\}}\ge\frac{\delta}{4}n\big\}}
  \prod_{k=1}^{n-1}(1-X_kX_k')\bigg]\\
  \nonumber
  &\qquad\le\sum_{r=1}^{2\tau}
  \Ex_{n,h,\omega}^{\otimes 2}\bigg[\mathds{1}_{\big\{\sum_{a\in\mathbb{A}_r}X_a'Y_a\mathds{1}_{\{\omega_a\ge\lambda\}}\ge\frac{\delta}{8\tau}n\big\}}
    \prod_{a\in\mathbb{A}_r}(1-X_aX_a')\bigg]\\
  \nonumber
  &\qquad=\sum_{r=1}^{2\tau}\sum_{x_1\in\{0,1\}}\cdots\sum_{x_n\in\{0,1\}}{}^r\!E_{n,h,\omega}(x_1,\ldots,x_n)\,\Ex_{n,h,\omega}\bigg[\prod_{k=1}^n\mathds{1}_{\{X_k=x_k\}}\bigg]
\end{align}
with
\begin{equation*}
  {}^r\!E_{n,h,\omega}(x_1,\ldots,x_n):=\Ex_{n,h,\omega}\bigg[\mathds{1}_{\big\{\sum_{a\in\mathbb{A}_r}x_aY_a\mathds{1}_{\{\omega_a\ge\lambda\}}\ge\frac{\delta}{8\tau}n\big\}}
    \prod_{a\in\mathbb{A}_r}(1-x_aX_a)\bigg]\,.
\end{equation*}
The last step to obtain bound (\ref{lem:decay_4}) is to show that
\begin{equation}
  {}^r\!E_{n,h,\omega}(x_1,\ldots,x_n)\le \bigg(\frac{\xi\tau^\xi}{\xi\tau^\xi+\ee^{h+\lambda}}\bigg)^{\frac{\delta}{8\tau}n}
  \label{lem:decay_5}
\end{equation}
uniformly with respect to $r$ and $x_1,\ldots,x_n$.  This is
facilitated by the fact that $Y_a$ and $Y_b$ for different $a$ and $b$
in a sublattice of spacing $2\tau$ do not overlap, in the sense that
depend on the non-overlapping sets
$\{X_{\max\{0,a-\tau+1\}},\ldots,X_{a+\tau-1}\}$ and
$\{X_{\max\{0,b-\tau+1\}},\ldots,X_{b+\tau-1}\}$ of binary variables, as
(\ref{eq:def_Ya}) makes evident.  Given $r\in\{1,\ldots,2\tau\}$
and $x_1,\ldots,x_n\in\{0,1\}$, to prove (\ref{lem:decay_5}) we expand
the indicator function in ${}^r\!E_{n,h,\omega}(x_1,\ldots,x_n)$ to
get
\begin{align}
  \nonumber
  &{}^r\!E_{n,h,\omega}(x_1,\ldots,x_n)\\
  \nonumber
  &\qquad\le\sum_{\mathcal{A}\subseteq\mathbb{A}_r}\mathds{1}_{\{|\mathcal{A}|\ge\frac{\delta}{8\tau}n\}}
  \Ex_{n,h,\omega}\bigg[\prod_{a\in\mathcal{A}}x_a(1-X_a)Y_a\mathds{1}_{\{\omega_a\ge\lambda\}}\prod_{a\in\mathbb{A}_r\setminus\mathcal{A}}
    \big(1-x_aY_a\mathds{1}_{\{\omega_a\ge\lambda\}}\big)\bigg]\,.
\end{align}
We are going to relax the constraints that $X_a=0$ for
$a\in\mathcal{A}$.  Bearing in mind that $Y_a$ and $Y_b$ for distinct
$a,b\in\mathbb{A}_r$ do not overlap, repeated applications of
(\ref{eq:prop_Ya}) yield for any $\mathcal{A}\subseteq\mathbb{A}_r$
  \begin{align}
  \nonumber
  &\Ex_{n,h,\omega}\bigg[\prod_{a\in\mathcal{A}}x_a(1-X_a)Y_a\mathds{1}_{\{\omega_a\ge\lambda\}}\prod_{a\in\mathbb{A}_r\setminus\mathcal{A}}
    \big(1-x_aY_a\mathds{1}_{\{\omega_a\ge\lambda\}}\big)\bigg]\\
  \nonumber
  &\qquad \le\prod_{a\in\mathcal{A}}\frac{\xi\tau^\xi}{\xi\tau^\xi+\ee^{h+\omega_a}}\,
  \Ex_{n,h,\omega}\bigg[\prod_{a\in\mathcal{A}}x_aY_a\mathds{1}_{\{\omega_a\ge\lambda\}}\prod_{a\in\mathbb{A}_r\setminus\mathcal{A}}
    \big(1-x_aY_a\mathds{1}_{\{\omega_a\ge\lambda\}}\big)\bigg]\\
  \nonumber
  &\qquad\le\bigg(\frac{\xi\tau^\xi}{\xi\tau^\xi+\ee^{h+\lambda}}\bigg)^{|\mathcal{A}|}\,
  \Ex_{n,h,\omega}\bigg[\prod_{a\in\mathcal{A}}x_aY_a\mathds{1}_{\{\omega_a\ge\lambda\}}\prod_{a\in\mathbb{A}_r\setminus\mathcal{A}}
    \big(1-x_aY_a\mathds{1}_{\{\omega_a\ge\lambda\}}\big)\bigg]\,.
  \end{align}
We conclude that 
\begin{align}
  \nonumber
  &{}^r\!E_{n,h,\omega}(x_1,\ldots,x_n)\\
  \nonumber
  &\qquad\le\bigg(\frac{\xi\tau^\xi}{\xi\tau^\xi+\ee^{h+\lambda}}\bigg)^{\frac{\delta}{8\tau}n}\,
  \prob_{n,h,\omega}\bigg[\sum_{a\in\mathbb{A}_r}x_aY_a\mathds{1}_{\{\omega_a\ge\lambda\}}\ge\frac{\delta}{8\tau}n\bigg]\,,
\end{align}
which proves (\ref{lem:decay_5}).
\end{proof}

We can now go into the first application of Lemmas \ref{lem:corr} and
\ref{lem:decay}.

\begin{lemma}
  \label{lem:mixing}
  For each closed set $H\subset(h_c,+\infty)$ there exist constants
  $\gamma>0$ and $G>0$ such that for every integers $0\le a\le b\le n$
  \begin{equation*}
   \Exd\bigg[ \sup_{h\in H}\,\frac{|\cov_{n,h,\cdot}[X_a,X_b]|}{\min\{\Ex_{n,h,\cdot}[X_a],\Ex_{n,h,\cdot}[X_b]\}}\bigg]\le G\ee^{-\gamma(b-a)}\,.
  \end{equation*}
\end{lemma}

\begin{proof}[Proof of Lemma \ref{lem:mixing}]
Fix any $H\subset(h_c,+\infty)$ closed.  Lemma \ref{lem:decay} ensures
us the existence of two constants $\gamma_o>0$ and $G_o>0$ such that
for all $n\in\N$
\begin{equation}
 \Exd\Bigg[ \sup_{h\in H}\Ex_{n,h,\cdot}^{\otimes 2}\bigg[\prod_{k=1}^{n-1}(1-X_kX_k')\bigg]\Bigg]\le G_o\ee^{-\gamma_o n}\,.
  \label{lem:mixing_0}
\end{equation}
The second part of Lemma \ref{lem:corr} states that for every integers
$0\le a\le b\le n$, $h\in\Rl$, and
$\omega:=\{\omega_c\}_{c\in\N_0}\in\Omega$
\begin{align}
  \nonumber
  &\frac{|\cov_{n,h,\omega}[X_a,X_b]|}{\min\{\Ex_{n,h,\omega}[X_a],\Ex_{n,h,\omega}[X_b]\}}\\
  &\qquad\le\sum_{i=0}^a\sum_{j=b}^\infty 2\xi (j-i)^\xi\ee^{-h-\min\{\omega_a,\omega_b\}}\Ex_{j-i,h,\vartheta^i\omega}^{\otimes 2}\bigg[\prod_{k=1}^{j-i-1}(1-X_kX_k')\bigg]\,,
\label{lem:mixing_1}
\end{align}
where $\xi$ is the number defined in \eqref{eq:xi_def}.  Since the
left-hand side of (\ref{lem:mixing_1}) is smaller than or equal to
$1$, we can further state that
\begin{align}
  \nonumber
  &\frac{|\cov_{n,h,\omega}[X_a,X_b]|}{\min\{\Ex_{n,h,\omega}[X_a],\Ex_{n,h,\omega}[X_b]\}}\\
  \nonumber
  &\qquad\le \mathds{1}_{\big\{|\omega_a|>\frac{\gamma_o}{3}(b-a)\big\}}+\mathds{1}_{\big\{|\omega_b|>\frac{\gamma_o}{3}(b-a)\big\}}\\
  \nonumber
  &\qquad\quad+\sum_{i=0}^a\sum_{j=b}^\infty 2\xi (j-i)^\xi\ee^{-h+\frac{\gamma_o}{3}(b-a)}\Ex_{j-i,h,\vartheta^i\omega}^{\otimes 2}\bigg[\prod_{k=1}^{j-i-1}(1-X_kX_k')\bigg]\,.
\end{align}
In this way, taking the supremum with respect to $h$, defining
$h_o:=\inf H$, and then integrating with respect to
$\probd[\dd\omega]$, we get thanks to (\ref{lem:mixing_0}) and to a
Chernoff--type bound
\begin{align}
  \nonumber
  &\Exd\bigg[\sup_{h\in H}\,\frac{|\cov_{n,h,\cdot}[X_a,X_b]|}{\min\{\Ex_{n,h,\cdot}[X_a],\Ex_{n,h,\cdot}[X_b]\}}\bigg]\\
  \nonumber
 &\qquad\le 2\int_\Omega\mathds{1}_{\big\{|\omega_0|>\frac{\gamma_o}{3}(b-a)\big\}}\probd[\dd\omega]+\sum_{i=0}^a\sum_{j=b}^\infty 2\xi (j-i)^\xi\ee^{-h_o+\frac{\gamma_o}{3}(b-a)} G_o\ee^{-\gamma_o (j-i)}\\
  \nonumber
  &\qquad\le 2\int_\Omega\ee^{\eta|\omega_0|-\frac{\eta\gamma_o}{3}(b-a)}\,\probd[\dd\omega]+
 \sum_{j\in\N} 2\xi G_o (j+1)j^\xi\ee^{-h_o-\frac{\gamma_o}{3}(b-a)-\frac{\gamma_o}{3}j}\,,
\end{align}
$\eta>0$ being the number introduced by Assumption \ref{assump:omega}.
This shows that the lemma holds with
$\gamma:=\min\{\eta\gamma_o/3,\gamma_o/3\}$ and $G:=2\int_\Omega\ee^{\eta|\omega_0|}\,\probd[\dd\omega]+
\sum_{j\in\N} 2\xi G_o (j+1)j^\xi\ee^{-h_o-\gamma_oj/3}$.
\end{proof}

One corollary of Lemma \ref{lem:mixing} is the following.

\begin{corollary}
  \label{cor:mixing}
  For each $H\subset(h_c,+\infty)$ closed there exist constants
  $\gamma>0$ and $G>0$ such that for every integers $0\le a\le b\le n$
  \begin{equation*}
    \sup_{h\in H}\bigg|\Exd\Big[\Ex_{n,h,\cdot}[X_a]\Ex_{n,h,\cdot}[X_b]\Big]-\Exd\Big[\Ex_{n,h,\cdot}[X_a]\Big]\Exd\Big[\Ex_{n,h,\cdot}[X_b]\Big]\bigg|\le G\ee^{-\gamma(b-a)}\,.
  \end{equation*}
\end{corollary}

\begin{proof}[Proof of Corollary \ref{cor:mixing}]
Fix a closed set $H\subset(h_c,+\infty)$ and integers $0\le a\le b\le
n$. Put $m:=\lfloor(a+b)/2\rfloor$. Lemma \ref{lem:fact} gives for all
$h\in\Rl$ and $\omega\in\Omega$
  \begin{align}
    \nonumber
\Ex_{n,h,\omega}[X_a]&=\frac{\Ex_{n,h,\omega}[X_aX_m]}{\Ex_{n,h,\omega}[X_m]}-\frac{\Ex_{n,h,\omega}[X_aX_m]-\Ex_{n,h,\omega}[X_a]\Ex_{n,h,\omega}[X_m]}{\Ex_{n,h,\omega}[X_m]}\\
    &=\Ex_{m,h,\omega}[X_a]-\frac{\cov_{n,h,\omega}[X_a,X_m]}{\Ex_{n,h,\omega}[X_m]}
    \label{cor:mixing_1}
  \end{align}
  and
  \begin{align}
    \nonumber
    \Ex_{n,h,\omega}[X_b]&=\frac{\Ex_{n,h,\omega}[X_mX_b]}{\Ex_{n,h,\omega}[X_m]}-\frac{\Ex_{n,h,\omega}[X_mX_b]-\Ex_{n,h,\omega}[X_m]\Ex_{n,h,\omega}[X_b]}
       {\Ex_{n,h,\omega}[X_m]}\\
       &=\Ex_{n-m,h,\vartheta^m\omega}[X_{b-m}]-\frac{\cov_{n,h,\omega}[X_m,X_b]}{\Ex_{n,h,\omega}[X_m]}\,.
        \label{cor:mixing_2}
  \end{align}
Lemma \ref{lem:mixing} implies that there exist constants $\gamma_o>0$
and $G_o>0$ independent of $a$, $b$, and $n$ such that
\begin{align}
  \nonumber
  \sup_{h\in H}\Exd\bigg[\frac{|\cov_{n,h,\cdot}[X_a,X_m]|}{\Ex_{n,h,\cdot}[X_m]}\bigg]\le
  \Exd\bigg[\sup_{h\in H}\frac{|\cov_{n,h,\cdot}[X_a,X_m]|}{\Ex_{n,h,\cdot}[X_m]}\bigg]&\le G_o\ee^{-\gamma_o(m-a)}\\
  \nonumber
  &\le G_o\ee^{-\frac{\gamma_o}{2}(b-a)+\gamma_o}
\end{align}
and, similarly,
\begin{equation*}
  \sup_{h\in H}\Exd\bigg[\frac{|\cov_{n,h,\cdot}[X_m,X_b]|}{\Ex_{n,h,\cdot}[X_m]}\bigg]\le G_o\ee^{-\gamma_o(b-m)}\le
  G_o\ee^{-\frac{\gamma_o}{2}(b-a)}\,.
\end{equation*}
In this way, since the expectation $\Ex_{m,h,\cdot}[X_a]$ and
$\Ex_{n-m,h,\vartheta^m\cdot}[X_{b-m}]$ are statistically independent,
we can conclude from (\ref{cor:mixing_1}) and (\ref{cor:mixing_2})
that for all $h\in H$
\begin{align}
  \nonumber
  &\bigg|\Exd\Big[\Ex_{n,h,\cdot}[X_a]\Ex_{n,h,\cdot}[X_b]\Big]-\Exd\Big[\Ex_{n,h,\cdot}[X_a]\Big]\Exd\Big[\Ex_{n,h,\cdot}[X_b]\Big]\bigg|\\
  \nonumber
  &\qquad\le 2\Exd\bigg[\frac{|\cov_{n,h,\cdot}[X_a,X_m]|}{\Ex_{n,h,\cdot}[X_m]}\bigg]+2\Exd\bigg[\frac{|\cov_{n,h,\cdot}[X_m,X_b]|}{\Ex_{n,h,\cdot}[X_m]}\bigg]
  \le 2G_o(\ee^{\gamma_o}+1)\ee^{-\frac{\gamma_o}{2}(b-a)}\,,
\end{align}
which proves the corollary with $\gamma:={\gamma_o}/{2}$ and
$G:=2G_o(\ee^{\gamma_o}+1)$.
\end{proof}

\subsection{First consequences of the correlation estimates}
\label{sec:maximal_excursion}

In this section we investigate the maximal polymer excursion
$M_n:=\max\{T_1,\ldots,T_{L_n}\}$ in the localized phase, so providing
a first application of the correlation bounds and proving Proposition
\ref{prop:maximal_excursion}.

\begin{proof}[Proof of Proposition \ref{prop:maximal_excursion}]
We shall show that for every $h_o>h_c$ and $\epsilon\in(0,1]$ and for $\pae$
\begin{equation}
  \adjustlimits\limspace_{n\uparrow\infty}\sup_{h\in U_{h_o,\epsilon}}\prob_{n,h,\omega}\bigg[\frac{M_n}{\log n}>\frac{1+\epsilon}{\mu(h)}\bigg]=0
\label{eq:max_uncond_1}
\end{equation}
and
\begin{equation}
  \adjustlimits\limspace_{n\uparrow\infty}\sup_{h\in U_{h_o,\epsilon}}\prob_{n,h,\omega}\bigg[\frac{M_n}{\log n}<\frac{1-\epsilon}{\mu(h)}\bigg]=0\,,
\label{eq:max_uncond_25}
\end{equation}
where $U_{h_o,\epsilon}$ is an open interval centered at $h_o$ with
radius $\min\{\epsilon\mu(h_o)/8,(h_o-h_c)/2\}>0$.  These limits prove
the proposition as follows. For $\epsilon\in(0,1]$, Lindel\"of's lemma
  ensures that from the open cover
  $\{U_{h_o,\epsilon}\}_{h_o\in(h_c,+\infty)}$ of $(h_c,+\infty)$ we
  can extract countably many elements corresponding to the points
  $h_o$ in a countable set $\mathbb{H}_\epsilon\subset(h_c,+\infty)$
  such that $(h_c,+\infty)\subseteq\cup_{h_o\in
    \mathbb{H}_\epsilon}U_{h_o,\epsilon}$.  Let
  $\Omega_{h_o,\epsilon}\in\mathcal{F}$ be a set of charges with full
  probability, that is $\probd[\Omega_{h_o,\epsilon}]=1$, for which
  \eqref{eq:max_uncond_1} and \eqref{eq:max_uncond_25} are valid and
  define
  $\Omega_o:=\cap_{i\in\N}\cap_{h_o\in\mathbb{H}_{1/i}}\Omega_{h_o,1/i}$. We
  have $\probd[\Omega_o]=1$.  By construction, for every
  $\omega\in\Omega_o$, $H\subset(h_c,+\infty)$ compact, and
  $\epsilon>0$ there exist $i\in\N$ and finitely many points
  $h_1,\ldots,h_r$ in $\mathbb{H}_{1/i}$ such that $\inf_{h\in
    H}\mu(h)\epsilon>1/i$, $H\subset \cup_{s=1}^rU_{h_s,1/i}$, and
  $\omega\in\Omega_{h_s,1/i}$ for all $s\in\{1,\ldots,r\}$. We remark
  that $\inf_{h\in H}\mu(h)>0$ because $\mu$ is Lipschitz continuous
  and strictly positive throughout $(h_c,+\infty)$.  It therefore
  follows from \eqref{eq:max_uncond_1} and \eqref{eq:max_uncond_25}
  that
\begin{align}
  \nonumber
  &\limsup_{n\uparrow\infty}\,\sup_{h\in H}\,\prob_{n,h,\omega}\Bigg[\bigg|\frac{M_n}{\log n}-\frac{1}{\mu(h)}\bigg|>\epsilon\Bigg]\\
  \nonumber
  &\qquad\le\adjustlimits\limsmallspace_{n\uparrow\infty}\maxp_{s\in\{1,\ldots,r\}}\,\sup_{h\in U_{h_s,1/i}}
  \prob_{n,h,\omega}\Bigg[\frac{M_n}{\log n}>\frac{1+1/i}{\mu(h)}\Bigg]\\
  \nonumber
  &\qquad\quad+\adjustlimits\limsmallspace_{n\uparrow\infty}\maxp_{s\in\{1,\ldots,r\}}\,\sup_{h\in U_{h_s,1/i}}
  \prob_{n,h,\omega}\Bigg[\frac{M_n}{\log n}<\frac{1-1/i}{\mu(h)}\Bigg]=0\,.
\end{align}
This demonstrates that, as $n$ goes to infinity, $M_n/\log n$
converges in probability to $1/\mu(h)$ with respect to the polymer
measure $\prob_{n,h,\omega}$ for all $\omega\in\Omega_o$ and $h>h_c$,
and the result is uniform in $h$ belonging to compact sets in the
localized phase.

\smallskip

\noindent\textit{The limit \eqref{eq:max_uncond_1}.}  Fix $h_o>h_c$
and $\epsilon\in(0,1]$ and put
  $\delta_o:=\min\{\epsilon\mu_o/8,(h_o-h_c)/2\}$ with
  $\mu_o:=\mu(h_o)>0$, so that
  $U_{h_o,\epsilon}=(h_o-\delta_o,h_o+\delta_o)$. Noting that
  $\mu(h)\le\mu_o+\delta_o$ for $h\in(h_o-\delta_o,h_o+\delta_o)$, as
  the function $\mu$ is Lipschitz continuous with Lipschitz constant
  equal to $1$ by Proposition \ref{prop:mu}, and setting
  $m_n:=\lfloor\{(1+\epsilon)/(\mu_o+\delta_o)\}\log n\rfloor$ for
  brevity, we write for every $n\in\N$, $h\in U_{h_o,\epsilon}$, and
  $\omega\in\Omega$
\begin{align}
  \nonumber
  \prob_{n,h,\omega}\bigg[\frac{M_n}{\log n}>\frac{1+\epsilon}{\mu(h)}\bigg]&\le\prob_{n,h,\omega}\bigg[\frac{M_n}{\log n}>\frac{1+\epsilon}{\mu_o+\delta_o}\bigg]\\
  \nonumber
  &=\prob_{n,h,\omega}\big[M_n>m_n\big]\\
  \nonumber
  &\le \sum_{i=0}^{n-m_n-1}\Ex_{n,h,\omega}\bigg[X_i\prod_{k=i+1}^{i+m_n}(1-X_k)\bigg]\\
  \nonumber
  &= \sum_{i=0}^{n-m_n-1}\sum_{j=i+m_n+1}^n\Ex_{n,h,\omega}\bigg[X_i\prod_{k=i+1}^{j-1}(1-X_k)X_j\bigg]\,,
\end{align}
where the second inequality is due to the fact that the condition
$M_n>m_n$ implies that there is a contact site
$i\in\{0,\ldots,n-m_n-1\}$ followed by an excursion of size at least
$m_n$, while the last equality exploits the identity
\begin{equation*}
  X_i\prod_{k=i+1}^{i+m_n}(1-X_k)=\sum_{j=i+m_n+1}^nX_i\prod_{k=i+1}^{j-1}(1-X_k)X_j+X_i\prod_{k=i+1}^n(1-X_k)\,,
\end{equation*}
as well as the fact that $X_n=1$ almost surely with respect to the
polymer measure $\prob_{n,h,\omega}$.  An application of Lemma
\ref{lem:fact} first and a Chernoff--type bound later give
\begin{align}
  \nonumber
  \prob_{n,h,\omega}\bigg[\frac{M_n}{\log n}>\frac{1+\epsilon}{\mu(h)}\bigg]&\le \sum_{i=0}^{n-m_n-1}\sum_{j=i+m_n+1}^n\prob_{j-i,h,\vartheta^i\omega}[T_1=j-i]\\
  \nonumber
  &\le \sum_{i=0}^{n-1}\sum_{j=m_n+1}^\infty\prob_{j,h,\vartheta^i\omega}[T_1=j]\\
  &\le\sum_{i=0}^{n-1}\sum_{j\in\N}\ee^{\zeta_o j-\zeta_o(m_n+1)}\,\prob_{j,h,\vartheta^i\omega}[T_1=j]\,,
  \label{eq:max_uncond_2}
\end{align}
where we have introduced the number $\zeta_o:=(1-\epsilon/4)\mu_o>0$.  We
now note that for all $j\in\N$, $h\in(h_o-\delta_o,h_o+\delta_o)$, and
$\omega:=\{\omega_a\}_{a\in\N_0}\in\Omega$
\begin{align}
  \nonumber
  \prob_{j,h,\omega}[T_1=j]&=\frac{p(j)}{\Ex[\ee^{\sum_{a=1}^{j-1}(h+\omega_a)X_a}X_j]}\\
  \nonumber
  &\le\ee^{\delta_o j}\,\frac{p(j)}{\Ex[\ee^{\sum_{a=1}^{j-1}(h_o+\omega_a)X_a}X_j]}=\ee^{\delta_o j}\,\prob_{j,h_o,\omega}[T_1=j]\,.
\end{align}
Similarly, we have
\begin{equation}
   \label{eq:max_uncond_100}
  \prob_{j,h,\omega}[T_1=j]\ge \ee^{-\delta_o j}\,\prob_{j,h_o,\omega}[T_1=j]\,.
\end{equation}
Thus, (\ref{eq:max_uncond_2}) finally yields
\begin{align}
  \nonumber
  \sup_{h\in U_{h_o,\epsilon}}\prob_{n,h,\omega}\bigg[\frac{M_n}{\log n}>\frac{1+\epsilon}{\mu(h)}\bigg]
  &\le  \ee^{-\zeta_o(m_n+1)}\sum_{i=0}^{n-1}\sum_{j\in\N}\ee^{(\zeta_o+\delta_o)j}\,\prob_{j,h_o,\vartheta^i\omega}[T_1=j]\\
  &\le  \frac{1}{n^{\zeta_o\frac{1+\epsilon}{\mu_o+\delta_o}}}\sum_{i=0}^{n-1}\Lambda(\vartheta^i\omega)\,,
  \label{eq:max_uncond_3}
\end{align}
$\Lambda$ being the random variable on $\Omega$ defined as
\begin{equation*}
  \Lambda:=\sum_{j\in\N}\ee^{(\zeta_o+\delta_o)j}\,\prob_{j,h_o,\cdot}[T_1=j]\,.
\end{equation*}
Recalling that $\zeta_o:=(1-\epsilon/4)\mu_o$ and noting that
$\delta_o\le\epsilon\mu_o/8$ by definition, Corollary \ref{cor:PT1}
shows that $\Exd[\Lambda]<+\infty$ since $\zeta_o+\delta_o\le
(1-\epsilon/8)\mu_o<\mu_o$. In this way, as
\begin{equation*}
\zeta_o\frac{1+\epsilon}{\mu_o+\delta_o}\ge 1+\epsilon\frac{5-2\epsilon}{8+\epsilon}\ge 1+\frac{\epsilon}{3}
\end{equation*}
because again $\delta_o\le\epsilon\mu_o/8$ and $\epsilon\le 1$,
\eqref{eq:max_uncond_1} follows from \eqref{eq:max_uncond_3} thanks to
Birkhoff's ergodic theorem.

\smallskip

\noindent\textit{The limit \eqref{eq:max_uncond_25}.} The correlation
estimates enter the proof of \eqref{eq:max_uncond_25} because in order
to exclude that all the excursions are below a certain threshold it
suffices to take into account only the excursions that are included on
well separated segments of the polymer. The dilution procedure that we
implement is even more drastic than this because it
takes into account only excursions of very specific size.  For this we
start by fixing $h_o>h_c$ and $\epsilon\in(0,1]$ and by putting
  $\delta_o:=\min\{\epsilon\mu_o/8,(h_o-h_c)/2\}$ with
  $\mu_o:=\mu(h_o)>0$ as before, so that
  $U_{h_o,\epsilon}=(h_o-\delta_o,h_o+\delta_o)$. Then, planning to
  deal with $r_n$ excursions of size $m_n$ and spacing $d_n$, we
  define
  \begin{equation}
    \label{eq:mdr_n}
    \begin{cases}
      m_n:=\big\lfloor \frac{1-\epsilon}{\mu_o-\delta_o}\log n\big\rfloor+1 \\[0.3em]
      d_n:=\lfloor (\log n)^2\rfloor+1 \\[0.3em]
      r_n:=\big\lfloor \frac{n}{m_n+d_n}\big\rfloor
  \end{cases}\,.
  \end{equation}  Set $l_s:=s(m_n+d_n)$ for
  $s\in\{0,\ldots,r_n-1\}$. Bearing in mind that
  $\mu(h)\ge\mu_o-\delta_o$ for $h\in(h_o-\delta_o,h_o+\delta_o)$ by
  the Lipschitz property of $\mu$, for every $n\in\N$, $h\in
  U_{h_o,\epsilon}$, and $\omega\in\Omega$ we have
\begin{align}
  \nonumber
  \prob_{n,h,\omega}\bigg[\frac{M_n}{\log n}< \frac{1-\epsilon}{\mu(h)}\bigg]&\le
  \prob_{n,h,\omega}\bigg[\frac{M_n}{\log n}\le \frac{1-\epsilon}{\mu_o-\delta_o}\bigg]\\
  \nonumber
  &=\prob_{n,h,\omega}[M_n<m_n]\\
  \nonumber
  &\le \Ex_{n,h,\omega}\Bigg[\prod_{s=0}^{r_n-1}\bigg\{1-X_{l_s}\prod_{k=l_s+1}^{l_s+m_n-1}(1-X_k)X_{l_s+m_n}\bigg\}\Bigg]\, ,
\end{align}
since the condition $M_n< m_n$ implies that an excursion from the
position $l_s$ to the position $l_s+m_n$ cannot occur for any choice
of $s$. This bound is the essence of the dilution procedure. Repeated
applications of the first part of Lemma \ref{lem:corr} give
\begin{align}
  \nonumber
  \prob_{n,h,\omega}\bigg[\frac{M_n}{\log n}<\frac{1-\epsilon}{\mu(h)}\bigg]
  &\le \prod_{s=0}^{r_n-1}\Ex_{n,h,\omega}\bigg[1-X_{l_s}\prod_{k=l_s+1}^{l_s+m_n-1}(1-X_k)X_{l_s+m_n}\bigg]\\
  &\quad+2\sum_{s=0}^{r_n-2}\,\sum_{i=0}^{l_s+m_n-1}\sum_{j=l_{s+1}}^\infty\Ex_{j-i,h,\vartheta^i\omega}^{\otimes 2}\bigg[\prod_{k=1}^{j-i-1}(1-X_kX_k')\bigg]\,.
  \label{eq:max_uncond_7}
\end{align}
We are going to show that for $\pae$
\begin{align}
  \nonumber
  &\adjustlimits\limsup_{n\uparrow\infty}\sup_{h\in U_{h_o,\epsilon}}
  \prod_{s=0}^{r_n-1}\Ex_{n,h,\omega}\bigg[1-X_{l_s}\prod_{k=l_s+1}^{l_s+m_n-1}(1-X_k)X_{l_s+m_n}\bigg]\\
&\qquad\le \lim_{n\uparrow\infty}\prod_{s=0}^{r_n-1}\Bigg\{1-\inf_{h\in U_{h_o,\epsilon}}\Ex_{n,h,\omega}\bigg[X_{l_s}\prod_{k=l_s+1}^{l_s+m_n-1}(1-X_k)X_{l_s+m_n}\bigg]\Bigg\}=0
\label{eq:max_uncond_8}
\end{align}
and
\begin{equation}
  \adjustlimits\limspace_{n\uparrow\infty}\sup_{h\in U_{h_o,\epsilon}}
  \sum_{s=0}^{r_n-2}\,\sum_{i=0}^{l_s+m_n-1}\sum_{j=l_{s+1}}^\infty\Ex_{j-i,h,\vartheta^i\omega}^{\otimes 2}\bigg[\prod_{k=1}^{j-i-1}(1-X_kX_k')\bigg]=0\,.
\label{eq:max_uncond_9}
\end{equation}
These limits demonstrate (\ref{eq:max_uncond_25}) via (\ref{eq:max_uncond_7}).

\smallskip

Let us prove (\ref{eq:max_uncond_9}) first, which is the simpler
limit among the two.  As $h_o-\delta_o>h_c$ because
$\delta_o\le(h_o-h_c)/2$ by definition, according to Lemma
\ref{lem:decay} there exist constants $\gamma>0$ and $G>0$ such that
\begin{equation*}
\Exd\Bigg[\sup_{h\in[h_o-\delta_o,+\infty)}\Ex_{j,h,\cdot}^{\otimes 2}\bigg[\prod_{k=1}^{j-1}(1-X_kX_k')\bigg]\Bigg]\le G\,\ee^{-\gamma j}
\end{equation*}
for all $j\in\N$. In this way, introducing the random variable
\begin{equation*}
  \Lambda:=\sum_{j=1}^\infty \ee^{\frac{\gamma}{2}j}\sup_{h\in[h_o-\delta_o,+\infty)}\Ex_{j,h,\cdot}^{\otimes 2}\bigg[\prod_{k=1}^{j-1}(1-X_kX_k')\bigg]\,,
\end{equation*}
we have $\Exd[\Lambda]<+\infty$ and 
\begin{align}
  \nonumber
  &\sup_{h\in U_{h_o,\epsilon}}
  \sum_{s=0}^{r_n-2}\,\sum_{i=0}^{l_s+m_n-1}\sum_{j=l_{s+1}}^\infty\Ex_{j-i,h,\vartheta^i\omega}^{\otimes 2}\bigg[\prod_{k=1}^{j-i-1}(1-X_kX_k')\bigg]\\
  \nonumber
  &\qquad\le
  r_n\sum_{i=0}^{n-1}\sum_{j=d_n}^\infty \sup_{h\in[h_o-\delta_o,+\infty)}\Ex_{j,h,\vartheta^i\omega}^{\otimes 2}\bigg[\prod_{k=1}^{j-1}(1-X_kX_k')\bigg]
    \le  r_n\ee^{-\frac{\gamma}{2}d_n}\sum_{i=0}^{n-1}\Lambda(\vartheta^i\omega)
\end{align}
for all $\omega\in\Omega$. Thus, (\ref{eq:max_uncond_9}) is due to
Birkhoff's ergodic theorem since $r_n\le n$ and $d_n\ge (\log n)^2$
according to \eqref{eq:mdr_n}.

\smallskip

Regarding (\ref{eq:max_uncond_8}), we use Lemma \ref{lem:fact} to
write down for every $n\in\N$, $h\in\Rl$, $\omega\in\Omega$, and
$s\in\{0,\ldots,r_n-1\}$ the identity
\begin{align}
  \nonumber
  &\Ex_{n,h,\omega}\bigg[X_{l_s}\prod_{k=l_s+1}^{l_s+m_n-1}(1-X_k)X_{l_s+m_n}\bigg]\\
  &\qquad=\Ex_{n,h,\omega}[X_{l_s}]\,\prob_{m_n,h,\vartheta^{l_s}\omega}[T_1=m_n]\,\Ex_{n-l_s,h,\vartheta^{l_s}\omega}[X_{m_n}]\,.
\label{eq:max_uncond_20}
\end{align}
The second part of Lemma \ref{lem:utile_per_tutto} entails for $0<a<n$
and $\lambda>0$ 
 \begin{equation*}
  \Ex_{n,h,\omega}\Big[\mathds{1}_{\{T_{L_a+1}\le \lambda\}}(1-X_a)\Big]
  \le \xi\ee^{-h-\omega_a}\lambda^\xi\,\Ex_{n,h,\omega}\Big[\mathds{1}_{\{T_{L_a}+T_{L_a+1}\le \lambda\}}X_a\Big]
\end{equation*}
with $\xi$ given in \eqref{eq:xi_def}, namely
 \begin{equation*}
  \prob_{n,h,\omega}[T_{L_a+1}\le \lambda]
  \le  \Ex_{n,h,\omega}\Big[\mathds{1}_{\{T_{L_a+1}\le \lambda\}}X_a\Big]+\xi\ee^{-h-\omega_a}\lambda^\xi\,\Ex_{n,h,\omega}\Big[\mathds{1}_{\{T_{L_a}+T_{L_a+1}\le \lambda\}}X_a\Big]\,.
 \end{equation*}
 This shows that
\begin{equation*}
  \prob_{n,h,\omega}[M_n\le \lambda]\le\prob_{n,h,\omega}[T_{L_a+1}\le \lambda]\le (1+\xi\ee^{-h-\omega_a}\lambda^\xi)\Ex_{n,h,\omega}[X_a]\,,
\end{equation*}
which trivially holds even for $a=0$ and $a=n$. The choice $a=l_s$ and
$\lambda=\{(1+\epsilon)/\mu(h)\}\log n$ yields
 \begin{equation*}
   \Ex_{n,h,\omega}[X_{l_s}]\ge\frac{\prob_{n,h,\omega}\big[\frac{M_n}{\log n}\le
       \frac{1+\epsilon}{\mu(h)}\big]}{1+\xi\ee^{-h-\omega_{l_s}}\big[\frac{1+\epsilon}{\mu(h)}\big]^\xi (\log n)^\xi}
   \ge \mathds{1}_{\{\omega_{l_s}\ge 0\}}\frac{1-\prob_{n,h,\omega}\big[\frac{M_n}{\log n}>
       \frac{1+\epsilon}{\mu(h)}\big]}{1+\xi\ee^{-h}\big[\frac{1+\epsilon}{\mu(h)}\big]^\xi (\log n)^\xi}\,.
 \end{equation*}
Lemma \ref{lem:utile_per_tutto}, through Remark
\ref{remark:enforce_contact}, also shows that
\begin{equation*}
  \Ex_{n-l_s,h,\vartheta^{l_s}\omega}[X_{m_n}]\ge \frac{1}{1+\xi \ee^{-h-\omega_{l_s+m_n}}m_n^\xi}
  \ge \mathds{1}_{\{\omega_{l_s+m_n}\ge 0\}}\frac{1}{1+\xi \ee^{-h}m_n^\xi}\,.
\end{equation*} 
Thus, having a glance at (\ref{eq:max_uncond_100}) and recalling that
$\mu(h)\ge\mu_o-\delta_o$ for $h\in(h_o-\delta_o,h_o+\delta_o)$, from
(\ref{eq:max_uncond_20}) we get
\begin{align}
  \nonumber
  &\inf_{h\in U_{h_o,\epsilon}}\Ex_{n,h,\omega}\bigg[X_{l_s}\prod_{k=l_s+1}^{l_s+m_n-1}(1-X_k)X_{l_s+m_n}\bigg]\\
  \nonumber
  &\qquad\ge  2\zeta_n\mathds{1}_{\{\omega_{l_s}\ge 0,\,\omega_{l_s+m_n}\ge 0\}}
  \Bigg(1-\sup_{h\in U_{h_o,\epsilon}}\prob_{n,h,\omega}\bigg[\frac{M_n}{\log n}>\frac{1+\epsilon}{\mu(h)}\bigg]\Bigg)\\
  &\qquad\qquad\qquad\qquad\qquad\qquad\qquad\qquad\qquad\quad\times\ee^{-\delta_om_n}\,\prob_{m_n,h_o,\vartheta^{l_s}\omega}[T_1=m_n]\,,
  \label{eq:max_uncond_23}
\end{align}
where we have defined the numerical factor
\begin{equation}
\label{eq:zeta_n}
  \zeta_n:=\frac{1}{2}\,\frac{1}{1+\xi\ee^{-h_o+\delta_o}\big(\frac{1+\epsilon}{\mu_o-\delta_o}\big)^\xi (\log n)^\xi}\,\frac{1}{1+\xi \ee^{-h_o+\delta_o}m_n^\xi}\,.
\end{equation}
Now we remark that (\ref{eq:max_uncond_1}) implies that for $\pae$ the
upper bound
\begin{equation*}
  \sup_{h\in U_{h_o,\epsilon}}\prob_{n,h,\omega}\bigg[\frac{M_n}{\log n}>\frac{1+\epsilon}{\mu(h)}\bigg] \le  \frac{1}{2}
\end{equation*}
is valid for all sufficiently large $n$. In this way,
(\ref{eq:max_uncond_23}) yields for $\pae$
\begin{align}
  \nonumber
  &\limsup_{n\uparrow\infty}
  \prod_{s=0}^{r_n-1}\Bigg\{1-\inf_{h\in U_{h_o,\epsilon}}\Ex_{n,h,\omega}\bigg[X_{l_s}\prod_{k=l_s+1}^{l_s+m_n-1}(1-X_k)X_{l_s+m_n}\bigg]\Bigg\}\\
  &\qquad\le\limsup_{n\uparrow\infty}
  \prod_{s=0}^{r_n-1}\bigg(1-\zeta_n\mathds{1}_{\{\omega_{l_s}\ge 0,\,\omega_{l_s+m_n}\ge 0\}}\,\ee^{-\delta_om_n}\,\prob_{m_n,h_o,\vartheta^{l_s}\omega}[T_1=m_n]\bigg)\,.
  \label{eq:max_uncond_123}
\end{align}
In order to demonstrate (\ref{eq:max_uncond_8}) we must prove that the
right-hand side of (\ref{eq:max_uncond_123}) is 0 for $\pae$. This is
due to the Borel--Cantelli Lemma since we are going to show that
$\sum_{n\in\N}E_n<+\infty$,
\begin{equation*}
 E_n:=\int_\Omega\prod_{s=0}^{r_n-1}\bigg(1-\zeta_n\mathds{1}_{\{\omega_{l_s}\ge 0,\,\omega_{l_s+m_n}\ge 0\}}\,\ee^{-\delta_om_n}\,\prob_{m_n,h_o,\vartheta^{l_s}\omega}[T_1=m_n]\bigg)
 \probd[\dd\omega]
\end{equation*}
being the expectation of the right-hand side of
(\ref{eq:max_uncond_123}). Note that
$c:=\int_\Omega\mathds{1}_{\{\omega_0\ge 0\}}\probd[\dd\omega]>0$
since $\int_\Omega \omega_0\,\probd[\dd\omega]=0$ and that
$\Exd[\prob_{m_n,h_o,\cdot}[T_1=m_n]]\ge\ee^{-(\mu_o+\delta_o)(m_n-1)+\delta_o}$
for all sufficiently large $n$ by Corollary \ref{cor:PT1}. Also note
that all the factors in the right-hand side of
(\ref{eq:max_uncond_123}) are statistically independent as
$\prob_{m_n,h_o,\vartheta^{l_s}\omega}[T_1=m_n]$ depends on
$\omega:=\{\omega_a\}_{a\in\N_0}$ only through the components
$\omega_{l_s+1},\ldots,\omega_{l_s+m_n-1}$.  Thus, carrying out the
integration and recalling that $m_n:=\lfloor
\{(1-\epsilon)/(\mu_o-\delta_o)\}\log n\rfloor+1$, for all sufficiently
large $n$ we find
\begin{align}
  \nonumber
  E_n&=\bigg(1-c^2\zeta_n\ee^{-\delta_om_n}\Exd\Big[\prob_{m_n,h_o,\cdot}[T_1=m_n]\Big]\bigg)^{\!r_n}\\
  \nonumber
  &\le \bigg(1-c^2\zeta_n\ee^{-(\mu_o+2\delta_o)(m_n-1)}\bigg)^{\!r_n}\le \exp\bigg\{\!-c^2\zeta_nr_n\ee^{-(\mu_o+2\delta_o)(m_n-1)}\bigg\}\\
  \nonumber
  &\le \exp\Bigg\{\!-c^2\frac{\zeta_nr_n}{n^{(\mu_o+2\delta_o)\big(\frac{1-\epsilon}{\mu_o-\delta_o}\big)}}\Bigg\}\,.
\end{align}
On the other hand, using that $\delta_o\le\epsilon\mu_o/8$ by
definition and that $\epsilon\le 1$ we deduce that
\begin{equation*}
  (\mu_o+2\delta_o)\frac{1-\epsilon}{\mu_o-\delta_o}
  \le 1-\epsilon\frac{5+2\epsilon}{8-\epsilon}\le1-\frac{\epsilon}{2}\,.
\end{equation*}
In conclusion, for all sufficiently large $n$ we have
\begin{equation*}
E_n\le \exp\bigg\{\!-c^2\frac{\zeta_nr_n}{n^{1-\frac{\epsilon}{2}}}\bigg\}\,,
\end{equation*}
which proves that $\sum_{n\in\N}E_n<+\infty$ because, bearing
\eqref{eq:mdr_n} and \eqref{eq:zeta_n} in mind, we see that
$\lim_{n\uparrow\infty}(\log n)^{2\xi+2}\zeta_nr_n/n>0$.
\end{proof}

\section{Regularity estimates}
\label{sec:regularity_estimates}

In this section we prove Theorem \ref{th:Cinfty} about the
differentiability of the free energy and, as a consequence, the
quenched concentration bound and the quenched CLT stated by Theorem
\ref{th:CLT+concentration} for the contact number $L_n$. We start with
a preliminary result about its variance.  In the introduction we
observed that
$\liminf_{n\uparrow\infty}\Exd[\Ex_{n,h,\cdot}[L_n]]/n>0$ for $h>h_c$.
The following lemma, which is inspired by \cite[Theorem
  B.1]{giacomin2020}, states that the same holds for
$\var_{n,h,\cdot}[L_n]$.

\begin{lemma}
   \label{lem:var_positive}
   For every $h>h_c$
   \begin{equation*}
  \liminf_{n\uparrow\infty}\frac{\Exd[\var_{n,h,\cdot}[L_n]]}{n}>0\,.
\end{equation*} 
 \end{lemma}

 \begin{proof}[Proof of Lemma \ref{lem:var_positive}]
   Setting $\#_n:=\lfloor(n-1)/2\rfloor$ for brevity, we shall show
   that for every $n\in\N$, $h\in\Rl$, and
   $\omega:=\{\omega_a\}_{a\in\N_0}\in\Omega$
\begin{equation}
  \label{lem:var_positive_0}
  \var_{n,h,\omega}[L_n]\ge \sum_{k=1}^{\#_n}\frac{\ee^{h+\omega_{2k}}p(1)^2p(2)}{[p(2)+\ee^{h+\omega_{2k}}p(1)^2]^2}\,\Ex_{n,h,\omega}[X_{2k-1}X_{2k+1}]
\end{equation}
and that, for $k\in\{1,\ldots,\#_n\}$,
\begin{align}
  \nonumber
  &\Ex_{n,h,\omega}[X_{2k-1}X_{2k+1}]\\
  &\qquad\ge \frac{1}{2}\frac{1}{1+\xi 2^\xi \ee^{-h-\omega_{2k-1}}}\frac{1}{1+\xi\ee^{-h-\omega_{2k+1}}}\Big(\Ex_{n,h,\omega}[X_{2k}]+\Ex_{n,h,\omega}[X_{2k+1}]\Big)\,.
  \label{lower_bound_XX}
\end{align}
These two lower bounds prove the lemma as follows. Pick $h>h_c$ and
recall that
$\delta:=\liminf_{n\uparrow\infty}\Exd[\Ex_{n,h,\cdot}[L_n]]/n>0$. Let
$\lambda>0$ be so small that $\probd[\Lambda\le\lambda]\le \delta/2$,
$\Lambda$ being the random variable that maps
$\omega:=\{\omega_a\}_{a\in\N_0}\in\Omega$ to
\begin{equation*}
\Lambda(\omega):=\frac{1}{2}\frac{\ee^{h+\omega_1}p(1)^2p(2)}{[p(2)+\ee^{h+\omega_1}p(1)^2]^2}\frac{1}{1+\xi 2^\xi \ee^{-h-\omega_0}}\frac{1}{1+\xi\ee^{-h-\omega_2}}\,.
\end{equation*}
Combining \eqref{lem:var_positive_0} with \eqref{lower_bound_XX} we
see that for all $n\in\N$ and $\omega\in\Omega$
\begin{align}
  \nonumber
  \var_{n,h,\omega}[L_n]&\ge \sum_{k=1}^{\#_n}\Lambda(\vartheta^{2k-1}\omega)\Big(\Ex_{n,h,\omega}[X_{2k}]+\Ex_{n,h,\omega}[X_{2k+1}]\Big)\\
  \nonumber
  &\ge\lambda\sum_{k=1}^{\#_n}\mathds{1}_{\{\Lambda(\vartheta^{2k-1}\omega)>\lambda\}}\Big(\Ex_{n,h,\omega}[X_{2k}]+\Ex_{n,h,\omega}[X_{2k+1}]\Big)\\
  \nonumber
  &\ge \lambda\sum_{k=1}^{\#_n}\Big(\Ex_{n,h,\omega}[X_{2k}]+\Ex_{n,h,\omega}[X_{2k+1}]\Big)-2\lambda\sum_{k=1}^{\#_n}\mathds{1}_{\{\Lambda(\vartheta^{2k-1}\omega)\le\lambda\}}\\
  \nonumber
  &\ge \lambda\,\Ex_{n,h,\omega}[L_n]-2\lambda-2\lambda\sum_{k=1}^{\#_n}\mathds{1}_{\{\Lambda(\vartheta^{2k-1}\omega)\le\lambda\}}\,.
\end{align}
Thus, integrating with respect to $\probd[\dd\omega]$, dividing by
$n$, and then sending $n$ to infinity we get
\begin{equation*}
  \liminf_{n\uparrow\infty}\frac{\Exd[\var_{n,h,\cdot}[L_n]]}{n}\ge \lambda \delta-\lambda\probd[\Lambda\le\lambda]\ge\frac{\lambda \delta}{2}>0\,.
\end{equation*}

\smallskip

\noindent\textit{The bound \eqref{lem:var_positive_0}.}  Fix $n\in\N$,
$h\in\Rl$, and $\omega:=\{\omega_a\}_{a\in\N_0}\in\Omega$. Assume that
$n\ge 3$, otherwise \eqref{lem:var_positive_0} is trivial. Denoting by
$\mathfrak{O}$ the $\sigma$-algebra generated by $X_1,X_3,X_5,\ldots$
and recalling that $X_n=1$ almost surely under the polymer measure, we
can state that
\begin{align}
  \nonumber
  \var_{n,h,\omega}[L_n]&=\Ex_{n,h,\omega}\big[L_n^2\big]-\Big(\Ex_{n,h,\omega}\big[\Ex_{n,h,\omega}[L_n|\mathfrak{O}]\big]\Big)^{\!2}\\
\nonumber
  &\ge\Ex_{n,h,\omega}\big[L_n^2\big]-\Ex_{n,h,\omega}\big[\Ex_{n,h,\omega}[L_n|\mathfrak{O}]^2\big]\\
  \nonumber
  &=\Ex_{n,h,\omega}\Big[\big\{L_n-\Ex_{n,h,\omega}[L_n|\mathfrak{O}]\big\}^{\!2}\Big]\\
  \nonumber
  &=\Ex_{n,h,\omega}\Bigg[\bigg\{\sum_{k=1}^{\#_n}\big(X_{2k}-\Ex_{n,h,\omega}[X_{2k}|\mathfrak{O}]\big)\bigg\}^{\!2}\Bigg]\,,
\end{align}
so that it suffices to demonstrate that
\begin{align}
  \nonumber
  &\Ex_{n,h,\omega}\Bigg[\bigg\{\sum_{k=1}^{\#_n}\big(X_{2k}-\Ex_{n,h,\omega}[X_{2k}|\mathfrak{O}]\big)\bigg\}^{\!2}\Bigg|\mathfrak{O}\Bigg]\\
  &\qquad\ge\sum_{k=1}^{\#_n}\frac{\ee^{h+\omega_{2k}}p(1)^2p(2)}{[p(2)+\ee^{h+\omega_{2k}}p(1)^2]^2}\,X_{2k-1}X_{2k+1}\,.
   \label{eq:var_O_1}
\end{align}
Bound \eqref{eq:var_O_1} is trivial if at most one variable among
$X_1,X_3,\ldots,X_{2\#_n+1}$ takes value 1 since the right-hand side
is $0$ in this case. Then, suppose that at least two of these
variables take value 1 and let $ 2l_1+1<\cdots<2l_r+1$ with $r\ge 2$
be the sites where the values 1 are attained. In this way, bound
\eqref{eq:var_O_1} is tantamount to
\begin{align}
  \nonumber
  &\Ex_{n,h,\omega}\Bigg[\bigg\{\sum_{k=1}^{\#_n}\big(X_{2k}-\Ex_{n,h,\omega}[X_{2k}|\mathfrak{O}]\big)\bigg\}^{\!2}\Bigg|\mathfrak{O}\Bigg]\\
  &\qquad\ge\sum_{s=2}^r\mathds{1}_{\{l_s=l_{s-1}+1\}}\frac{\ee^{h+\omega_{2l_s}}p(1)^2p(2)}{[p(2)+\ee^{h+\omega_{2l_s}}p(1)^2]^2}\,.
   \label{eq:var_O_2}
\end{align}
Let us verify bound \eqref{eq:var_O_2}.  To this aim, we observe that
conditioning makes the sums
$\sum_{k=1}^{l_1}X_{2k},\sum_{k=l_1+1}^{l_2}X_{2k},\ldots,\sum_{k=l_{r-1}+1}^{l_r}X_{2k},\sum_{k=l_r+1}^{\#_n}X_{2k}$
independent thanks to Lemma \ref{lem:fact}, so that
\begin{align}
  \nonumber
  &\Ex_{n,h,\omega}\Bigg[\bigg\{\sum_{k=1}^{\#_n}\big(X_{2k}-\Ex_{n,h,\omega}[X_{2k}|\mathfrak{O}]\big)\bigg\}^{\!2}\Bigg|\mathfrak{O}\Bigg]\\
  \nonumber
  &\qquad=\Ex_{n,h,\omega}\Bigg[\bigg\{\sum_{k=1}^{l_1}\big(X_{2k}-\Ex_{n,h,\omega}[X_{2k}|\mathfrak{O}]\big)\bigg\}^{\!2}\Bigg|\mathfrak{O}\Bigg]\\
  \nonumber
  &\qquad\quad+\sum_{s=2}^r\Ex_{n,h,\omega}\Bigg[\bigg\{\sum_{k=l_{s-1}+1}^{l_s}\big(X_{2k}-\Ex_{n,h,\omega}[X_{2k}|\mathfrak{O}]\big)\bigg\}^{\!2}\Bigg|\mathfrak{O}\Bigg]\\
  \nonumber
  &\qquad\quad\quad+\Ex_{n,h,\omega}\Bigg[\bigg\{\sum_{k=l_r+1}^{\#_n}\big(X_{2k}-\Ex_{n,h,\omega}[X_{2k}|\mathfrak{O}]\big)\bigg\}^{\!2}\Bigg|\mathfrak{O}\Bigg]\\
  \nonumber
  &\qquad\ge\sum_{s=2}^r\mathds{1}_{\{l_s=l_{s-1}+1\}}\Ex_{n,h,\omega}\Big[\big\{X_{2l_s}-\Ex_{n,h,\omega}[X_{2l_s}|\mathfrak{O}]\big\}^{\!2}\Big|\mathfrak{O}\Big]\\
  &\qquad=\sum_{s=2}^r\mathds{1}_{\{l_s=l_{s-1}+1\}}\Ex_{n,h,\omega}[X_{2l_s}|\mathfrak{O}]\Big(1-\Ex_{n,h,\omega}[X_{2l_s}|\mathfrak{O}]\Big)\,.
  \label{eq:var_O_3}
\end{align}
On the other hand, when $l_s=l_{s-1}+1$ Lemma \ref{lem:fact} also
shows that
\begin{equation}
  \Ex_{n,h,\omega}[X_{2l_s}|\mathfrak{O}]=\Ex_{2l_s+1,h,\omega}[X_{2l_s}|\mathfrak{O}]=\Ex_{2,h,\vartheta^{2l_s-1}\omega}[X_1]
  =\frac{p(1)^2\ee^{h+\omega_{2l_s}}}{p(2)+p(1)^2\ee^{h+\omega_{2l_s}}}\,.
\label{eq:var_O_4}
\end{equation}
We obtain \eqref{eq:var_O_2} by combining \eqref{eq:var_O_3} with
\eqref{eq:var_O_4}.

\smallskip

\noindent\textit{The bound \eqref{lower_bound_XX}.}  Lemma
\ref{lem:fact} shows that for $k\in\{1,\ldots,\#_n\}$
\begin{align}
  \nonumber
  \Ex_{n,h,\omega}[X_{2k-1}X_{2k+1}]&=\Ex_{n,h,\omega}[X_{2k-1}|X_{2k+1}=1]\Ex_{n,h,\omega}[X_{2k+1}]\\
  \nonumber
  &=\Ex_{2k+1,h,\omega}[X_{2k-1}]\Ex_{n,h,\omega}[X_{2k+1}]
\end{align}
and
\begin{align}
  \nonumber
  \Ex_{n,h,\omega}[X_{2k+1}]\ge \Ex_{n,h,\omega}[X_{2k}X_{2k+1}]&=\Ex_{n,h,\omega}[X_{2k+1}|X_{2k}=1]\Ex_{n,h,\omega}[X_{2k}]\\
  \nonumber
  &=\Ex_{n-2k,h,\vartheta^{2k}\omega}[X_1]\Ex_{n,h,\omega}[X_{2k}]\,,
\end{align}
so we can write down the inequality
\begin{align}
  \nonumber
  \Ex_{n,h,\omega}[X_{2k-1}X_{2k+1}]&\ge \frac{1}{2}\Ex_{2k+1,h,\omega}[X_{2k-1}]\Ex_{n,h,\omega}[X_{2k+1}]\\
  \nonumber
  &\quad +\frac{1}{2}\Ex_{2k+1,h,\omega}[X_{2k-1}]\Ex_{n-2k,h,\vartheta^{2k}\omega}[X_1]\Ex_{n,h,\omega}[X_{2k}]\\
  \nonumber
  &\ge \frac{1}{2}\Ex_{2k+1,h,\omega}[X_{2k-1}]\Ex_{n-2k,h,\vartheta^{2k}\omega}[X_1]\Big(\Ex_{n,h,\omega}[X_{2k}]+\Ex_{n,h,\omega}[X_{2k+1}]\Big)\,.
\end{align}
This entails the lower bound \eqref{lower_bound_XX} because Lemma
\ref{lem:utile_per_tutto}, through Remark
\ref{remark:enforce_contact}, gives
\begin{equation*}
 \Ex_{2k+1,h,\omega}[X_{2k-1}]\ge \frac{1}{1+\xi 2^\xi \ee^{-h-\omega_{2k-1}}}
\end{equation*}
and
\begin{equation*}
  \Ex_{n-2k,h,\vartheta^{2k}\omega}[X_1]\ge \frac{1}{1+\xi\ee^{-h-\omega_{2k+1}}}\,.
  \qedhere
\end{equation*}
\end{proof}

The first step towards Theorem \ref{th:Cinfty} is to bound the
derivatives of $\log Z_{n,h}(\omega)$ with respect to $h$ by means of
a Birkhoff sum. The following result is due to Lemma \ref{lem:corr}.

\begin{lemma}
  \label{Birkhoff}
  For every $h\in\Rl$, $\omega\in\Omega$, and integers $n\ge 1$ and
  $r\ge 2$
\begin{equation*}
    \Big|\partial^r_h\log Z_{n,h}(\omega)\Big|\le 2^r(r!)^2\sum_{i=0}^{n-1}\sum_{j\in\N}j^r\Ex_{j+1,h,\vartheta^i\omega}^{\otimes 2}\bigg[\prod_{k=1}^j(1-X_kX_k')\bigg]\,.
\end{equation*}
\end{lemma}

Before going into the proof of Lemma \ref{Birkhoff} we recall a few
elementary facts from the theory of cumulants. The cumulant of order
$r\in\N$ of the contact number $L_n$ with respect to the polymer
measure $\prob_{n,h,\omega}$ is the number
${}^{r\!}K_{n,h,\omega}:=\partial_\lambda^r\log\Ex_{n,h,\omega}[\ee^{\lambda
    L_n}]|_{\lambda=0}$. It can be expanded as
  \begin{equation}
  {}^{r\!}K_{n,h,\omega}=\sum_{a_1=1}^n\cdots\sum_{a_r=1}^nU_{n,h,\omega}(a_1,\ldots,a_r)\,,
\label{eq:expansion_cumulant}
  \end{equation}
  where
  $U_{n,h,\omega}(a_1,\ldots,a_r):=\partial_{\lambda_{a_1}}\!\!\cdots\,\partial_{\lambda_{a_r}}\Ex_{n,h,\omega}[\ee^{\sum_{a=1}^n\lambda_aX_a}]|_{\lambda_1=\cdots=\lambda_n=0}$
  is the joint cumulant of $X_{a_1},\ldots,X_{a_r}$, also known in
  statistical mechanics with the name of \textit{Ursell function}
  (see, e.g., \cite{sylvester1975,shlosman1986,cf:taming95}).  The
  joint cumulant can be expressed as an alternate sum of products of
  mixed moments as (see \cite[Proposition 3.2.1]{peccati2011})
\begin{equation*}
U_{n,h,\omega}(a_1,\ldots,a_r)\,=\sum_{l=1}^r\sum_{\{I_1,\ldots,I_l\}\in\mathcal{P}_l^r} (-1)^{l-1}(l-1)! 
\prod_{j=1}^l\Ex_{n,h,\omega}\Big[ \prod_{k \in I_j} X_{a_k}\Big]\,,
\end{equation*}
where $\mathcal{P}_l^r$ is the collection of all the partitions of the
set $\{1,\ldots,r\}$ into $l$ components. In particular one finds
\begin{equation*}
  U_{n,h,\omega}(a_1,a_2)=\Ex_{n,h,\omega}[X_{a_1}X_{a_2}]-\Ex_{n,h,\omega}[X_{a_1}]\Ex_{n,h,\omega}[X_{a_2}]=\cov_{n,h,\omega}[X_{a_1},X_{a_2}]\,.
\end{equation*}

Importantly, cumulants are related to the derivatives of the
finite-volume free energy. In fact, since
$\Ex_{n,h,\omega}[\ee^{\lambda
    L_n}]=Z_{n,h+\lambda}(\omega)/Z_{n,h}(\omega)$ for all
$\lambda\in\Rl$, we see that ${}^{r\!}K_{n,h,\omega}=\partial^r_h\log
Z_{n,h}(\omega)$ for any $r\in\N$. Similarly, for every $r\in\N$ and
$a_1,\ldots,a_r\in\{1,\ldots,n\}$ we have
$U_{n,h,\omega}(a_1,\ldots,a_r)=\partial_{\omega_{a_1}}\!\!\cdots\,\partial_{\omega_{a_r}}\log
Z_{n,h}(\omega)$.

\begin{proof}[Proof of Lemma \ref{Birkhoff}]
Fix $h\in\Rl$, $\omega\in\Omega$, and integers $n\ge 1$ and $r\ge 2$.
Since the derivative $\partial^r_h\log Z_{n,h}(\omega)$ coincides with
the cumulant ${}^{r\!}K_{n,h,\omega}$ of $L_n$, formula
\eqref{eq:expansion_cumulant} turns out to be an expansion for
$\partial^r_h\log Z_{n,h}(\omega)$. Then, the symmetry of
$U_{n,h,\omega}(a_1,\ldots,a_r)$ with respect to the arguments
$a_1,\ldots,a_r$ allows us to state that
\begin{equation}
    \Big|\partial^r_h\log Z_{n,h}(\omega)\Big|\le r!\sum_{1\le a_1\le\cdots\le a_r\le n}\Big|U_{n,h,\omega}(a_1,\ldots,a_r)\Big|\,.
    \label{Birkhoff_0}
\end{equation}
We shall verify later that, given an integer $s\in\{1,\ldots,r-1\}$,
the cumulant $U_{n,h,\omega}(a_1,\ldots,a_r)$ with $1\le
a_1\le\cdots\le a_r\le n$ can be written by adding and subtracting
less than $2^{r-1}(r-1)!$ terms of the form
\begin{equation*}
 \cov_{n,h,\omega}\bigg[\prod_{k\in I}X_{a_k},\prod_{k\in J}X_{a_k}\bigg]
  \prod_{j=1}^l\Ex_{n,h,\omega}\bigg[\prod_{k\in R_j}X_{a_k}\bigg]
\end{equation*}
with sets $I\subseteq\{1,\ldots,s\}$, $J\subseteq\{s+1,\ldots,r\}$,
and $R_1,\ldots,R_l$ constituting a partition of
$\{1,\ldots,r\}\setminus I\cup J$. In this way, as by Lemma
\ref{lem:corr}
\begin{equation*}
  \Bigg|\cov_{n,h,\omega}\bigg[\prod_{k\in I}X_{a_k},\prod_{k\in J}X_{a_k}\bigg]\Bigg|
  \le 2\sum_{i=0}^{a_s-1}\sum_{j=a_{s+1}+1}^\infty \Ex_{j-i,h,\vartheta^i\omega}^{\otimes 2}\bigg[\prod_{k=1}^{j-i-1}(1-X_kX_k')\bigg]\,,
\end{equation*}
we deduce that
\begin{align}
  \nonumber
  &\Big|U_{n,h,\omega}(a_1,\ldots,a_r)\Big|\\
  \nonumber
  &\qquad \le 2^r(r-1)!\sum_{i=0}^{a_s-1}\sum_{j=a_{s+1}+1}^\infty \Ex_{j-i,h,\vartheta^i\omega}^{\otimes 2}\bigg[\prod_{k=1}^{j-i-1}(1-X_kX_k')\bigg]\\
  \nonumber
  &\qquad \le 2^r(r-1)!\sum_{i=0}^{n-1}\sum_{j=i+2}^\infty \mathds{1}_{\{i<a_s\le j-1,\,a_{s+1}-a_s<j-i-1\}}\Ex_{j-i,h,\vartheta^i\omega}^{\otimes 2}\bigg[\prod_{k=1}^{j-i-1}(1-X_kX_k')\bigg]\\
  \nonumber
  &\qquad =2^r(r-1)!\sum_{i=0}^{n-1}\sum_{j\in\N} \mathds{1}_{\{i<a_s\le j+i,\,a_{s+1}-a_s<j\}}\Ex_{j+1,h,\vartheta^i\omega}^{\otimes 2}\bigg[\prod_{k=1}^j(1-X_kX_k')\bigg]\,.
\end{align}
The arbitrariness of $s$ gives for $1\le a_1\le\cdots\le a_r\le n$ the
improved bound
\begin{align} \nonumber
  \Big|U_{n,h,\omega}(a_1,\ldots,a_r)\Big|&\le 2^r(r-1)!\sum_{s=1}^{r-1}\prod_{t=1}^{r-1}\mathds{1}_{\{a_{t+1}-a_t\le a_{s+1}-a_s\}}\\
  &\quad\times\sum_{i=0}^{n-1}\sum_{j\in\N}\mathds{1}_{\{i<a_s\le j+i,\,a_{s+1}-a_s<j\}}\Ex_{j+1,h,\vartheta^i\omega}^{\otimes 2}\bigg[\prod_{k=1}^j(1-X_kX_k')\bigg]\,.
\label{eq:per_wh_positive}
\end{align}
The latter in combination with \eqref{Birkhoff_0} proves the lemma
since for every $i$ and $j$
\begin{align}
  \nonumber
&\sum_{1\le a_1\le\cdots\le a_r\le n}
  \sum_{s=1}^{r-1}\mathds{1}_{\{i<a_s\le i+j,\,a_{s+1}-a_s<j\}}\prod_{t=1}^{r-1}\mathds{1}_{\{a_{t+1}-a_t\le a_{s+1}-a_s\}}\\
\nonumber
  &\qquad\qquad\le\sum_{-\infty<a_1\le\cdots\le a_r<\infty}
\sum_{s=1}^{r-1}\mathds{1}_{\{0< a_s\le j\}}\prod_{t=1}^{r-1}\mathds{1}_{\{a_{t+1}-a_t<j\}}=(r-1)j^r\,.
\end{align}

\smallskip

\noindent\textit{The expansion of $U_{n,h,\omega}(a_1,\ldots,a_r)$.}
To conclude, let us demonstrate that for every $r\ge 2$,
$s\in\{1,\ldots,r-1\}$, and integers $1\le a_1\le\cdots\le a_r\le n$
the cumulant $U_{n,h,\omega}(a_1,\ldots,a_r)$ can be written by adding
and subtracting less than $2^{r-1}(r-1)!$ terms of the form
\begin{equation}
  \cov_{n,h,\omega}\bigg[\prod_{k\in I}X_{a_k},\prod_{k\in J}X_{a_k}\bigg]
  \prod_{j=1}^l\Ex_{n,h,\omega}\bigg[\prod_{k\in R_j}X_{a_k}\bigg]
\label{term_in_U_0}
\end{equation}
with sets $I\subseteq\{1,\ldots,s\}$, $J\subseteq\{s+1,\ldots,r\}$,
and $R_1,\ldots,R_l$ constituting a partition of
$\{1,\ldots,r\}\setminus I\cup J$. We also show that the rule
according to which the terms are added and subtracted is independent
of $\omega$. We proceed by induction over $r$. For $r=2$ the claim is
trivially true as
$U_{n,h,\omega}(a_1,a_2)=\cov_{n,h,\omega}[X_{a_1},X_{a_2}]$.  For
$r>2$, fix integers $1\le a_1\le\cdots\le a_r\le n$ and
$s\in\{1,\ldots,r-1\}$ and note that
\begin{align*}
  U_{n,h,\omega}(a_1,\ldots,a_r)&=\partial_{\omega_{a_1}}\!\!\cdots\,\partial_{\omega_{a_r}}\log Z_{n,h}(\omega)\\
  &=\partial_{\omega_{a_r}}U_{n,h,\omega}(a_1,\ldots,a_{r-1})=\partial_{\omega_{a_1}}U_{n,h,\omega}(a_2,\ldots,a_r)\,.
\end{align*}
If $s<r-1$, then by the inductive hypothesis we can write
$U_{n,h,\omega}(a_1,\ldots,a_{r-1})$ by adding and subtracting less
than $2^{r-2}(r-2)!$ terms of the form
\begin{equation}
  \cov_{n,h,\omega}\bigg[\prod_{k\in I}X_{a_k},\prod_{k\in J}X_{a_k}\bigg]
  \prod_{j=1}^l\Ex_{n,h,\omega}\bigg[\prod_{k\in R_j}X_{a_k}\bigg]
  \label{term_in_U}
\end{equation}
with sets $I\subseteq\{1,\ldots,s\}$, $J\subseteq\{s+1,\ldots,r-1\}$,
and $R_1,\ldots,R_l$ constituting a partition of
$\{1,\ldots,r-1\}\setminus I\cup J$. Moreover, the rule according to
which the terms are added and subtracted is independent of $\omega$,
so that $\partial_{\omega_{a_r}}U_{n,h,\omega}(a_1,\ldots,a_{r-1})$
can be computed by combining with the same rule the derivatives of the
terms \eqref{term_in_U} with respect to $\omega_{a_r}$.  Simple
calculations show that each of these derivatives originates
$2l+3<2(r-1)$ terms of the form (\ref{term_in_U_0}), which are added
and subtracted according to a rule that is independent of $\omega$:
\begin{align}
  \nonumber
   &\partial_{\omega_{a_r}}\Bigg\{\cov_{n,h,\omega}\bigg[\prod_{k\in I}X_{a_k},\prod_{k\in J}X_{a_k}\bigg]
   \prod_{j=1}^l\Ex_{n,h,\omega}\bigg[\prod_{k\in R_j}X_{a_k}\bigg]\Bigg\}\\
   \nonumber
   &\qquad=\cov_{n,h,\omega}\bigg[\prod_{k\in I}X_{a_k},\prod_{k\in J\cup\{r\}}X_{a_k}\bigg]
   \prod_{j=1}^l\Ex_{n,h,\omega}\bigg[\prod_{k\in R_j}X_{a_k}\bigg]\\
   \nonumber
    &\qquad\quad-\cov_{n,h,\omega}\bigg[\prod_{k\in I}X_{a_k},X_{a_r}\bigg]\Ex_{n,h,\omega}\bigg[\prod_{k\in J}X_{a_k}\bigg]
   \prod_{j=1}^l\Ex_{n,h,\omega}\bigg[\prod_{k\in R_j}X_{a_k}\bigg]\\
   \nonumber
   &\qquad\quad\quad+\sum_{i=1}^l\cov_{n,h,\omega}\bigg[\prod_{k\in I}X_{a_k},\prod_{k\in J}X_{a_k}\bigg]
  \Ex_{n,h,\omega}\bigg[\prod_{k\in R_i\cup\{r\}}X_{a_k}\bigg] \prod_{\substack{j=1\\j\ne i}}^l\Ex_{n,h,\omega}\bigg[\prod_{k\in R_j}X_{a_k}\bigg]\\
    \nonumber
   &\qquad\quad\quad\quad-(l+1)\,\cov_{n,h,\omega}\bigg[\prod_{k\in I}X_{a_k},\prod_{k\in J}X_{a_k}\bigg]\Ex_{n,h,\omega}[X_r]\prod_{j=1}^l\Ex_{n,h,\omega}\bigg[\prod_{k\in R_j}X_{a_k}\bigg]\,.
\end{align}
This proves that the claim is true for $r>2$ and $s<r-1$. It remains
to discuss the case $s=r-1$.

If $s>1$, then we can write $U_{n,h,\omega}(a_2,\ldots,a_r)$ by adding
and subtracting less than $2^{r-2}(r-2)!$ terms of the form
\begin{equation*}
  \cov_{n,h,\omega}\bigg[\prod_{k\in I}X_{a_k},\prod_{k\in J}X_{a_k}\bigg]
  \prod_{j=1}^l\Ex_{n,h,\omega}\bigg[\prod_{k\in R_j}X_{a_k}\bigg]
\end{equation*}
with sets $I\subseteq\{2,\ldots,s\}$, $J\subseteq\{s+1,\ldots,r\}$,
and $R_1,\ldots,R_l$ constituting a partition of
$\{2,\ldots,r\}\setminus I\cup J$. The terms are added and subtracted
according to a rule that is independent of $\omega$.  Thus, taking the
derivative with respect to $\omega_{a_1}$, a repetition of the above
arguments demonstrates that the claim holds for $r>2$ and $s>1$.
\end{proof}

We use Lemma \ref{Birkhoff} in combination with Lemma \ref{lem:decay}
to prove Theorem \ref{th:Cinfty}. The strict convexity of $f$ will
come from Lemma \ref{lem:var_positive}.

\begin{proof}[Proof of Theorem \ref{th:Cinfty}]
For $s\in\N$, put $h_s:=-s$ if $h_c=-\infty$ or $h_s:=h_c+1/s$ if
$h_c>-\infty$ and note that $h_s>h_c$ and
$\lim_{s\uparrow\infty}h_s=h_c$. By Lemma \ref{lem:decay} there exist
constants $\gamma_s>0$ and $G_s>0$ such that
\begin{equation*}
\Exd\Bigg[\sup_{h\in [h_s,+\infty)}\Ex_{j,h,\cdot}^{\otimes 2}\bigg[\prod_{k=1}^{j-1}(1-X_kX_k')\bigg]\Bigg]\le G_s\ee^{-\gamma_s j}
\end{equation*}
for all $j\in\N$. This shows that the random variables
\begin{equation*}
\Lambda_s:=\frac{\gamma_s}{4}+\sum_{j\in\N}\ee^{\frac{\gamma_s}{2}j}\sup_{h\in [h_s,+\infty)}\Ex_{j+1,h,\cdot}^{\otimes 2}\bigg[\prod_{k=1}^j(1-X_kX_k')\bigg]
\end{equation*}
possess finite expectation. Lemma \ref{Birkhoff}, together with the
inequality $\zeta^r\le r!\ee^\zeta$ for $\zeta\ge 0$, gives for every
integers $n\ge 1$, $r\ge 2$, and $s\ge 1$ and for every $h\ge h_s$ and
$\omega\in\Omega$
\begin{align}
  \nonumber
  \Big|\partial^r_h\log Z_{n,h}(\omega)\Big|&\le 2^r(r!)^2\sum_{i=0}^{n-1}
  \sum_{j\in\N} j^r\Ex_{j+1,h,\vartheta^i\omega}^{\otimes 2}\bigg[\prod_{k=1}^j(1-X_kX_k')\bigg]\\
  &\le\bigg(\frac{4}{\gamma_s}\bigg)^{\!\!r}(r!)^3 \sum_{i=0}^{n-1}\Lambda_s(\vartheta^i\omega)\,.
  \label{eq:Cinfty_1}
\end{align}
This bound holds also for $r=1$, as $|\partial_h\log
Z_{n,h}(\omega)|=\Ex_{n,h,\omega}[L_n]\le n$ on the one hand and
$\Lambda_s\ge \gamma_s/4$ by construction on the other hand, and is
the key for proving the theorem.

\smallskip

\textit{Results for the expectation of the finite-volume free energy.}
Let $f_n^{(r)}$ for $r\in\N_0$ be the function that maps $h\in\Rl$ to
$f_n^{(r)}(h):=\Exd[(1/n)\partial^r_h\log Z_{n,h}]$.  We know that
$\lim_{n\uparrow\infty}f_n^{(0)}(h)=f(h)$ for all $h\in\Rl$ (see
\cite[Theorem 4.6]{giacomin2007}).  Let us show that $f$ is infinitely
differentiable on $(h_c,+\infty)$ and that
$\lim_{n\uparrow\infty}\sup_{h\in H}|f^{(r)}_n(h)-\partial_h^rf(h)|=0$
for any $r\in\N_0$ and compact interval $H\subset(h_c,+\infty)$. The
latter obviously implies that
$\lim_{n\uparrow\infty}f^{(r)}_n(h)=\partial_h^rf(h)$ for all $h>h_c$
and $r\in\N$.

Pick a compact interval $H\subset(h_c,+\infty)$ and choose $s\in\N$ so
large that $[h_s,+\infty)\supset H$.  Bound (\ref{eq:Cinfty_1}) yields
  $|f_n^{(r)}(h)|\le (4/\gamma_s)^r(r!)^3 \,\Exd[\Lambda_s]$ for all
  $n,r\in\N$ and $h\ge h_s$.  It also allows us to invoke
    Fubini's theorem to state that
    $f_n^{(r)}(h)-f_n^{(r)}(h')=\int_{h'}^h
    f_n^{(r+1)}(\zeta)\dd\zeta$ for every $r\in\N_0$, $n\in\N$, and
    $h,h'\ge h_s$.  These facts permit us an application of the
  Arzel\`a--Ascoli theorem, which, in combination with a
  diagonalization argument, implies the existence of continuous
  functions $f^{(0)},f^{(1)},\ldots$ on $H$ and of an increasing
  sequence $\{n_i\}_{i\in\N}$ in $\N$ such that
  $\lim_{i\uparrow\infty}\sup_{h\in H}|f_{n_i}^{(r)}(h)-f^{(r)}(h)|=0$
  for all $r\in\N_0$. It follows that $f^{(0)}(h)=f(h)$ and that
  $f^{(r)}(h)-f^{(r)}(h')=\int_{h'}^hf^{(r+1)}(\zeta)\dd\zeta$
  for every $r\in\N_0$ and $h,h'\in H$. Thus, $f$ turns out to be
  infinitely differentiable on $H$, and hence on $(h_c,+\infty)$
  because the interval $H$ is arbitrary, and $f^{(r)}=\partial_h^r f$,
  so that $\lim_{i\uparrow\infty}\sup_{h\in
    H}|f_{n_i}^{(r)}(h)-\partial_h^r f(h)|=0$ for all $r\in\N_0$.
  Convergence of the derivatives is not restricted to the sequence
  $\{n_i\}_{i\in\N}$.  In fact, assume by contradiction that there
  exist $r_o\in\N_0$, $\epsilon>0$, and a divergent sequence
  $\{m_i\}_{i\in\N}$ in $\N$ with the property that $\sup_{h\in
    H}|f^{(r_o)}_{m_i}(h)-\partial_h^{r_o}f(h)|\ge\epsilon$ for all
  $i\in\N$. By repeating all the above arguments, we can find an
  increasing subsequence $\{m_{i_j}\}_{j\in\N}$ of $\{m_i\}_{i\in\N}$
  such that $\lim_{j\uparrow\infty}\sup_{h\in
    H}|f^{(r)}_{m_{i_j}}(h)-\partial_h^rf(h)|=0$ for all $r\in\N_0$,
  which is a contradiction when $r=r_o$.

The function $f$ is strictly convex on $(h_c,+\infty)$ thanks to Lemma
\ref{lem:var_positive} because
$\partial_h^2f(h)=\lim_{n\uparrow\infty}f_n^{(2)}(h)$ for all $h>h_c$
on the one hand and $f_n^{(2)}(h)=\Exd[\var_{n,h,\cdot}[L_n]]/n$ for
all $n\in\N$ and $h\in\Rl$ on the other hand. Moreover, for every
closed set $H\subset(h_c,+\infty)$ there exists $c>0$ such that
$\sup_{h\in H}|\partial_h^rf(h)|\le c^r(r!)^3$ for all $r\in\N$.  In
fact, given $H\subset(h_c,+\infty)$ closed and $s\in\N$ such that
$[h_s,+\infty)\supset H$, bound (\ref{eq:Cinfty_1}) yields
  $|f_n^{(r)}(h)|\le (4/\gamma_s)^r(r!)^3 \,\Exd[\Lambda_s]$ for all
  $n,r\in\N$ and $h\in H$. By sending $n$ to infinity we get
  $|\partial_h^rf(h)|\le c^r(r!)^3$ for all $r\in\N$ and $h\in H$ with
  $c:=(4/\gamma_s)\max\{1,\Exd[\Lambda_s]\}$.

\smallskip

\textit{Results for the typical behaviour of the finite-volume free
  energy.}  Since $\Exd[\Lambda_s]<+\infty$, Birkhoff's ergodic
theorem states that there exists $\Omega_o\in\mathcal{F}$ with
$\probd[\Omega_o]=1$ such that for every $\omega\in\Omega_o$ and
$s\in\N$
\begin{equation}
A_s(\omega):=\sup_{n\in\N}\,\frac{1}{n}\sum_{i=0}^{n-1}\Lambda_s(\vartheta^i\omega)<+\infty\,.
\label{eq:Cinfty_2}
\end{equation}
As for each $h\in\Rl$ we have $\lim_{n\uparrow\infty}(1/n)\log
Z_{n,h}(\omega)=f(h)$ for $\pae$ (see \cite[Theorem
  4.6]{giacomin2007}) and as the set $\mathbb{Q}$ of rational numbers
is countable, we can even suppose that the limit
$\lim_{n\uparrow\infty}(1/n)\log Z_{n,h}(\omega)=f(h)$ holds for all
$\omega\in\Omega_o$ and $h\in\mathbb{Q}$. Then, recalling the bounds
$|(1/n)\log Z_{n,h}(\omega)-(1/n)\log Z_{n,h'}(\omega)|\le|h-h'|$ and
$|f(h)-f(h')|\le|h-h'|$ valid for all $n$, $\omega$, $h$, and $h'$, we
deduce that $\lim_{n\uparrow\infty}(1/n)\log Z_{n,h}(\omega)=f(h)$ for
all $\omega\in\Omega_o$ and $h\in\Rl$. Fixing $\omega\in\Omega_o$,
setting $f_n^{(r)}(h):=(1/n)\partial^r_h\log Z_{n,h}(\omega)$ for
$r\in\N_0$, and combining (\ref{eq:Cinfty_1}) with
(\ref{eq:Cinfty_2}), a repetition of all the above arguments with
$A_s(\omega)$ in place of $\Exd[\Lambda_s]$ shows that
$\lim_{n\uparrow\infty}\sup_{h\in H}|f_n^{(r)}(h)-\partial_h^rf(h)|=0$
for all $H\subset(h_c,+\infty)$ compact and $r\in\N_0$.
\end{proof}

\subsection{A concentration inequality and the CLT for the contact number}
\label{sec:CLTL}

The control we have on the derivatives of the finite-volume free
energy, which coincide with the cumulants of the contact number as
seen earlier, allows us to exploit the method of cumulants for normal
approximation (see \cite{doring2022}). In fact, the arguments of the
proof of Theorem \ref{th:Cinfty} also give the statements of Theorem
\ref{th:CLT+concentration}.

\begin{proof}[Proof of Theorem \ref{th:CLT+concentration}]
The concentration bound in part (i) of the theorem is essentially
\cite[Theorem 2.5]{doring2022} because the properties of cumulants
under affine transformations assure us that for each $r\ge 2$ the
derivative $\partial^r_h\log Z_{n,h}(\omega)$ is the cumulant of order
$r$ of $L_n-\Ex_{n,h,\omega}[L_n]$ with respect to the polymer law
$\prob_{n,h,\omega}$, and also of $-L_n+\Ex_{n,h,\omega}[L_n]$ except
possibly for a minus sign. The CLT in part (ii) follows from
\cite[Theorem 2.4]{doring2022} since $\partial^r_h\log
Z_{n,h}(\omega)/\sqrt{\var_{n,h,\omega}[L_n]^r}$ is the cumulant of
order $r\ge 2$ of the centered, unit-variance random variable
$(L_n-\Ex_{n,h,\omega}[L_n])/\sqrt{\var_{n,h,\omega}[L_n]}$. Let us go
into details.

\smallskip

Let $h_s$, $\gamma_s$, $\Lambda_s$, and $\Omega_o$ be as in the proof
of Theorem \ref{th:Cinfty}. Then, for all $\omega\in\Omega_o$,
$s\in\N$, and $H\subset(h_c,+\infty)$ compact we have
\begin{equation*}
  A_s(\omega):=\sup_{n\in\N}\,\frac{1}{n}\sum_{i=0}^{n-1}\Lambda_s(\vartheta^i\omega)<+\infty
\end{equation*}
and, recalling that $v(h):=\partial_h^2f(h)$,
\begin{equation}
  \adjustlimits\lim_{n\uparrow\infty}\sup_{h\in H}\bigg|\frac{\var_{n,h,\omega}[L_n]}{n}-v(h)\bigg|=0\,.
  \label{eq:stand_dev}
\end{equation}

\smallskip

\noindent\textit{The concentration inequality.}  Given
$\omega\in\Omega_o$ and $H\subset(h_c,+\infty)$ closed, combining
bound (\ref{eq:Cinfty_1}) with \cite[Theorem 2.5]{doring2022} after
choosing $s$ so large that $[h_s,+\infty)\supset H$, we realize that
  for every $n\in\N$, $h\in H$, and $u\ge 0$
\begin{equation*}
   \prob_{n,h,\omega}\Big[\big|L_n-\Ex_{n,h,\omega}[L_n]\big|>u\Big] \le 2\ee^{-\frac{u^2}{256 A_s(\omega)n/\gamma_s^2+2(4/\gamma_s)^{1/3} u^{5/3}}}\,.
\end{equation*}
This prove part (i) of Theorem \ref{th:CLT+concentration} with
$1/\kappa_\omega:=\max\{256
A_s(\omega)/\gamma_s^2,2(4/\gamma_s)^{1/3}\}$.

\smallskip

\noindent\textit{The CLT.} Fix $\omega\in\Omega_o$ and a compact set
$H\subset(h_c,+\infty)$ and let $s$ be so large that
$[h_s,+\infty)\supset H$. Bound (\ref{eq:Cinfty_1}) shows that the following
  \textit{Statulevi\v{c}ius condition} is satisfied: for every
  $n\in\N$, $h\in H$, and integer $r\ge 3$
\begin{equation*}
  \bigg|\frac{\partial^r_h\log Z_{n,h}(\omega)}{\sqrt{\var_{n,h,\omega}[L_n]^r}}\bigg|\le (r!)^3
  \Bigg(\max\bigg\{\frac{4}{\gamma_s\sqrt{\var_{n,h,\omega}[L_n]}},\frac{64A_s(\omega)n}{\gamma_s^3\sqrt{\var_{n,h,\omega}[L_n]^3}}\bigg\}\Bigg)^{\!\!r-2}\,,
\end{equation*}
because for positive numbers $\zeta$ and $z$ and an integer $r\ge 3$
we have $\zeta^rz=\zeta^{r-2}\zeta^2z\le
\zeta^{r-2}(\max\{1,\zeta^2z\})^{r-2}=(\max\{\zeta,\zeta^3z\})^{r-2}$.
In this way, \cite[Theorem 2.4]{doring2022} assures us that for all
$n\in\N$ and $h\in H$
\begin{align}
  \nonumber
 &\sup_{u\in\Rl}\,\Bigg|\prob_{n,h,\omega}\bigg[\frac{L_n-\Ex_{n,h,\omega}[L_n]}{\sqrt{\var_{n,h,\omega}[L_n]}}\le u\bigg]-
  \frac{1}{\sqrt{2\pi}}\int_{-\infty}^u\ee^{-\frac{1}{2}z^2}\dd z\Bigg|\\
  \nonumber
  &\qquad\qquad\qquad\qquad\qquad
  \le 145  \Bigg(\max\bigg\{\frac{4}{\gamma_s\sqrt{\var_{n,h,\omega}[L_n]}},\frac{64A_s(\omega)n}{\gamma_s^3\sqrt{\var_{n,h,\omega}[L_n]^3}}\bigg\}\Bigg)^{\!\!\frac{1}{5}}\,.
\end{align}
Replacing $u$ with the quantity $\sqrt{nv(h)/\var_{n,h,\omega}[L_n]}\,u$
in the left-hand side of this bound we deduce that
\begin{align}
  \nonumber
 &\sup_{u\in\Rl}\,\Bigg|\prob_{n,h,\omega}\bigg[\frac{L_n-\Ex_{n,h,\omega}[L_n]}{\sqrt{nv(h)}}\le u\bigg]-
  \frac{1}{\sqrt{2\pi}}\int_{-\infty}^u\ee^{-\frac{1}{2}z^2}\dd z\Bigg|\\
  \nonumber
  &\qquad\qquad\le 145\Bigg(\max\bigg\{\frac{4}{\gamma_s\sqrt{\var_{n,h,\omega}[L_n]}},\frac{64A_s(\omega)n}{\gamma_s^3\sqrt{\var_{n,h,\omega}[L_n]^3}}\bigg\}\Bigg)^{\!\!\frac{1}{5}}\\
  &\qquad\qquad\quad+\sup_{u\in\Rl}\,\Bigg| \frac{1}{\sqrt{2\pi}}\int_{-\infty}^{\sqrt{\frac{nv(h)}{\var_{n,h,\omega}[L_n]}}\,u}\ee^{-\frac{1}{2}z^2}\dd z
  -\frac{1}{\sqrt{2\pi}}\int_{-\infty}^u\ee^{-\frac{1}{2}z^2}\dd z\Bigg|\,.
  \label{eq:CLT_primobound}
\end{align}
We shall demonstrate in a moment that for all $u\in\Rl$ and $\zeta>0$
\begin{equation}
 \Bigg|\frac{1}{\sqrt{2\pi}}\int_{-\infty}^{u/\zeta}\ee^{-\frac{1}{2}z^2}\dd z-\frac{1}{\sqrt{2\pi}}\int_{-\infty}^u\ee^{-\frac{1}{2}z^2}\dd z\Bigg|
 \le |1-\zeta^2|\,,
  \label{eq:CLT_secondobound}
\end{equation}
so \eqref{eq:CLT_primobound} finally yields
\begin{align}
  \nonumber
 &\sup_{u\in\Rl}\,\Bigg|\prob_{n,h,\omega}\bigg[\frac{L_n-\Ex_{n,h,\omega}[L_n]}{\sqrt{nv(h)}}\le u\bigg]-
  \frac{1}{\sqrt{2\pi}}\int_{-\infty}^u\ee^{-\frac{1}{2}z^2}\dd z\Bigg|\\
  \nonumber
  &\qquad\le 145\Bigg(\max\bigg\{\frac{4}{\gamma_s\sqrt{\var_{n,h,\omega}[L_n]}},\frac{64A_s(\omega)n}{\gamma_s^3\sqrt{\var_{n,h,\omega}[L_n]^3}}\bigg\}\Bigg)^{\!\!\frac{1}{5}}
  +\frac{|\var_{n,h,\omega}[L_n]-nv(h)|}{nv(h)}\,.
\end{align}
To conclude, we note that there exists $h_o\in H$ such that
$\inf_{h\in H}v(h)=v(h_o)>0$ because $H$ is compact on the one hand
and $v:=\partial^2_h f$ is continuous and strictly positive throughout
$(h_c,+\infty)$ on the other hand. It follows that $\inf_{h\in
  H}\var_{n,h,\omega}[L_n]\ge v(h_o)n/4$ for all sufficiently large
$n$ by \eqref{eq:stand_dev}. Thus, for all sufficiently large $n$ we
find
\begin{align}
  \nonumber
 &\sup_{h\in H}\,\sup_{u\in\Rl}\,\Bigg|\prob_{n,h,\omega}\bigg[\frac{L_n-\Ex_{n,h,\omega}[L_n]}{\sqrt{nv(h)}}\le u\bigg]-
  \frac{1}{\sqrt{2\pi}}\int_{-\infty}^u\ee^{-\frac{1}{2}z^2}\dd z\Bigg|\\
  \nonumber
  &\qquad\le \frac{145}{n^{\frac{1}{10}}}\Bigg(\max\bigg\{\frac{8}{\gamma_s\sqrt{v(h_o)}},\frac{512A_s(\omega)}{\gamma_s^3\sqrt{v(h_o)^3}}\bigg\}\Bigg)^{\!\!\frac{1}{5}}
  +\frac{1}{v(h_o)}\sup_{h\in H}\bigg|\frac{\var_{n,h,\omega}[L_n]}{n}-v(h)\bigg|\,.
\end{align}
This proves part (ii) of Theorem \ref{th:CLT+concentration} thanks
to another application of \eqref{eq:stand_dev}.

\smallskip

\noindent\textit{The bound \eqref{eq:CLT_secondobound}.} Pick
$u\in\Rl$ and $\zeta>0$. If $\zeta\le 1/\sqrt{\ee}$, then we can write
\begin{align}
  \nonumber
  \bigg|\frac{1}{\sqrt{2\pi}}\int_{-\infty}^{u/\zeta}\ee^{-\frac{1}{2}z^2}\dd z-\frac{1}{\sqrt{2\pi}}\int_{-\infty}^u\ee^{-\frac{1}{2}z^2}\dd z\bigg|
  &=\frac{1}{\sqrt{2\pi}}\int_{|u|}^{|u|/\zeta}\ee^{-\frac{1}{2}z^2}\dd z\\
  \nonumber
  &\le \frac{1}{\sqrt{2\pi}}\int_0^{+\infty}\ee^{-\frac{1}{2}z^2}\dd z=\frac{1}{2}\,.
\end{align}
If $1/\sqrt{\ee}<\zeta\le 1$, then we have
\begin{align}
  \nonumber
  \bigg|\frac{1}{\sqrt{2\pi}}\int_{-\infty}^{u/\zeta}\ee^{-\frac{1}{2}z^2}\dd z-\frac{1}{\sqrt{2\pi}}\int_{-\infty}^u\ee^{-\frac{1}{2}z^2}\dd z\bigg|
  &=\frac{1}{\sqrt{2\pi}}\int_{|u|}^{|u|/\zeta}\ee^{-\frac{1}{2}z^2}\dd z\\
  \nonumber
  &\le \bigg(\frac{1}{\zeta}-1\bigg)|u|\ee^{-\frac{1}{2}u^2}\le \frac{1-\zeta}{\zeta\sqrt{\ee}}
\end{align}
and if instead $\zeta>1$, then we find
\begin{align}
  \nonumber
  \bigg|\frac{1}{\sqrt{2\pi}}\int_{-\infty}^{u/\zeta}\ee^{-\frac{1}{2}z^2}\dd z-\frac{1}{\sqrt{2\pi}}\int_{-\infty}^u\ee^{-\frac{1}{2}z^2}\dd z\bigg|
  &=\frac{1}{\sqrt{2\pi}}\int_{|u|/\zeta}^{|u|}\ee^{-\frac{1}{2}z^2}\dd z\\
  \nonumber
  &\le (\zeta-1)\frac{|u|}{\zeta}\ee^{-\frac{1}{2}(\frac{u}{\zeta})^2}\le\frac{\zeta-1}{\sqrt{\ee}}\,.
\end{align}
In all the three cases the right-hand side is smaller than or equal to $|1-\zeta^2|$.
\end{proof}

\section{A concentration bound and the CLT for the centering variable}
\label{sec:CLTEL}

In this section we study the centering variable $\Ex_{n,h,\cdot}[L_n]$
in order to prove Theorem \ref{th:centering} and Proposition
\ref{prop:fluctuations_EL}. Basic tools to this aim are the following
formulas due to Lemma \ref{lem:fact}: for every $h\in\Rl$,
$\omega\in\Omega$, and integers $0\le a\le m\le b\le n$
  \begin{align}
    \nonumber
    \Ex_{n,h,\omega}[X_a]&=\frac{\Ex_{n,h,\omega}[X_aX_m]}{\Ex_{n,h,\omega}[X_m]}-\frac{\Ex_{n,h,\omega}[X_aX_m]-\Ex_{n,h,\omega}[X_a]\Ex_{n,h,\omega}[X_m]}{\Ex_{n,h,\omega}[X_m]}\\
    &=\Ex_{m,h,\omega}[X_a]-\frac{\cov_{n,h,\omega}[X_a,X_m]}{\Ex_{n,h,\omega}[X_m]}
    \label{basic_tool_1}
  \end{align}
  and
  \begin{align}
    \nonumber
    \Ex_{n,h,\omega}[X_b]&=\frac{\Ex_{n,h,\omega}[X_mX_b]}{\Ex_{n,h,\omega}[X_m]}-\frac{\Ex_{n,h,\omega}[X_mX_b]-\Ex_{n,h,\omega}[X_m]\Ex_{n,h,\omega}[X_b]}{\Ex_{n,h,\omega}[X_m]}\\
    &=\Ex_{n-m,h,\vartheta^m\omega}[X_{b-m}]-\frac{\cov_{n,h,\omega}[X_m,X_b]}{\Ex_{n,h,\omega}[X_m]}\,.
    \label{basic_tool_2}
  \end{align}
We shall use these formulas several times in combination with Lemma
\ref{lem:mixing}.

\subsection{Cumulants of the centering variable}
Let
${}^{r\!}\mathcal{K}_{n,h}:=\partial_\lambda^r\log\Exd[\ee^{\lambda\Ex_{n,h,\cdot}[L_n]}]|_{\lambda=0}$
be the cumulant of order $r$ of $\Ex_{n,h,\cdot}[L_n]$ with respect to
the law $\probd$. The following lemma provides an overall description
of ${}^{r\!}\mathcal{K}_{n,h}$.

\begin{lemma}
  \label{lem:cumulants_EL}
  For each closed set $H\subset(h_c,+\infty)$ there exists a constant
  $c>0$ such that for all $r,n\in\N$
  \begin{equation*}
    \sup_{h\in H}\big|{}^{r\!}\mathcal{K}_{n,h}\big|\le c^r(r!)^3 n\,.
  \end{equation*}
\end{lemma}

\begin{proof}[Proof of Lemma \ref{lem:cumulants_EL}]
Fix any $H\subset(h_c,+\infty)$ closed. According to Lemma
\ref{lem:mixing}, there exist constants $\gamma>0$ and $G>0$ such
that for every integers $1\le a\le b\le n$
\begin{align}
  \nonumber
    \sup_{h\in H}\Exd\bigg[\frac{|\cov_{n,h,\cdot}[X_a,X_b]|}{\min\{\Ex_{n,h,\cdot}[X_a],\Ex_{n,h,\cdot}[X_b]\}}\bigg]&\le
    \Exd\bigg[\sup_{h\in H}\frac{|\cov_{n,h,\cdot}[X_a,X_b]|}{\min\{\Ex_{n,h,\cdot}[X_a],\Ex_{n,h,\cdot}[X_b]\}}\bigg]\\
    &\le G\ee^{-\gamma(b-a)}\,.
    \label{lem:cumulants_EL_0}
  \end{align}
  Put $c:=\max\{1,4\ee/\gamma,4\ee G\ee^{2\gamma}\}$.  We are going to
  prove that $\sup_{h\in H}|{}^{r\!}\mathcal{K}_{n,h}|\le c^r(r!)^3 n$
  for all $r,n\in\N$.

The case $r=1$ is trivial as
$|{}^{1\!}\mathcal{K}_{n,h}|=\Exd[\Ex_{n,h,\cdot}[L_n]]\le n$ and $c\ge
1$ by construction. Suppose $r\ge 2$ and fix $n\in\N$. We
have \begin{equation*}
  {}^{r\!}\mathcal{K}_{n,h}=\sum_{a_1=1}^n\cdots\sum_{a_r=1}^n\mathcal{U}_{n,h}(a_1,\ldots,a_r)\,,
  \end{equation*}
$\mathcal{U}_{n,h}(a_1,\ldots,a_r):=\partial_{\lambda_{a_1}}\!\!\cdots\,\partial_{\lambda_{a_r}}\log\Exd[\ee^{\sum_{a=1}^n\lambda_a\Ex_{n,h,\cdot}[X_a]}]|_{\lambda_1=\cdots=\lambda_n=0}$
being the joint cumulant of
$\Ex_{n,h,\cdot}[X_{a_1}],\ldots,\Ex_{n,h,\cdot}[X_{a_r}]$ with
respect to the probability measure $\probd$. The symmetry of
$\mathcal{U}_{n,h}(a_1,\ldots,a_r)$ yields
\begin{equation}
  \big|{}^{r\!}\mathcal{K}_{n,h}\big|\le r!\sum_{1\le a_1\le\cdots\le a_r\le n}\Big|\,\mathcal{U}_{n,h}(a_1,\ldots,a_r)\Big|\,.
  \label{lem:cumulants_EL_123}
\end{equation}

Repeating some of the arguments of the proof of Lemma \ref{Birkhoff}
we can state that, given an integer $s\in\{1,\ldots,r-1\}$, the
cumulant $\mathcal{U}_{n,h}(a_1,\ldots,a_r)$ with $1\le a_1\le\cdots\le a_r\le
n$ can be written by adding and subtracting less than $2^{r-1}(r-1)!$
terms of the form
\begin{multline}
\label{term_in_EU}
  \Bigg\{\Exd\bigg[\prod_{k\in I}\Ex_{n,h,\cdot}[X_{a_k}]\prod_{k\in J}\Ex_{n,h,\cdot}[X_{a_k}]\bigg]
  -\Exd\bigg[\prod_{k\in I}\Ex_{n,h}[X_{a_k}]\bigg]\Exd\bigg[\prod_{k\in J}\Ex_{n,h,\cdot}[X_{a_k}]\bigg]\Bigg\}\\
  \times
  \prod_{j=1}^l\Exd\bigg[\prod_{k\in R_j}\Ex_{n,h,\cdot}[X_{a_k}]\bigg]\,,
\end{multline}
with sets $I\subseteq\{1,\ldots,s\}$, $J\subseteq\{s+1,\ldots,r\}$,
and $R_1,\ldots,R_l$ constituting a partition of
$\{1,\ldots,r\}\setminus I\cup J$.  We claim that for $h\in H$ the
modulus of these terms is bounded above by
$2rG\ee^{-\frac{\gamma}{2}(a_{s+1}-a_s)+\gamma}$, so that
\begin{equation}
  \sup_{h\in H}\Big|\,\mathcal{U}_{n,h}(a_1,\ldots,a_r)\Big|\le 2^rr!\,G\ee^{-\frac{\gamma}{2}(a_{s+1}-a_s)+\gamma}\,.
\label{lem:cumulants_EL_1}
\end{equation}
In fact, bearing in mind that
$|\prod_{i=1}^lu_i-\prod_{i=1}^lv_i|\le\sum_{i=1}^l|u_i-v_i|$ for
every $l\in\N$ and $u_1,\ldots,u_l,v_1,\ldots,v_l\in[0,1]$, formula
(\ref{basic_tool_1}) with $m:=\lfloor{(a_s+a_{s+1})}/{2}\rfloor$
gives for all $\omega\in\Omega$
\begin{align}
  \nonumber
  \bigg|\prod_{k\in I}\Ex_{n,h,\omega}[X_{a_k}]-\prod_{k\in I}\Ex_{m,h,\omega}[X_{a_k}]\bigg|&\le\sum_{k\in I}\Big|\Ex_{n,h,\omega}[X_{a_k}]-\Ex_{m,h,\omega}[X_{a_k}]\Big|\\
  &\le \sum_{k\in I}\frac{|\cov_{n,h,\omega}[X_{a_k},X_m]|}{\Ex_{n,h,\omega}[X_m]}\,.
  \label{eq:repeat_1}
\end{align}
Similarly, formula (\ref{basic_tool_2}) with
$m:=\lfloor{(a_s+a_{s+1})}/{2}\rfloor$ yields for all
$\omega\in\Omega$
\begin{align}
  \nonumber
  \bigg|\prod_{k\in J}\Ex_{n,h,\omega}[X_{a_k}]-\prod_{k\in J}\Ex_{n-m,h,\vartheta^m\omega}[X_{a_k-m}]\bigg|
  &\le\sum_{k\in J}\Big|\Ex_{n,h,\omega}[X_{a_k}]-\Ex_{n-m,h,\vartheta^m\omega}[X_{a_k-m}]\Big|\\
  &\le\sum_{k\in J}\frac{|\cov_{n,h,\omega}[X_m,X_{a_k}]|}{\Ex_{n,h,\omega}[X_m]}\,.
   \label{eq:repeat_2}
\end{align}
Importantly, $\prod_{k\in I}\Ex_{m,h,\cdot}[X_{a_k}]$ and $\prod_{k\in
  J}\Ex_{n-m,h,\vartheta^m\cdot}[X_{a_k-m}]$ are statistically
independent. Then, by replacing in (\ref{term_in_EU}) the products
$\prod_{k\in I}\Ex_{n,h,\cdot}[X_{a_k}]$ and $\prod_{k\in
  J}\Ex_{n,h,\cdot}[X_{a_k}]$ with the products $\prod_{k\in
  I}\Ex_{m,h,\cdot}[X_{a_k}]$ and $\prod_{k\in
  J}\Ex_{n-m,h,\vartheta^m\cdot}[X_{a_k-m}]$, respectively, thanks to
(\ref{lem:cumulants_EL_0}) we find for $h\in H$
\begin{align}
  \nonumber
  &\Bigg|\Exd\bigg[\prod_{k\in I}\Ex_{n,h,\cdot}[X_{a_k}]\prod_{k\in J}\Ex_{n,h,\cdot}[X_{a_k}]\bigg]
  -\Exd\bigg[\prod_{k\in I}\Ex_{n,h}[X_{a_k}]\bigg]\Exd\bigg[\prod_{k\in J}\Ex_{n,h,\cdot}[X_{a_k}]\bigg]\Bigg|\\
  \nonumber
  &\qquad\le 2\sum_{k\in I}\Exd\bigg[\frac{|\cov_{n,h,\omega}[X_{a_k},X_m]|}{\Ex_{n,h,\omega}[X_m]}\bigg]
  +2\sum_{k\in J}\Exd\bigg[\frac{|\cov_{n,h,\omega}[X_m,X_{a_k}]|}{\Ex_{n,h,\omega}[X_m]}\bigg]\\
  \nonumber
  &\qquad\le 2\sum_{k\in I}G\ee^{-\gamma(m-a_k)}+ 2\sum_{k\in J}G\ee^{-\gamma(a_k-m)}\\
  \nonumber
  &\qquad\le 2sG\ee^{-\gamma(m-a_s)}+ 2(r-s)G\ee^{-\gamma(a_{s+1}-m)}\le 2rG\ee^{-\frac{\gamma}{2}(a_{s+1}-a_s)+\gamma}\,.
\end{align}

The arbitrariness of $s\in\{1,\ldots,r-1\}$ in
(\ref{lem:cumulants_EL_1}) gives for all $1\le a_1\le\cdots\le a_r\le
n$ 
\begin{equation*}
  \sup_{h\in H}\Big|\,\mathcal{U}_{n,h}(a_1,\ldots,a_r)\Big|\le 2^rr!\,G\ee^{-\frac{\gamma}{2(r-1)}\sum_{s=1}^{r-1}(a_{s+1}-a_s)+\gamma}\,.
\end{equation*}
Combining this bound with (\ref{lem:cumulants_EL_123}) and bearing in
mind that $1/(1-\ee^{-\zeta})\le \ee^\zeta/\zeta$ for
$\zeta>0$ and that $(r-1)^{r-1}\le r!\ee^r$, we finally find
\begin{align}
  \nonumber
  ~~~~~~~~~~~~~~~~~~~~\sup_{h\in H}\big|{}^{r\!}\mathcal{K}_{n,h}\big|&\le r!\sum_{1\le a_1\le\cdots\le a_r\le n}\sup_{h\in H}\Big|\,\mathcal{U}_{n,h}(a_1,\ldots,a_r)\Big|\\
  \nonumber
  &\le \frac{2^r(r!)^2 G\ee^{\gamma}}{\big(1-\ee^{-\frac{\gamma}{2(r-1)}}\big)^{r-1}}\,n
  \le 2^r(r!)^2 G\ee^{2\gamma}\bigg[\frac{2(r-1)}{\gamma}\bigg]^{\!r-1}n\\
  \nonumber
  &\le \gamma G\ee^{2\gamma}\bigg(\frac{4\ee}{\gamma}\bigg)^{\!\!r}(r!)^3n\le c^r(r!)^3 n\,.
  ~~~~~~~~~~~~~~~~~~~~~~~~~~~~~~~~~~~~~~~~~~~~~~~~\qedhere
\end{align}
\end{proof}

The acquired control of the cumulants of $\Ex_{n,h,\cdot}[L_n]$
immediately implies the concentration inequality presented in part
(i) of Theorem \ref{th:centering}.

\begin{proof}[Proof of part (i) of Theorem \ref{th:centering}] Lemma \ref{lem:cumulants_EL} states
that for each $H\subset(h_c,+\infty)$ closed there exists a constant
$c>0$ such that $\sup_{h\in H}|{}^{r\!}\mathcal{K}_{n,h}|\le
  c^r(r!)^3n$ for all $r,n\in\N$. Since ${}^{r\!}\mathcal{K}_{n,h}$ is
  the cumulant of $\Ex_{n,h,\cdot}[L_n]-\Exd[\Ex_{n,h,\cdot}[L_n]]$
  for $r\ge 2$, and also of
  $-\Ex_{n,h,\cdot}[L_n]+\Exd[\Ex_{n,h,\cdot}[L_n]]$ except possibly
  for a minus sign, \cite[Theorem 2.5]{doring2022} shows that for
  every $n\in\N$, $h\in H$, and $u\ge 0$
\begin{equation*}
   \probd\bigg[\Big|\Ex_{n,h,\cdot}[L_n]-\Exd\big[\Ex_{n,h,\cdot}[L_n]\big]\Big|>u\bigg] \le 2\ee^{-\frac{u^2}{16c^2n+2c^{-1/3}u^{5/3}}}\,.
 \end{equation*}
Thus, part (i) of Theorem \ref{th:centering} holds with
$1/\kappa:=\max\big\{16c^2,2c^{-1/3}\big\}$.
\end{proof}

\subsection{More on the mean and the variance}
We now deepen the study of the first two cumulants of
$\Ex_{n,h,\cdot}[L_n]$. First of all, we describe the convergence of
$\Exd[\Ex_{n,h,\cdot}[L_n]]/n$ towards $\rho(h):=\partial_h f(h)$ as
$n$ goes to infinity, thus proving part (ii) of Theorem
\ref{th:centering}.

\begin{proof}[Proof of part (ii) of Theorem \ref{th:centering}]
  Fix a closed set $H\subset(h_c,+\infty)$. According to Lemma
  \ref{lem:mixing}, there exist constants $\gamma>0$ and $G>0$
  such that for all integers $1\le a\le 2n$
  \begin{equation*}
    \sup_{h\in H}\Exd\bigg[\frac{|\cov_{2n,h,\cdot}[X_a,X_n]|}{\Ex_{2n,h,\cdot}[X_n]}\bigg]\le
    \Exd\bigg[\sup_{h\in H}\frac{|\cov_{2n,h,\cdot}[X_a,X_n]|}{\Ex_{2n,h,\cdot}[X_n]}\bigg]\le G\ee^{-\gamma|n-a|}\,.
  \end{equation*}
  We shall show in a moment that for every $n\in\N$ and $h\in H$
  \begin{equation}
    \label{lem:conv_L_0}
    \bigg|\Exd\Big[\Ex_{n,h,\cdot}[L_n]\Big]-\frac{1}{2}\Exd\Big[\Ex_{2n,h,\cdot}[L_{2n}]\Big]\bigg|\le \frac{G}{1-\ee^{-\gamma}}\,.
\end{equation}
Iterating this bound we get for all $n,i\in\N$, and $h\in H$
\begin{equation*} 
    \bigg|\Exd\Big[\Ex_{n,h,\cdot}[L_n]\Big]-\frac{1}{2^i}\Exd\Big[\Ex_{2^in,h,\cdot}[L_{2^in}]\Big]\bigg|\le
    \frac{2G}{1-\ee^{-\gamma}}\,.
\end{equation*}
The latter proves part (ii) of Theorem \ref{th:centering} with
$c:=2G/(1-\ee^{-\gamma})$ by sending $i$ to infinity thanks to Theorem
\ref{th:Cinfty} and (\ref{eq:contact_fraction}).

\smallskip

\noindent\textit{The bound \eqref{lem:conv_L_0}.} Formula
(\ref{basic_tool_1}) gives for every $\omega\in\Omega$ and
$a\in\{1,\ldots,n\}$
  \begin{equation*}
 \Ex_{2n,h,\omega}[X_a]=\Ex_{n,h,\omega}[X_a]-\frac{\cov_{2n,h,\omega}[X_a,X_n]}{\Ex_{2n,h,\omega}[X_n]}\,.
  \end{equation*}
  Formula (\ref{basic_tool_2}) gives for every $\omega\in\Omega$ and
  $a\in\{n+1,\ldots,2n\}$
  \begin{align}
    \nonumber
    \Ex_{2n,h,\omega}[X_a]=\Ex_{n,h,\vartheta^n\omega}[X_{a-n}]-\frac{\cov_{2n,h,\omega}[X_a,X_n]}{\Ex_{2n,h,\omega}[X_n]}\,.
  \end{align}
Taking the sum over $a$ and integrating with respect to
$\probd[\dd\omega]$, we realize that for all $n\in\N$ and $h\in H$
  \begin{equation*}
    \bigg|\Exd\Big[\Ex_{2n,h,\cdot}[L_{2n}]\Big]-2\Exd\Big[\Ex_{n,h,\cdot}[L_n]\Big]\bigg|
    \le\sum_{a=1}^{2n}\Exd\bigg[\frac{|\cov_{2n,h,\cdot}[X_a,X_n]|}{\Ex_{2n,h,\cdot}[X_n]}\bigg]\le \frac{2G}{1-\ee^{-\gamma}}\,.
    \qedhere
\end{equation*}
\end{proof}

Next, we investigate the asymptotics of the variance of
$\Ex_{n,h,\cdot}[L_n]$, demonstrating part (iii) of Theorem
\ref{th:centering}.

\begin{proof}[Proof of part (iii) of Theorem
\ref{th:centering}] Given a random variable $\Lambda$ on $\Omega$,
  here we denote for brevity $\Lambda-\Exd[\Lambda]$ by
  $\overline{\Lambda}$.  The proof of part (iii) of Theorem
  \ref{th:centering} is quite involved and we split it into two
  parts. We first show that the limit
\begin{equation}
  \lim_{n\uparrow\infty}\Exd\Bigg[\bigg(\frac{\overline{\Ex_{n,h,\cdot}[L_n]}}{\sqrt{n}}\bigg)^{\!\!2}\Bigg]=
  \lim_{n\uparrow\infty}\frac{1}{n}\Exd\Big[\overline{\Ex_{n,h,\cdot}[L_n]}^2\Big]=:w(h)
  \label{eq:existence_limit_w}
\end{equation}
exists and is finite for all $h>h_c$. Then we prove that for every
closed set $H\subset(h_c,+\infty)$ there exists a constant $c>0$ such
that for all $n\in\N$ and $h,h'\in H$
\begin{equation}
  \bigg|\Exd\Big[\overline{\Ex_{n,h,\cdot}[L_n]}^2\Big]-\Exd\Big[\overline{\Ex_{n,h',\cdot}[L_n]}^2\Big]\bigg|\le cn|h-h'|\,.
  \label{eq:w_equicontinuity}
\end{equation}
The latter demonstrates immediately that the function $w$ that associates
$h\in(h_c,+\infty)$ with $w(h)$ is locally Lipschitz continuous. It also
demonstrates that the limit \eqref{eq:existence_limit_w} is uniform on
the compact sets in the localized phase. In fact, given
$H\subset(h_c,+\infty)$ compact, assume by way of contradiction that
there exists $\epsilon>0$ and a sequence $\{n_i\}_{i\in\N}$ of
positive integers such that for all $i\in\N$
\begin{equation}
  \sup_{h\in H}\bigg|\frac{1}{n_i}\Exd\Big[\overline{\Ex_{n_i,h,\cdot}[L_{n_i}]}^2\Big]-w(h)\bigg|\ge \epsilon\,.
\label{eq:w_uniform_contradiction}
\end{equation}
Since the sequence of functions whose $i^{\mathrm{th}}$ term maps
$h\in H$ to $\Exd[\overline{\Ex_{n_i,h,\cdot}[L_{n_i}]}^2]/n_i$ is
uniformly bounded and equicontinuous by \eqref{eq:existence_limit_w}
and \eqref{eq:w_equicontinuity}, according to the Arzel\`a--Ascoli
theorem there exists a subsequence $\{n_{i_j}\}_{j\in\N}$ of
$\{n_i\}_{i\in\N}$ that convergences uniformly. Its limit is the
restriction of $w$ to $H$ because of \eqref{eq:existence_limit_w}, but
this contradicts \eqref{eq:w_uniform_contradiction}.

\textit{Existence of the limit \eqref{eq:existence_limit_w}}.  Fix
$h>h_c$ and put
$\nu_n:=\Exd\big[\overline{\Ex_{2n,h,\cdot}[L_n]}^2\big]$.  We shall
show that there exists a constant $\delta>0$ such that for all $m,n\in\N$
\begin{equation}
  \label{eq:var_EL_1}
  \nu_{m+n}\le \nu_m+\nu_n+\delta\big(\sqrt{\nu_{\min\{m,n\}}}+1\big)
\end{equation}
and
\begin{equation}
  \label{eq:var_EL_2}
  \bigg|\Exd\Big[\overline{\Ex_{n,h,\cdot}[L_n]}^2\Big]-\nu_n\bigg|\le \delta\big(\sqrt{\nu_n}+1\big) \,.
\end{equation}
Bounds (\ref{eq:var_EL_1}) and (\ref{eq:var_EL_2}) prove the limit
\eqref{eq:existence_limit_w} as follows.  The former entails that
$\nu_{n+1}\le\nu_n+\nu_1+\delta(\sqrt{\nu_1}+1)\le\nu_n+1+2\delta$,
which in turn gives $\nu_n\le(1+2\delta)n$ for each $n\in\N$. Thus,
since $\sqrt{\min\{m,n\}}\le 2(\sqrt{m}+\sqrt{n}-\sqrt{m+n})$, we
deduce from (\ref{eq:var_EL_1}) again that
$\nu_n+2\delta\sqrt{(1+2\delta)n}+\delta$ is the $n^{\mathrm{th}}$
term of a subadditive sequence. It therefore follows that
$\lim_{n\uparrow\infty}\nu_n/n=:w(h)$ exists and is finite. Bound
(\ref{eq:var_EL_2}) implies that
$\lim_{n\uparrow\infty}\Exd[\overline{\Ex_{n,h,\cdot}[L_n]}^2]/n=w(h)$.

We point out that the sequence with $n^{\mathrm{th}}$ term
$\Exd\big[\overline{\Ex_{n,h,\cdot}[L_n]}^2\big]$ does not appear to
have a direct subadditivity property like (\ref{eq:var_EL_1}). For
this reason, we found it convenient to double the system size and work
with $\nu_n$ in place of
$\Exd\big[\overline{\Ex_{n,h,\cdot}[L_n]}^2\big]$, and then to come
back to $\Exd\big[\overline{\Ex_{n,h,\cdot}[L_n]}^2\big]$ by means of
(\ref{eq:var_EL_2}).

\smallskip

Let us verify (\ref{eq:var_EL_1}).  To begin with, we note that by
Lemma \ref{lem:mixing} and Corollary \ref{cor:mixing} there exist
constants $\gamma>0$ and $G>0$ such that for every integers $0\le
a\le b\le l$
  \begin{equation}
    \Exd\bigg[\frac{|\cov_{l,h,\cdot}[X_a,X_b]|}{\min\{\Ex_{l,h,\cdot}[X_a],\Ex_{l,h,\cdot}[X_b]\}}\bigg]\le G\ee^{-\gamma(b-a)}
    \label{eq:var_EL_3}
  \end{equation}
and
  \begin{equation}
    \bigg|\Exd\Big[\Ex_{l,h,\cdot}[X_a]\Ex_{l,h,\cdot}[X_b]\Big]-\Exd\Big[\Ex_{l,h,\cdot}[X_a]\Big]\Exd\Big[\Ex_{l,h,\cdot}[X_b]\Big]\bigg|\le G\ee^{-\gamma(b-a)}\,.
    \label{eq:var_EL_31}
  \end{equation}
Pick $m,n\in\N$ and suppose, without restriction, that $m\le
n$. Decomposing $L_{m+n}$ as the sum of $L_n$ and $L_{m+n}-L_n$ we can
write
\begin{align}
  \nonumber
  \nu_{m+n}:=\Exd\Big[\overline{\Ex_{{2m+2n},h,\cdot}[L_{m+n}]}^2\Big]&=\Exd\Big[\overline{\Ex_{2m+2n,h,\cdot}[L_n]}^2\Big]\\
  \nonumber
  &\quad +\Exd\Big[\overline{\Ex_{2m+2n,h,\cdot}[L_{m+n}-L_n]}^2\Big]\\
  &\quad\quad +2\Exd\Big[\overline{\Ex_{2m+2n,h,\cdot}[L_n]}\,\overline{\Ex_{2m+2n,h,\cdot}[L_{m+n}-L_n]}\Big]\,.
\label{eq:var_EL_4}
\end{align}
Regarding the third term in the right-hand side of
(\ref{eq:var_EL_4}), bound (\ref{eq:var_EL_31}) shows that
\begin{align}
  \nonumber
  &\Exd\Big[\overline{\Ex_{2m+2n,h,\cdot}[L_n]}\,\overline{\Ex_{2m+2n,h,\cdot}[L_{m+n}-L_n]}\Big]\\
  \nonumber
  &\qquad=\sum_{a=1}^n\sum_{b=n+1}^{m+n}\Exd\Big[\overline{\Ex_{2m+2n,h,\cdot}[X_a]}\,\overline{\Ex_{2m+2n,h,\cdot}[X_b]}\Big]\\
  \nonumber
  &\qquad\le\sum_{a=1}^n\sum_{b=n+1}^{m+n}\bigg|\Exd\Big[\Ex_{2m+2n,h,\cdot}[X_a]\Ex_{2m+2n,h,\cdot}[X_b]\Big]-\Exd\Big[\Ex_{2m+2n,h,\cdot}[X_a]\Big]
    \Exd\Big[\Ex_{2m+2n,h,\cdot}[X_b]\Big]\bigg|\\
  \nonumber
  &\qquad\le \frac{G\ee^{\gamma}}{(\ee^\gamma-1)^2}\le \frac{G}{\gamma(1-\ee^{-\gamma})}\,.
\end{align}
Thus, (\ref{eq:var_EL_4}) implies
\begin{equation}
  \label{eq:var_EL_5}
  \nu_{m+n}\le\Exd\Big[\overline{\Ex_{2m+2n,h,\cdot}[L_n]}^2\Big]
  +\Exd\Big[\overline{\Ex_{2m+2n,h,\cdot}[L_{m+n}-L_n]}^2\Big]+\frac{2G}{\gamma(1-\ee^{-\gamma})}\,.
\end{equation}

We now analyze the first term in the right-hand side of
(\ref{eq:var_EL_5}). Decomposing $\Ex_{2m+2n,h,\omega}[L_n]$ as the
sum of $\Ex_{2n,h,\omega}[L_n]$ and
$\Ex_{2m+2n,h,\omega}[L_n]-\Ex_{2n,h,\omega}[L_n]$ and bearing in mind
that $0\le L_n\le n$ we deduce the inequality
\begin{align}
  \nonumber
  \Exd\Big[\overline{\Ex_{2m+2n,h,\cdot}[L_n]}^2\Big]
  &=\nu_n+2\Exd\Big[\overline{\Ex_{2n,h,\cdot}[L_n]}\,\overline{\Ex_{2m+2n,h,\cdot}[L_n]-\Ex_{2n,h,\cdot}[L_n]}\Big]\\
  \nonumber
  &\quad+\Exd\Big[\overline{\Ex_{2m+2n,h,\cdot}[L_n]-\Ex_{2n,h,\cdot}[L_n]}^2\Big]\\
  \nonumber
  &\le \nu_n+3n\Exd\bigg[\Big|\overline{\Ex_{2m+2n,h,\cdot}[L_n]-\Ex_{2n,h,\cdot}[L_n]}\Big|\bigg]\\
  \nonumber
  &\le  \nu_n+6n\Exd\bigg[\Big|\Ex_{2m+2n,h,\cdot}[L_n]-\Ex_{2n,h,\cdot}[L_n]\Big|\bigg]\,.
\end{align}
On the other hand, formula (\ref{basic_tool_1}) and bound
(\ref{eq:var_EL_3}) yield
\begin{align}
  \nonumber
  \Exd\bigg[\Big|\Ex_{2m+2n,h,\cdot}[L_n]-\Ex_{2n,h,\cdot}[L_n]\Big|\bigg]
  &\le\sum_{a=1}^n \Exd\bigg[\Big|\Ex_{2m+2n,h,\cdot}[X_a]-\Ex_{2n,h,\cdot}[X_a]\Big|\bigg]\\
  \nonumber
  &=\sum_{a=1}^n\Exd\bigg[\frac{|\cov_{2m+2n,h,\cdot}[X_a,X_{2n}]|}{\Ex_{2m+2n,h,\cdot}[X_{2n}]}\bigg]
  \le \frac{G\ee^{-\gamma n}}{1-\ee^{-\gamma}}\,.
\end{align}
It therefore follows that
\begin{align}
  \nonumber
  \Exd\Big[\overline{\Ex_{2m+2n,h,\cdot}[L_n]}^2\Big]
  \le  \nu_n+\frac{6G}{1-\ee^{-\gamma}}n\ee^{-\gamma n}\le\nu_n+\frac{3G}{\gamma(1-\ee^{-\gamma})}\,,
\end{align}
and (\ref{eq:var_EL_5}) entails
\begin{equation}
  \label{eq:var_EL_6}
  \nu_{m+n}\le\nu_n
  +\Exd\Big[\overline{\Ex_{2m+2n,h,\cdot}[L_{m+n}-L_n]}^2\Big]+\frac{5G}{\gamma(1-\ee^{-\gamma})}\,.
\end{equation}

To conclude, we investigate the second term in the right-hand side of
(\ref{eq:var_EL_6}). Decomposing the expectation
$\Ex_{2m+2n,h,\omega}[L_{m+n}-L_n]$ as the sum of the expectation
$\Ex_{2m,h,\vartheta^n\omega}[L_m]$ and the difference
$\Ex_{2m+2n,h,\omega}[L_{m+n}-L_n]-\Ex_{2m,h,\vartheta^n\omega}[L_m]$
and appealing to the Cauchy--Schwarz inequality we get
\begin{align}
  \nonumber
  &\Exd\Big[\overline{\Ex_{2m+2n,h,\cdot}[L_{m+n}-L_n]}^2\Big]\\
  \nonumber
  &\qquad=\nu_m+2\Exd\Big[\overline{\Ex_{2m,h,\vartheta^n\cdot}[L_m]}\,\overline{\Ex_{2m+2n,h,\cdot}[L_{m+n}-L_n]-\Ex_{2m,h,\vartheta^n\cdot}[L_m]}\Big]\\
  \nonumber
  &\qquad\quad +\Exd\Big[\overline{\Ex_{2m+2n,h,\cdot}[L_{m+n}-L_n]-\Ex_{2m,h,\vartheta^n\cdot}[L_m]}^2\Big]\\
  \nonumber
  &\qquad\le \nu_m+2\sqrt{\nu_m}\sqrt{\Exd\Big[\overline{\Ex_{2m+2n,h,\cdot}[L_{m+n}-L_n]-\Ex_{2m,h,\vartheta^n\cdot}[L_m]}^2\Big]}\\
  &\qquad\quad+\Exd\Big[\overline{\Ex_{2m+2n,h,\cdot}[L_{m+n}-L_n]-\Ex_{2m,h,\vartheta^n\cdot}[L_m]}^2\Big]\,.
  \label{eq:var_EL_99}
\end{align}
Moreover, decomposing the difference
$\Ex_{2m+2n,h,\omega}[L_{m+n}-L_n]-\Ex_{2m,h,\vartheta^n\omega}[L_m]$
as the sum of
$\Ex_{2m+2n,h,\omega}[L_{m+n}-L_n]-\Ex_{2m+n,h,\vartheta^n\omega}[L_m]$
and $\Ex_{2m+n,h,\vartheta^n\omega}[L_m]-\Ex_{2m,h,\vartheta^n\omega}[L_m]$, and then
recalling that $0\le L_m\le m$, we find
\begin{align}
  \nonumber
  &\Exd\Big[\overline{\Ex_{2m+2n,h,\cdot}[L_{m+n}-L_n]-\Ex_{2m,h,\vartheta^n\cdot}[L_m]}^2\Big]\\
  \nonumber
  &\qquad\le \Exd\bigg[\Big(\Ex_{2m+2n,h,\cdot}[L_{m+n}-L_n]-\Ex_{2m,h,\vartheta^n\cdot}[L_m]\Big)^{\!2}\bigg]\\
  \nonumber
  &\qquad\le 2\Exd\bigg[\Big(\Ex_{2m+2n,h,\cdot}[L_{m+n}-L_n]-\Ex_{2m+n,h,\vartheta^n\cdot}[L_m]\Big)^{\!2}\bigg]\\
  \nonumber
  &\qquad\quad+2\Exd\bigg[\Big(\Ex_{2m+n,h,\cdot}[L_m]-\Ex_{2m,h,\cdot}[L_m]\Big)^{\!2}\bigg]\\
  \nonumber
  &\qquad\le 2\Exd\bigg[\Big(\Ex_{2m+2n,h,\cdot}[L_{m+n}-L_n]-\Ex_{2m+n,h,\vartheta^n\cdot}[L_m]\Big)^{\!2}\bigg]\\
  &\qquad\quad+2m\Exd\bigg[\Big|\Ex_{2m+n,h,\cdot}[L_m]-\Ex_{2m,h,\cdot}[L_m]\Big|\bigg]\,.
  \label{eq:var_EL_7}
\end{align}
The first term in the right-hand side of (\ref{eq:var_EL_7}) reads
\begin{align}
  \nonumber
  &\Exd\bigg[\Big(\Ex_{2m+2n,h,\cdot}[L_{m+n}-L_n]-\Ex_{2m+n,h,\vartheta^n\cdot}[L_m]\Big)^{\!2}\bigg]\\
  \nonumber
  &\qquad=\sum_{a=1}^m\sum_{b=1}^m\Exd\bigg[\Big(\Ex_{2m+2n,h,\cdot}[X_{n+a}]-\Ex_{2m+n,h,\vartheta^n\cdot}[X_a]\Big)\\
    \nonumber
    &\qquad\qquad\qquad\qquad\qquad\quad\times\Big(\Ex_{2m+2n,h,\cdot}[X_{n+b}]-\Ex_{2m+n,h,\vartheta^n\cdot}[X_b]\Big)\bigg]\,,
\end{align}
and, thanks to the fact that
$|\Ex_{2m+2n,h,\cdot}[X_{n+a}]-\Ex_{2m+n,h,\vartheta^n\cdot}[X_a]|\le
1$ for every $a$, the Cauchy--Schwarz inequality shows that
\begin{align}
  \nonumber
  &\Exd\bigg[\Big(\Ex_{2m+2n,h,\cdot}[L_{m+n}-L_n]-\Ex_{2m+n,h,\vartheta^n\cdot}[L_m]\Big)^{\!2}\bigg]\\
  &\qquad\le\Bigg\{\sum_{a=1}^m\sqrt{\Exd\bigg[\Big|\Ex_{2m+2n,h,\cdot}[X_{n+a}]-\Ex_{2m+n,h,\vartheta^n\cdot}[X_a]\Big|\bigg]}\Bigg\}^{\!2}\,.
  \label{eq:var_EL_101}
\end{align}
Thus, formula (\ref{basic_tool_2}) and bound
(\ref{eq:var_EL_3}) give
\begin{align}
  \nonumber
  &\Exd\bigg[\Big(\Ex_{2m+2n,h,\cdot}[L_{m+n}-L_n]-\Ex_{2m+n,h,\vartheta^n\cdot}[L_m]\Big)^{\!2}\bigg]\\
  &\qquad\le\Bigg\{\sum_{a=1}^m\sqrt{\Exd\bigg[\frac{|\cov_{2m+2n,h,\cdot}[X_n,X_{n+a}]|}{\Ex_{2m+2n,h,\cdot}[X_n]}\bigg]}\Bigg\}^{\!2}
  \le \frac{G}{(\ee^{\frac{\gamma}{2}}-1)^2}\le \frac{4G}{\gamma(1-\ee^{-\gamma})}\,.
   \label{eq:var_EL_102}
\end{align}
Moving to the second term in the right-hand side of
(\ref{eq:var_EL_7}), we invoke (\ref{basic_tool_1}) and
(\ref{eq:var_EL_3}) to obtain
\begin{align}
  \nonumber
  \Exd\bigg[\Big|\Ex_{2m+n,h,\cdot}[L_m]-\Ex_{2m,h,\cdot}[L_m]\Big|\bigg]&\le\sum_{a=1}^m \Exd\bigg[\Big|\Ex_{2m+n,h,\cdot}[X_a]-\Ex_{2m,h,\cdot}[X_a]\Big|\bigg]\\
  \nonumber
  &=\sum_{a=1}^m\Exd\bigg[\frac{|\cov_{2m+n,h,\cdot}[X_a,X_{2m}]|}{\Ex_{2m+n,h,\cdot}[X_{2m}]}\bigg]
  \le \frac{G\ee^{-\gamma m}}{1-\ee^{-\gamma}}\,.
\end{align}
In this way, (\ref{eq:var_EL_7}) shows that
\begin{align}
  \nonumber
  \Exd\Big[\overline{\Ex_{2m+2n,h,\cdot}[L_{m+n}-L_n]-\Ex_{2m,h,\vartheta^n\cdot}[L_m]}^2\Big]
  &\le \frac{8G}{\gamma(1-\ee^{-\gamma})}+\frac{2G}{1-\ee^{-\gamma}}m\ee^{-\gamma m}\\
  \nonumber
  &\le \frac{9G}{\gamma(1-\ee^{-\gamma})}\,,
\end{align}
which, coming back to \eqref{eq:var_EL_99}, in turn yields
 \begin{equation*}
  \Exd\Big[\overline{\Ex_{2m+2n,h,\cdot}[L_{m+n}-L_n]}^2\Big]
  \le \nu_m+\sqrt{\nu_m}\sqrt{\frac{36G}{\gamma(1-\ee^{-\gamma})}}+\frac{9G}{\gamma(1-\ee^{-\gamma})}\,.
\end{equation*}
Combining this inequality with (\ref{eq:var_EL_6}) we finally find
  \begin{equation*}
  \nu_{m+n}\le\nu_m+\nu_n+\sqrt{\nu_m}\sqrt{\frac{36G}{\gamma(1-\ee^{-\gamma})}}+\frac{14G}{\gamma(1-\ee^{-\gamma})}\,, 
\end{equation*}
  which proves (\ref{eq:var_EL_1}) with
\begin{equation}
\delta:=\max\Bigg\{\sqrt{\frac{36G}{\gamma(1-\ee^{-\gamma})}},\,\frac{14G}{\gamma(1-\ee^{-\gamma})}\Bigg\}\,.
\label{deg_G_var_EL}
\end{equation}

\smallskip

It remains to demonstrate (\ref{eq:var_EL_2}). To this aim, we fix
$n\in\N$ and write for all $\omega\in\Omega$
  \begin{align}
    \nonumber
    \overline{\Ex_{n,h,\omega}[L_n]}^2-\overline{\Ex_{2n,h,\omega}[L_n]}^2&=2\overline{\Ex_{2n,h,\omega}[L_n]}\,\overline{\Ex_{n,h,\omega}[L_n]-\Ex_{2n,h,\omega}[L_n]}\\
    &\quad +\overline{\Ex_{n,h,\omega}[L_n]-\Ex_{2n,h,\omega}[L_n]}^2\,.
\end{align}
  Integrating with respect to $\probd[\dd\omega]$ and applying the
  Cauchy--Schwarz inequality we get
  \begin{align}
    \nonumber
    \bigg|\Exd\Big[\overline{\Ex_{n,h,\cdot}[L_n]}^2\Big]-\nu_n\bigg|&\le 2\sqrt{\nu_n}\sqrt{\Exd\Big[\overline{\Ex_{2n,h,\cdot}[L_n]-\Ex_{n,h,\cdot}[L_n]}^2\Big]}\\
    \nonumber
    &\quad+\Exd\Big[\overline{\Ex_{2n,h,\cdot}[L_n]-\Ex_{n,h,\cdot}[L_n]}^2\Big]\,.
  \end{align}
Similarly to \eqref{eq:var_EL_101} and \eqref{eq:var_EL_102}, the use
of the Cauchy--Schwarz inequality first, then formula
(\ref{basic_tool_1}), and finally bound (\ref{eq:var_EL_3}) allow us
to deduce that
 \begin{align}
    \nonumber
    &\Exd\Big[\overline{\Ex_{2n,h,\cdot}[L_n]-\Ex_{n,h,\cdot}[L_n]}^2\Big]\\
    \nonumber
    &\qquad \le \Exd\bigg[\Big(\Ex_{2n,h,\cdot}[L_n]-\Ex_{n,h,\cdot}[L_n]\Big)^{\!2}\bigg]\\
    \nonumber
    &\qquad\le\Bigg\{\sum_{a=1}^n\sqrt{\Exd\bigg[\Big|\Ex_{2n,h,\cdot}[X_a]-\Ex_{n,h,\cdot}[X_a]\Big|\bigg]}\Bigg\}^{\!2}\\
    \nonumber
    &\qquad=\Bigg\{\sum_{a=1}^n\sqrt{\Exd\bigg[\frac{|\cov_{2n,h,\cdot}[X_a,X_n]|}{\Ex_{2n,h,\cdot}[X_n]}\bigg]}\Bigg\}^{\!2}
    \le \frac{G}{(\ee^{\frac{\gamma}{2}}-1)^2}\le \frac{4G}{\gamma(1-\ee^{-\gamma})}\,.
\end{align}
  In conclusion, we find
  \begin{equation*}
    \bigg|\Exd\Big[\overline{\Ex_{n,h,\cdot}[L_n]}^2\Big]-\nu_n\bigg|\le 
    \sqrt{\nu_n}\sqrt{\frac{16G}{\gamma(1-\ee^{-\gamma})}}+\frac{4G}{\gamma(1-\ee^{-\gamma})}\le \delta\big(\sqrt{\nu_n}+1\big)
\end{equation*}
with $\delta$ as in (\ref{deg_G_var_EL}).

\textit{Lipschitz property \eqref{eq:w_equicontinuity}.}  Pick a
closed set $H\subset(h_c,+\infty)$ and let $H_o$ be the smallest
closed interval that contains $H$. Since $H_o\subset(h_c,+\infty)$,
by Lemma \ref{lem:mixing} there exist constants $\gamma>0$ and $G>0$
such that for every integers $0\le a\le b\le n$
  \begin{equation}
    \Exd\bigg[\sup_{h\in H_o}\,\frac{|\cov_{n,h,\cdot}[X_a,X_b]|}{\min\{\Ex_{n,h,\cdot}[X_a],\Ex_{n,h,\cdot}[X_b]\}}\bigg]\le G\ee^{-\gamma(b-a)}\,.
    \label{eq:mixing_w_equicontinuity}
  \end{equation}
  We are going to show that for all $n\in\N$ and $h\in H_o$
  \begin{equation*}
  \bigg|\partial_h\Exd\Big[\overline{\Ex_{n,h,\cdot}[L_n]}^2\Big]\bigg|\le\frac{26G\ee^{\gamma/3}n}{(1-\ee^{-\gamma/3})^2}\,,
\end{equation*}
which proves the Lipschitz property \eqref{eq:w_equicontinuity} with
$c:=26G\ee^{\gamma/3}/(1-\ee^{-\gamma/3})^2$ by Lagrange's mean value
theorem.

Fix any $n\in\N$ and $h\in H_o$.  To begin with, we note that
\begin{equation*}
    \partial_h\Exd\Big[\overline{\Ex_{n,h,\cdot}[L_n]}^2\Big]
    =2\Exd\bigg[\overline{\Ex_{n,h,\cdot}[L_n]}\Big(\Ex_{n,h,\cdot}[L_n^2]-\Ex_{n,h,\cdot}[L_n]^2\Big)\bigg]\,,
\end{equation*}
which using that $L_n=\sum_{a=1}^nX_a$ gives
\begin{equation}
  \bigg|\partial_h\Exd\Big[\overline{\Ex_{n,h,\cdot}[L_n]}^2\Big]\bigg|
  \le 4\sum_{a=1}^n\sum_{b=1}^n\sum_{c=b}^n\bigg|\Exd\Big[\overline{\Ex_{n,h,\cdot}[X_a]}\,\cov_{n,h,\cdot}[X_b,X_c]\Big]\bigg|\,.
 \label{eq:zero_bound_equicontinuity}
\end{equation}
Since $|\overline{\Ex_{n,h,\cdot}[X_a]}|\le 1$, bound
\eqref{eq:mixing_w_equicontinuity} shows that for all $a$ and $b\le c$
\begin{equation}
  \bigg|\Exd\Big[\overline{\Ex_{n,h,\cdot}[X_a]}\,\cov_{n,h,\cdot}[X_b,X_c]\Big]\bigg|\le
  \Exd\Big[\big|\cov_{n,h,\cdot}[X_b,X_c]\big|\Big]\le G\ee^{-\gamma(c-b)}\,.
  \label{eq:first_bound_equicontinuity}
\end{equation}
This estimate suffices when $b\le a\le c$, whereas it needs to be
improved when $a<b\le c$ and $b\le c<a$.  In order to do that, suppose
first that $a<b\le c$. Lemma \ref{lem:fact} implies that
$\Ex_{n,h,\cdot}[X_bX_c]=\Ex_{c,h,\cdot}[X_b]\Ex_{n,h,\cdot}[X_c]$, so we can write
\begin{equation}
  \Exd\Big[\overline{\Ex_{n,h,\cdot}[X_a]}\,\cov_{n,h,\cdot}[X_b,X_c]\Big]=\Exd\Big[\overline{\Ex_{n,h,\cdot}[X_a]}\big(\Ex_{c,h,\cdot}[X_b]-\Ex_{n,h,\cdot}[X_b]\big)\Ex_{n,h,\cdot}[X_c]\Big]\,.
\label{eq:second_bound_equicontinuity_0}
\end{equation}
Identity \eqref{basic_tool_1} with $m:=\lfloor (a+b)/2\rfloor$ allows
us to write
\begin{equation}
    \Ex_{n,h,\omega}[X_a]=\Ex_{m,h,\omega}[X_a]-\frac{\cov_{n,h,\omega}[X_a,X_m]}{\Ex_{n,h,\omega}[X_m]}\,.
\label{eq:second_bound_equicontinuity_1}
\end{equation}
Similarly, identity \eqref{basic_tool_2} with $m:=\lfloor
(a+b)/2\rfloor$ yields
\begin{equation}
    \Ex_{c,h,\omega}[X_b]=\Ex_{c-m,h,\vartheta^m\omega}[X_{b-m}]-\frac{\cov_{c,h,\omega}[X_m,X_b]}{\Ex_{c,h,\omega}[X_m]}\,,
\label{eq:second_bound_equicontinuity_2}
\end{equation}
\begin{equation}
    \Ex_{n,h,\omega}[X_b]=\Ex_{n-m,h,\vartheta^m\omega}[X_{b-m}]-\frac{\cov_{n,h,\omega}[X_m,X_b]}{\Ex_{n,h,\omega}[X_m]}\,,
\label{eq:second_bound_equicontinuity_21}
\end{equation}
and
\begin{equation}
    \Ex_{n,h,\omega}[X_c]=\Ex_{n-m,h,\vartheta^m\omega}[X_{c-m}]-\frac{\cov_{n,h,\omega}[X_m,X_c]}{\Ex_{n,h,\omega}[X_m]}\,.
\label{eq:second_bound_equicontinuity_3}
\end{equation}
Thus, since $\Ex_{m,h,\cdot}[X_a]$ is statistically independent of
$\Ex_{c-m,h,\vartheta^m\cdot}[X_{b-m}]$,
$\Ex_{n-m,h,\vartheta^m\omega}[X_{b-m}]$, and
$\Ex_{n-m,h,\vartheta^m\cdot}[X_{c-m}]$ and since all these
expectations do not exceed $1$, plugging
\eqref{eq:second_bound_equicontinuity_1},
\eqref{eq:second_bound_equicontinuity_2},
      \eqref{eq:second_bound_equicontinuity_21}, and
      \eqref{eq:second_bound_equicontinuity_3} in
      \eqref{eq:second_bound_equicontinuity_0} we deduce by the same
      arguments that led to \eqref{eq:repeat_1} and
      \eqref{eq:repeat_2} that
\begin{align}
  \nonumber
  &\bigg|\Exd\Big[\overline{\Ex_{n,h,\cdot}[X_a]}\,\cov_{n,h,\cdot}[X_b,X_c]\Big]\bigg|\\
  \nonumber
  &\qquad\le 2\Exd\bigg[\frac{|\cov_{n,h,\omega}[X_a,X_m]|}{\Ex_{n,h,\omega}[X_m]}\bigg]+\Exd\bigg[\frac{|\cov_{c,h,\omega}[X_m,X_b]|}{\Ex_{c,h,\omega}[X_m]}\bigg]\\
  \nonumber
  &\qquad\quad +\Exd\bigg[\frac{|\cov_{n,h,\omega}[X_m,X_b]|}{\Ex_{n,h,\omega}[X_m]}\bigg]+\Exd\bigg[\frac{|\cov_{n,h,\omega}[X_m,X_c]|}{\Ex_{n,h,\omega}[X_m]}\bigg]\,.
\end{align}
In this way, \eqref{eq:mixing_w_equicontinuity} implies
\begin{equation*}
  \bigg|\Exd\Big[\overline{\Ex_{n,h,\cdot}[X_a]}\,\cov_{n,h,\cdot}[X_b,X_c]\bigg]\Bigg|\le 5G\ee^{-\frac{\gamma}{2}(b-a)+\gamma}\,,
\end{equation*}
and combining this inequality with
\eqref{eq:first_bound_equicontinuity} we finally find
\begin{equation}
  \bigg|\Exd\Big[\overline{\Ex_{n,h,\cdot}[X_a]}\,\cov_{n,h,\cdot}[X_b,X_c]\Big]\bigg|
  \le G(5\ee^{\gamma})^{\frac{2}{3}}\,\ee^{-\frac{\gamma}{3}(b-a)-\frac{\gamma}{3}(c-b)}\,.
   \label{eq:second_bound_equicontinuity}
\end{equation}

When $b\le c<a$, repeating the above arguments with $m:=\lfloor
(c+a)/2\rfloor$ we get
\begin{align}
  \nonumber
  &\bigg|\Exd\Big[\overline{\Ex_{n,h,\cdot}[X_a]}\,\cov_{n,h,\cdot}[X_b,X_c]\Big]\bigg|\\
  \nonumber
  &\qquad\le \Exd\bigg[\frac{|\cov_{n,h,\omega}[X_c,X_m]|}{\Ex_{n,h,\omega}[X_m]}\bigg]+\Exd\bigg[\frac{|\cov_{n,h,\omega}[X_b,X_m]|}{\Ex_{n,h,\omega}[X_m]}\bigg]+2\Exd\bigg[\frac{|\cov_{n,h,\omega}[X_m,X_a]|}{\Ex_{n,h,\omega}[X_m]}\bigg]\\
  \nonumber
  &\qquad\le 4G\ee^{-\frac{\gamma}{2}(a-c)+\gamma}\,,
\end{align}
so
\begin{equation}
  \bigg|\Exd\Big[\overline{\Ex_{n,h,\cdot}[X_a]}\,\cov_{n,h,\cdot}[X_b,X_c]\Big]\bigg|
  \le G(4\ee^{\gamma})^{\frac{2}{3}}\,\ee^{-\frac{\gamma}{3}(c-b)-\frac{\gamma}{3}(a-c)}\,.
 \label{eq:third_bound_equicontinuity}
\end{equation}
We stress that $\Ex_{c,h,\cdot}[X_b]$ is independent of
$\Ex_{n-m,h,\vartheta^m\cdot}[X_{a-m}]$ and, for this reason, it does
not need any manipulation to achieve
\eqref{eq:third_bound_equicontinuity}.

At this point, we use bound \eqref{eq:first_bound_equicontinuity} in
\eqref{eq:zero_bound_equicontinuity} for $b\le a\le c$, bound
\eqref{eq:second_bound_equicontinuity} for $a<b\le c$, and bound
\eqref{eq:third_bound_equicontinuity} for $b\le c<a$ to get
\begin{equation*}
  \bigg|\partial_h\Exd\Big[\overline{\Ex_{n,h,\cdot}[L_n]}^2\Big]\bigg|
  \le \frac{4Gn}{(1-\ee^{-\gamma})^2}+\frac{4G(5^{2/3}+4^{2/3})\ee^{\gamma/3}n}{(1-\ee^{-\gamma/3})^2}
  \le\frac{26G\ee^{\gamma/3}n}{(1-\ee^{-\gamma/3})^2}\,.
  \qedhere
\end{equation*}

\end{proof}

The following lemma demonstrates that the function $w$ defined in the
localized phase by part (iii) of Theorem \ref{th:centering} is
strictly positive unless the model is non disordered.

\begin{lemma}
  \label{lem:wh_positive}
  $w(h)>0$ for all $h>h_c$ provided that $\int_\Omega\omega_0^2\,
  \probd[\dd\omega]>0$.
\end{lemma}

Before going into the proof of Lemma~\ref{lem:wh_positive}, let us
point out that if $\omega_0$ is Gaussian distributed with standard
deviation $\beta>0$, then one can exploit Gaussian integration by
parts and obtain $w(h)\ge \beta^2 v(h)^2>0$ for $h>h_c$. We give the
argument here because it inspired the (substantially more involved)
proof for general disorder laws. The Cauchy--Schwarz inequality gives
for all $n\in\N$ and $h\in\Rl$
\begin{align}
  \nonumber
  \int_\Omega \sum_{a=1}^n\omega_a \Ex_{n,h,\omega}[L_n]\,\probd[\dd\omega]&=
  \int_\Omega \sum_{a=1}^n\omega_a \Big(\Ex_{n,h,\omega}[L_n]-\Exd\big[\Ex_{n,h,\cdot}[L_n]\big]\Big)\probd[\dd\omega]\\
  &\le \beta\sqrt{n}\sqrt{\Exd\bigg[\Big(\Ex_{n,h,\cdot}[L_n]-\Exd\big[\Ex_{n,h,\cdot}[L_n]\big]\Big)^{\!2}\bigg]}\,.
  \label{eq:remark_1}
\end{align}
On the other hand, Gaussian integration by parts yields
\begin{align}
  \nonumber
  \int_\Omega \sum_{a=1}^n\omega_a \Ex_{n,h,\omega}[L_n]\,\probd[\dd\omega]&=
  \beta^2\int_\Omega \sum_{a=1}^n \partial_{\omega_a} \Ex_{n,h,\omega}[L_n]\,\probd[\dd\omega]\\
  &=\beta^2\,\Exd\bigg[\Ex_{n,h,\cdot}\Big[\big(L_n-\Ex_{n,h,\cdot}[L_n]\big)^{\!2}\Big]\bigg]\,.
  \label{eq:remark_2}
\end{align}
Combining (\ref{eq:remark_1}) with (\ref{eq:remark_2}), dividing by
$n$, and then letting $n$ go to infinity we find $w(h)\ge \beta^2
v(h)^2$ for $h>h_c$ thanks to part (iii) of Theorem
\ref{th:centering} and (\ref{eq:lim_vh}).

\begin{proof}[Proof of Lemma \ref{lem:wh_positive}]
Fix $h>h_c$ and assume that $\int_\Omega\omega_0^2\,
\probd[\dd\omega]>0$. The idea for the proof is to find suitable
bounded, zero-mean random variables
$\Lambda_{n,1},\ldots,\Lambda_{n,n}$ on $\Omega$ to be plugged in the
following bound due to the Cauchy--Schwarz inequality: for every
$n\in\N$
\begin{align}
  \nonumber
  \Exd\bigg[\sum_{a=1}^n\Lambda_{n,a}\Ex_{n,h,\cdot}[L_n]\bigg]&=
  \Exd\bigg[\sum_{a=1}^n\Lambda_{n,a}\Big(\Ex_{n,h,\cdot}[L_n]-\Exd\big[\Ex_{n,h,\cdot}[L_n]\big]\Big)\bigg]\\
  \nonumber
  &\le \sqrt{\Exd\bigg[\bigg(\sum_{a=1}^n\Lambda_{n,a}\bigg)^{\!\!2}\bigg]}\sqrt{\Exd\bigg[\Big(\Ex_{n,h,\cdot}[L_n]-\Exd\big[\Ex_{n,h,\cdot}[L_n]\big]\Big)^{\!2}\bigg]}\,.
\end{align}
Therefore $w(h)>0$ follows if $\Lambda_{n,1},\ldots,\Lambda_{n,n}$
satisfy
\begin{equation}
  \liminf_{n\uparrow\infty}\frac{1}{n}\Exd\bigg[\sum_{a=1}^n\Lambda_{n,a}\Ex_{n,h,\cdot}[L_n]\bigg]>0
  \label{eq:wh_positive_1}
\end{equation}
and
\begin{equation}
  \limsup_{n\uparrow\infty}\frac{1}{n}\Exd\bigg[\bigg(\sum_{a=1}^n\Lambda_{n,a}\bigg)^{\!\!2}\bigg]<+\infty\,.
  \label{eq:wh_positive_2}
\end{equation}

Let us introduce a possible choice for the auxiliary variables
$\Lambda_{n,1},\ldots,\Lambda_{n,n}$. We note that $\int_\Omega
\omega_0\mathds{1}_{\{\omega_0\le 0\}}\probd[\dd\omega]<0$ and
$\int_\Omega \omega_0\mathds{1}_{\{\omega_0>0\}}\probd[\dd\omega]>0$
because $\int_\Omega \omega_0\,\probd[\dd\omega]=0$ and $\int_\Omega
\omega_0^2\,\probd[\dd\omega]>0$ by hypothesis. Moreover, according to
Lemma \ref{lem:decay} and Lemma \ref{lem:mixing} there exist constants
$\gamma>0$ and $G>0$ such that for every integers $1\le a\le b\le n$
\begin{equation}
  \Exd\Bigg[\Ex_{n,h,\cdot}^{\otimes 2}\bigg[\prod_{k=1}^{n-1}(1-X_kX_k')\bigg]\Bigg]\le G\ee^{-\gamma n}
  \label{eq:wh_positive_12345}
\end{equation}
and
\begin{equation}
  \label{eq:wh_positive_00}
   \Exd\Big[\big|\cov_{n,h,\cdot}[X_a,X_b]\big|\Big]\le G\ee^{-\gamma(b-a)}\,.
\end{equation}
In this way, recalling that $v(h):=\partial_h^2f(h)>0$, by elementary
considerations we can find a large $\lambda>0$ such that
\begin{equation}
  \begin{cases}
    \displaystyle{q_-:=-\int_\Omega \omega_0\mathds{1}_{\{-\lambda<\omega_0\le 0\}}\probd[\dd\omega]>0}\\[1em]
    \displaystyle{q_+:=\int_\Omega \omega_0\mathds{1}_{\{0<\omega_0\le\lambda\}}\probd[\dd\omega]>0}\\[1em]
     \displaystyle{\sqrt{G\!\int_\Omega\mathds{1}_{\{|\omega_0|>\lambda\}}\probd[\dd\omega]}\,\frac{1+\ee^{-\frac{\gamma}{2}}}{1-\ee^{-\frac{\gamma}{2}}}\le v(h)}
  \end{cases}\,.
  \label{eq:wh_positive_0}
\end{equation}
Put
$q(\zeta):=q_-\zeta\mathds{1}_{\{0<\zeta\le\lambda\}}+q_+\zeta\mathds{1}_{\{-\lambda<\zeta\le
  0\}}$ for $\zeta\in\Rl$ and, given
$\omega:=\{\omega_b\}_{b\in\N_0}\in\Omega$ and $a\in\N$, denote by
${}^a\omega:=\{\omega_0,\ldots,\omega_{a-1},0,\omega_{a+1},\ldots\}$ a
system of charges where the charge at site $a$ is suppressed.  For
$n\in\N$ and $a\in\{1,\ldots,n\}$ we finally define the random
variable $\Lambda_{n,a}$ as the function that maps $\omega$ to
\begin{equation*}
 \Lambda_{n,a}(\omega):=q(\omega_a)\,\cov_{n,h,{}^a\omega}[X_a,L_n]\,.
\end{equation*}
We have $\Exd[\Lambda_{n,a}]=0$ as $\cov_{n,h,{}^a\omega}[X_a,L_n]$ is
independent of $\omega_a$ and as $\int_\Omega
q(\omega_0)\probd[\dd\omega]=0$ by construction.

\smallskip

The proof of \eqref{eq:wh_positive_1} and \eqref{eq:wh_positive_2}
needs a formula that makes explicit the dependence of the expectation
of an observable on the charge at a certain site $a\in\{1,\ldots,n\}$.
Exploiting the binary nature of $X_a$, for any bounded measurable
function $\Phi$ on $(\mathcal{S},\mathfrak{S})$ we can write
\begin{align}
  \nonumber
  \Ex_{n,h,\omega}[\Phi]&=\frac{\Ex[\Phi\,\ee^{\sum_{b=1}^n(h+\omega_b)X_b}X_n]}{\Ex[\ee^{\sum_{b=1}^n(h+\omega_b)X_b}X_n]}\\
  \nonumber
  &=\frac{\Ex_{n,h,{}^a\omega}[\Phi]+(\ee^{\omega_a}-1)\Ex_{n,h,{}^a\omega}[X_a\Phi]}{1+(\ee^{\omega_a}-1)\,\Ex_{n,h,{}^a\omega}[X_a]}\\
  &=\Ex_{n,h,{}^a\omega}[\Phi]+\frac{\ee^{\omega_a}-1}{1+(\ee^{\omega_a}-1)\,\Ex_{n,h,{}^a\omega}[X_a]}\,\cov_{n,h,{}^a\omega}[X_a,\Phi]\,.
  \label{eq:expectation_a}
\end{align}
This formula allows us to isolate $\omega_a$.  Taking the derivative
with respect to $\omega_a$ we also find
\begin{equation}
  \cov_{n,h,\omega}[X_a,\Phi]=\frac{\ee^{\omega_a}}{\big\{1+(\ee^{\omega_a}-1)\,\Ex_{n,h,{}^a\omega}[X_a]\big\}^{\!2}}\,\cov_{n,h,{}^a\omega}[X_a,\Phi]\,.
  \label{eq:covariance_a}
\end{equation}
We are now ready to prove \eqref{eq:wh_positive_1} and
\eqref{eq:wh_positive_2}.

\smallskip

\noindent\textit{The bound \eqref{eq:wh_positive_1}.}  Formula
\eqref{eq:expectation_a} gives for all $n\in\N$ and
$\omega:=\{\omega_b\}_{b\in\N_0}$
\begin{align}
  \nonumber
  \Lambda_{n,a}(\omega)\Ex_{n,h,\omega}[L_n]&=\Lambda_{n,a}(\omega)\Ex_{n,h,{}^a\omega}[L_n]\\
  \nonumber
  &\quad+\frac{q(\omega_a)(\ee^{\omega_a}-1)}{1+(\ee^{\omega_a}-1)\,\Ex_{n,h,{}^a\omega}[X_a]}\Big(\cov_{n,h,{}^a\omega}[X_a,L_n]\Big)^{\!2}\,.
\end{align}
Then, since $q(\zeta)(\ee^\zeta-1)\ge
q_-\zeta(\ee^\zeta-1)\mathds{1}_{\{0<\zeta\le\lambda\}}\ge
q_-\zeta^2\mathds{1}_{\{0<\zeta\le\lambda\}}$, we can state that
\begin{equation*}
  \Lambda_{n,a}(\omega)\Ex_{n,h,\omega}[L_n]\ge  \Lambda_{n,a}(\omega)\Ex_{n,h,{}^a\omega}[L_n]+q_-\ee^{-\lambda}\mathds{1}_{\{0<\omega_a\le \lambda\}}\omega_a^2
  \Big(\cov_{n,h,{}^a\omega}[X_a,L_n]\Big)^{\!2}\,.
\end{equation*}
In this way, integrating with respect to $\probd[\dd\omega]$ and
bearing in mind that $\int_\Omega q(\omega_0)\probd[\dd\omega]=0$ and
that both $\cov_{n,h,{}^a\omega}[X_a,L_n]$ and
$\Ex_{n,h,{}^a\omega}[L_n]$ are independent of $\omega_a$, we get for
all $n\in\N$
\begin{align}
  \nonumber
  \Exd\bigg[\sum_{a=1}^n\Lambda_{n,a}\Ex_{n,h,\cdot}[L_n]\bigg]&\ge q_-\ee^{-\lambda}\int_\Omega\mathds{1}_{\{0<\omega_0\le \lambda\}}\omega_0^2\,\probd[\dd\omega]
  \int_\Omega\sum_{a=1}^n \Big(\cov_{n,h,{}^a\omega}[X_a,L_n]\Big)^{\!2}\,\probd[\dd\omega]\\
  &\ge q_- q_+^2\ee^{-\lambda}\int_\Omega\sum_{a=1}^n \Big(\cov_{n,h,{}^a\omega}[X_a,L_n]\Big)^{\!2}\,\probd[\dd\omega]\,,
\label{eq:wh_positive_4}
\end{align}
where the second inequality is due to the Cauchy--Schwarz inequality.
At the same time, formula (\ref{eq:covariance_a}) entails
\begin{align}
  \nonumber
  \Big(\cov_{n,h,{}^a\omega}[X_a,L_n]\Big)^{\!2}&=\ee^{-2\omega_a}\Big\{1+(\ee^{\omega_a}-1)\Ex_{n,h,{}^a\omega}[X_a]\Big\}^{\!4}\Big(\cov_{n,h,\omega}[X_a,L_n]\Big)^{\!2}\\
  \nonumber
  &\ge \ee^{-2|\omega_a|}\Big(\cov_{n,h,\omega}[X_a,L_n]\Big)^{\!2}\\
  \nonumber
  &\ge \ee^{-2\lambda}\mathds{1}_{\{|\omega_a|\le\lambda\}}\Big(\cov_{n,h,\omega}[X_a,L_n]\Big)^{\!2}\,,
\end{align}
where the second step follows by considering separately the values of
$\omega_a$ positive and negative.  Plugging this bound into
(\ref{eq:wh_positive_4}), two applications of the Cauchy--Schwarz
inequality yield
\begin{align}
  \nonumber
  &\Exd\bigg[\sum_{a=1}^n\Lambda_{n,a}\Ex_{n,h,\cdot}[L_n]\bigg]\\
  \nonumber
  &\qquad\ge q_- q_+^2\ee^{-3\lambda}\int_\Omega\sum_{a=1}^n \mathds{1}_{\{|\omega_a|\le\lambda\}}\Big(\cov_{n,h,\omega}[X_a,L_n]\Big)^{\!2}\,\probd[\dd\omega]\\
  \nonumber
  &\qquad\ge \frac{q_- q_+^2\ee^{-3\lambda}}{n}\int_\Omega\bigg\{\sum_{a=1}^n \mathds{1}_{\{|\omega_a|\le\lambda\}}\cov_{n,h,\omega}[X_a,L_n]\bigg\}^{\!2}\,\probd[\dd\omega]\\
  \nonumber
  &\qquad\ge \frac{q_- q_+^2\ee^{-3\lambda}}{n}\Bigg\{\int_\Omega\sum_{a=1}^n \mathds{1}_{\{|\omega_a|\le\lambda\}}\cov_{n,h,\omega}[X_a,L_n]\,\probd[\dd\omega]\Bigg\}^{\!2}\\
  &\qquad=\frac{q_- q_+^2\ee^{-3\lambda}}{n}\Bigg\{\Exd\Big[\var_{n,h,\cdot}[L_n]\Big]-
  \int_\Omega\sum_{a=1}^n \mathds{1}_{\{|\omega_a|>\lambda\}}\cov_{n,h,\omega}[X_a,L_n]\,\probd[\dd\omega]\Bigg\}^{\!2}\,.
\label{eq:wh_positive_7}
\end{align}
Finally, a further use of the Cauchy--Schwarz inequality, combined
with the fact that $|\cov_{n,h,\omega}[X_a,X_b]|\le 1$, gives thanks
to (\ref{eq:wh_positive_00}) and the last of (\ref{eq:wh_positive_0})
\begin{align}
  \nonumber
  &\int_\Omega\sum_{a=1}^n \mathds{1}_{\{|\omega_a|>\lambda\}}\cov_{n,h,\omega}[X_a,L_n]\,\probd[\dd\omega]\\
  \nonumber
  &\qquad =
  \sum_{a=1}^n\sum_{b=1}^n\int_\Omega\mathds{1}_{\{|\omega_a|>\lambda\}}\cov_{n,h,\omega}[X_a,X_b]\,\probd[\dd\omega]\\
  \nonumber
  &\qquad\le \sum_{a=1}^n\sum_{b=1}^n\sqrt{\int_\Omega\mathds{1}_{\{|\omega_0|>\lambda\}}\probd[\dd\omega]}\sqrt{\Exd\bigg[\Big|\cov_{n,h,\cdot}[X_a,X_b]\Big|\bigg]}\\
  \nonumber
  &\qquad\le n\sqrt{G\!\int_\Omega\mathds{1}_{\{|\omega_0|>\lambda\}}\probd[\dd\omega]}\,\frac{1+\ee^{-\frac{\gamma}{2}}}{1-\ee^{-\frac{\gamma}{2}}}\le n\frac{v(h)}{2}\,.
\end{align}
In this way, (\ref{eq:wh_positive_7}) and (\ref{eq:lim_vh}) allow us
to conclude that
\begin{equation*}
  \liminf_{n\uparrow\infty}\frac{1}{n}\Exd\bigg[\sum_{a=1}^n\Lambda_{n,a}\Ex_{n,h,\cdot}[L_n]\bigg]\ge \frac{q_- q_+^2v(h)^2\ee^{-3\lambda}}{4}\,,
\end{equation*}
which demonstrates (\ref{eq:wh_positive_1}).

\smallskip

\noindent\textit{The bound \eqref{eq:wh_positive_2}.}  Writing for
every $n\in\N$
\begin{equation}
\Exd\bigg[\bigg(\sum_{a=1}^n\Lambda_{n,a}\bigg)^{\!\!2}\bigg]=\sum_{a=1}^n\Exd\big[\Lambda_{n,a}^2\big]+2\sum_{a=1}^{n-1}\sum_{b=a+1}^n\Exd\big[\Lambda_{n,a}\Lambda_{n,b}\big]\,,
\label{eq:wh_positive_49}
\end{equation}
we analyze separately the expectations $\Exd[\Lambda_{n,a}^2]$ for
$1\le a\le n$ and $\Exd[\Lambda_{n,a}\Lambda_{n,b}]$ for $1\le a<b\le
n$.  Regarding the former, the Cauchy--Schwarz inequality and the fact
that $|\cov_{n,h,\omega}[X_a,X_c]|\le 1$ give for any $a$
\begin{align}
  \nonumber
  \Exd\big[\Lambda_{n,a}^2\big]&=\sum_{c=1}^n\sum_{d=1}^n\int_\Omega q(\omega_a)^2\,\cov_{n,h,{}^a\omega}[X_a,X_c]\,\cov_{n,h,{}^a\omega}[X_a,X_d]\,\probd[\dd\omega]\\
  &\le \Bigg\{\sum_{c=1}^n\sqrt{\int_\Omega q(\omega_a)^2\Big|\cov_{n,h,{}^a\omega}[X_a,X_c]\Big|\,\probd[\dd\omega]}\Bigg\}^{\!2}\,.
  \label{eq:wh_positive_10}
\end{align}
On the other hand, formula (\ref{eq:covariance_a}) allows us to
conclude that
\begin{align}
  \nonumber
  &\int_\Omega q(\omega_a)^2\Big|\cov_{n,h,{}^a\omega}[X_a,X_c]\Big|\,\probd[\dd\omega]\\
  \nonumber
  &\qquad=\int_\Omega q(\omega_a)^2\ee^{-\omega_a}\Big\{1+(\ee^{\omega_a}-1)\Ex_{n,h,{}^a\omega}[X_a]\Big\}^{\!2}\Big|\cov_{n,h,\omega}[X_a,X_c]\Big|\,\probd[\dd\omega]\\
  &\qquad\le \int_\Omega q(\omega_a)^2\ee^{|\omega_a|}\Big|\cov_{n,h,\omega}[X_a,X_c]\Big|\,\probd[\dd\omega]\,.
   \label{eq:wh_positive_11}
\end{align}
Thus, since $|q(\zeta)|\le\max\{q_-,q_+\}\lambda$ for
$\zeta\in(-\lambda,\lambda]$ and $q(\zeta)=0$ for
  $\zeta\notin(-\lambda,\lambda]$, from (\ref{eq:wh_positive_10}) and
    (\ref{eq:wh_positive_11}) first and (\ref{eq:wh_positive_00})
    later we get
    \begin{align}
  \nonumber
  \Exd\big[\Lambda_{n,a}^2\big]&\le  \max\{q_-^2,q_+^2\}\lambda^2\ee^{\lambda}\Bigg\{\sum_{c=1}^n\sqrt{\Exd\Big[\big|\cov_{n,h,\cdot}[X_a,X_c]\big|\Big]}\Bigg\}^{\!2}\\
  &\le G\ee^{\lambda}\bigg(\max\{q_-,q_+\}\lambda\,\frac{1+\ee^{-\frac{\gamma}{2}}}{1-\ee^{-\frac{\gamma}{2}}}\bigg)^{\!\!2}\,.
  \label{eq:wh_positive_50}
\end{align}

The study of the expectations $\Exd[\Lambda_{n,a}\Lambda_{n,b}]$ is
more involved. Fix integers $1\le a<b\le n$ and denote by
${}^{a,b}\omega$ the system of charges
$\{\omega_0,\ldots,\omega_{a-1},0,\omega_{a+1},\ldots,\omega_{b-1},0,\omega_{b+1},\ldots\}$.
We have
\begin{align}
  \nonumber
  \Exd\big[\Lambda_{n,a}\Lambda_{n,b}\big]&=\int_\Omega q(\omega_a)q(\omega_b)\cov_{n,h,{}^a\omega}[X_a,L_n]\cov_{n,h,{}^b\omega}[X_b,L_n]\,\probd[\dd\omega]\\
  \nonumber
  &=\int_\Omega q(\omega_a)q(\omega_b)\Big(\cov_{n,h,{}^a\omega}[X_a,L_n]-\cov_{n,h,{}^{a,b}\omega}[X_a,L_n]\Big)\\
  \nonumber
  &\qquad\qquad\qquad\quad\times\Big(\cov_{n,h,{}^b\omega}[X_b,L_n]-\cov_{n,h,{}^{a,b}\omega}[X_b,L_n]\Big)\,\probd[\dd\omega]
\end{align}
because $\int_\Omega q(\omega_0)\probd[\dd\omega]=0$ and because
$\cov_{n,h,{}^a\omega}[X_a,L_n]$ and $\cov_{n,h,{}^b\omega}[X_b,L_n]$
are independent of $\omega_a$ and $\omega_b$, respectively, while
$\cov_{n,h,{}^{a,b}\omega}[X_a,L_n]$ and
$\cov_{n,h,{}^{a,b}\omega}[X_b,L_n]$ are independent of both
$\omega_a$ and $\omega_b$. The Cauchy--Schwarz inequality, combined
with the fact that
$|\cov_{n,h,{}^a\omega}[X_a,X_c]-\cov_{n,h,{}^{a,b}\omega}[X_a,X_c]|\le
2$ and
$|\cov_{n,h,{}^b\omega}[X_b,X_d]-\cov_{n,h,{}^{a,b}\omega}[X_b,X_d]|\le
2$, yields
\begin{align}
  \nonumber
  &\Exd\big[\Lambda_{n,a}\Lambda_{n,b}\big]\\
  \nonumber
  &\qquad \le 2\sum_{c=1}^n\sqrt{\int_\Omega \big|q(\omega_a)q(\omega_b)\big|\Big|\cov_{n,h,{}^a\omega}[X_a,X_c]-\cov_{n,h,{}^{a,b}\omega}[X_a,X_c]\Big|\,\probd[\dd\omega]}\\
  &\qquad\quad\times \sum_{d=1}^n\sqrt{\int_\Omega \big|q(\omega_a)q(\omega_b)\big|\Big|\cov_{n,h,{}^b\omega}[X_b,X_d]-\cov_{n,h,{}^{a,b}\omega}[X_b,X_d]\Big|\,\probd[\dd\omega]}\,.
  \label{eq:wh_positive_30}
\end{align}
The two factors in the right-hand side of (\ref{eq:wh_positive_30})
just differ because $a\neq b$, but they can be treated in the same
way. Let us look at the first.  We use that
$\partial_{\omega_b}\cov_{n,h,{}^a\omega}[X_a,X_c]=U_{n,h,{}^a\omega}(a,b,c)=U_{n,h,{}^a\omega}(l,m,r)$,
where $U_{n,h,{}^a\omega}(a,b,c)$ is the joint cumulant of $X_a$,
$X_b$, and $X_c$ with respect to the law $\prob_{n,h,{}^a\omega}$ and
$(l,m,r)$ is a permutation of $(a,b,c)$ such that $l\le m\le
r$. Therefore, Lagrange's mean value theorem states that for every
$\omega:=\{\omega_e\}_{e\in\N_0}\in\Omega$ and $c\in\{1,\ldots,n\}$
there exists $z\in[0,1]$ that verifies
\begin{align}
  \nonumber
  &\cov_{n,h,{}^a\omega}[X_a,X_c]-\cov_{n,h,{}^{a,b}\omega}[X_a,X_c]\\
  \nonumber
  &\qquad=\omega_b\,U_{n,h,\{\omega_0,\ldots,\omega_{a-1},0,\omega_{a+1},\ldots,\omega_{b-1},z\omega_b,\omega_{b+1},\ldots\}}(l,m,r)\,.
\end{align}
Setting
$\varpi:=\{\omega_0,\ldots,\omega_{a-1},0,\omega_{a+1},\ldots,\omega_{b-1},z\omega_b,\omega_{b+1},\ldots\}$
for brevity, bound (\ref{eq:per_wh_positive}) gives
\begin{align}
  \nonumber
  \Big|U_{n,h,\varpi}(l,m,r)\Big|&\le
  16\sum_{i=0}^{n-1}\sum_{j\in\N} \mathds{1}_{\{i<l\le j+i,\,r-m\le m-l<j\}}\Ex_{j+1,h,\vartheta^i\varpi}^{\otimes 2}\bigg[\prod_{k=1}^j(1-X_kX_k')\bigg]\\
  \nonumber
  &\quad+16\sum_{i=0}^{n-1}\sum_{j\in\N} \mathds{1}_{\{i<m\le j+i,\,m-l\le r-m<j\}}\Ex_{j+1,h,\vartheta^i\varpi}^{\otimes 2}\bigg[\prod_{k=1}^j(1-X_kX_k')\bigg]\, ,
\end{align}
with
\begin{align}
    \nonumber
    &\Ex_{j+1,h,\vartheta^i\varpi}^{\otimes 2}\bigg[\prod_{k=1}^j(1-X_kX_k')\bigg]\\
    \nonumber
    &\qquad=\frac{\Ex^{\otimes 2}[\prod_{k=1}^j(1-X_kX_k')\ee^{\sum_{e=1}^{j+1}(h+\varpi_{i+e})(X_e+X_e')}X_{j+1}X_{j+1}']}
         {\Ex^{\otimes 2}[\ee^{\sum_{e=1}^{j+1}(h+\varpi_{i+e})(X_e+X_e')}X_{j+1}X_{j+1}']}\\
         \nonumber
         &\qquad\le \ee^{2|\omega_a|+2|\omega_b|}\,\frac{\Ex^{\otimes 2}[\prod_{k=1}^j(1-X_kX_k')\ee^{\sum_{e=1}^{j+1}(h+\omega_{i+e})(X_e+X_e')}X_{j+1}X_{j+1}']}
             {\Ex^{\otimes 2}[\ee^{\sum_{e=1}^{j+1}(h+\omega_{i+e})(X_e+X_e')}X_{j+1}X_{j+1}']}\\
             \nonumber
             &\qquad=\ee^{2|\omega_a|+2|\omega_b|}\,\Ex_{j+1,h,\vartheta^i\omega}^{\otimes 2}\bigg[\prod_{k=1}^j(1-X_kX_k')\bigg]\,.
  \end{align}
Putting the pieces together we find
\begin{align}
  \nonumber
  &\Big|\cov_{n,h,{}^a\omega}[X_a,X_c]-\cov_{n,h,{}^{a,b}\omega}[X_a,X_c]\Big|\\
  \nonumber
  &\qquad\le 16|\omega_b|\ee^{2|\omega_a|+2|\omega_b|}
  \sum_{i=0}^{n-1}\sum_{j\in\N} \mathds{1}_{\{i<l\le j+i,\,r-m\le m-l<j\}}\Ex_{j+1,h,\vartheta^i\omega}^{\otimes 2}\bigg[\prod_{k=1}^j(1-X_kX_k')\bigg]\\
  \nonumber
  &\qquad\quad+16|\omega_b|\ee^{2|\omega_a|+2|\omega_b|}\sum_{i=0}^{n-1}\sum_{j\in\N} \mathds{1}_{\{i<m\le j+i,\,m-l\le r-m<j\}}
  \Ex_{j+1,h,\vartheta^i\omega}^{\otimes 2}\bigg[\prod_{k=1}^j(1-X_kX_k')\bigg]\,.
\end{align}
Next, recalling that $|q(\zeta)|\le\max\{q_-,q_+\}\lambda$ for
$\zeta\in(-\lambda,\lambda]$ and $q(\zeta)=0$ for
  $\zeta\notin(-\lambda,\lambda]$, multiplying by
    $|q(\omega_a)q(\omega_b)|$ and integrating with respect to
    $\probd[\dd\omega]$ we get  
    \begin{align}
      \nonumber
      &\int_\Omega \big|q(\omega_a)q(\omega_b)\big|\Big|\cov_{n,h,{}^a\omega}[X_a,X_c]-\cov_{n,h,{}^{a,b}\omega}[X_a,X_c]\Big|\,\probd[\dd\omega]\\
      \nonumber
      &\quad\le \lambda^3\Big(4\max\{q_-,q_+\}\ee^{2\lambda}\Big)^{\!2}
  \sum_{i=0}^{n-1}\sum_{j\in\N} \mathds{1}_{\{i<l\le j+i,\,r-m\le m-l<j\}}\Exd\Bigg[\Ex_{j+1,h,\cdot}^{\otimes 2}\bigg[\prod_{k=1}^j(1-X_kX_k')\bigg]\Bigg]\\
  \nonumber
  &~~\quad+\lambda^3\Big(4\max\{q_-,q_+\}\ee^{2\lambda}\Big)^{\!2}\sum_{i=0}^{n-1}\sum_{j\in\N} \mathds{1}_{\{i<m\le j+i,\,m-l\le r-m<j\}}
  \Exd\Bigg[\Ex_{j+1,h,\cdot}^{\otimes 2}\bigg[\prod_{k=1}^j(1-X_kX_k')\bigg]\Bigg]\\
  \nonumber
&\quad \le2\lambda^3\Big(4\max\{q_-,q_+\}\ee^{2\lambda}\Big)^{\!2}\sum_{j\in\N} \mathds{1}_{\{j>\max\{m-l,r-m\}\}}
  j\,\Exd\Bigg[\Ex_{j+1,h,\cdot}^{\otimes 2}\bigg[\prod_{k=1}^j(1-X_kX_k')\bigg]\Bigg]\,.
    \end{align}
    At this point, using \eqref{eq:wh_positive_12345} and the bound
    $\mathds{1}_{\{j>\zeta\}}\le \ee^{\gamma(j-\zeta)/2}$ for
    $\zeta\in\Rl$ we can state that
\begin{align}
      \nonumber
      &\int_\Omega \big|q(\omega_a)q(\omega_b)\big|\Big|\cov_{n,h,{}^a\omega}[X_a,X_c]-\cov_{n,h,{}^{a,b}\omega}[X_a,X_c]\Big|\,\probd[\dd\omega]\\
      \nonumber
      &\qquad\qquad\qquad\qquad\qquad\le \frac{2\lambda^3G(4\max\{q_-,q_+\}\ee^{2\lambda})^2}{\big(\ee^{\frac{\gamma}{2}}-1\big)^2}\ee^{-\frac{\gamma}{2}\max\{m-l,r-m\}}\,.
    \end{align}
Finally, by extracting the square root and carrying out the sum over
$c$, simple algebra yields the inequality
 \begin{align}
      \nonumber
      &\sum_{c=1}^n\sqrt{\int_\Omega \big|q(\omega_a)q(\omega_b)\big|\Big|\cov_{n,h,{}^a\omega}[X_a,X_c]-\cov_{n,h,{}^{a,b}\omega}[X_a,X_c]\Big|\,\probd[\dd\omega]}\\
      \nonumber
      &\qquad\qquad\qquad\qquad\qquad\qquad\qquad\le \frac{\sqrt{2\lambda^3G}\,20\max\{q_-,q_+\}\ee^{2\lambda}}{\big(\ee^{\frac{\gamma}{2}}-1\big)\big(1-\ee^{-\frac{\gamma}{4}}\big)}
      \,\ee^{-\frac{\gamma}{8}(b-a)}\,.
 \end{align}
The same arguments show that the very same bound holds for the second
factor in the right-hand side of (\ref{eq:wh_positive_30}), so that
\begin{equation}
  \Exd\big[\Lambda_{n,a}\Lambda_{n,b}\big]
  \le \frac{4\lambda^3G\,\big(20\max\{q_-,q_+\}\ee^{2\lambda}\big)^2}{\big(\ee^{\frac{\gamma}{2}}-1\big)^2\big(1-\ee^{-\frac{\gamma}{4}}\big)^2}\,\ee^{-\frac{\gamma}{4}(b-a)}\,.
  \label{eq:wh_positive_51}
\end{equation}

In conclusion, starting from (\ref{eq:wh_positive_49}), the
inequalities (\ref{eq:wh_positive_50}) and (\ref{eq:wh_positive_51})
give
\begin{equation*}
  \frac{1}{n}\Exd\bigg[\bigg(\sum_{a=1}^n\Lambda_{n,a}\bigg)^{\!\!2}\bigg]
  \le G\ee^{\lambda}\bigg(\max\{q_-,q_+\}\lambda\,\frac{1+\ee^{-\frac{\gamma}{2}}}{1-\ee^{-\frac{\gamma}{2}}}\bigg)^{\!2}
  +\frac{4\lambda^3G\,\big(20\max\{q_-,q_+\}\ee^{2\lambda}\big)^2}{\big(\ee^{\frac{\gamma}{2}}-1\big)^2\big(1-\ee^{-\frac{\gamma}{4}}\big)^3}\,,
\end{equation*}
which proves (\ref{eq:wh_positive_2}) thanks to the arbitrariness of $n$.
\end{proof}

Putting the pieces together, we are finally able to demonstrate the CLT
for the centering variable $\Ex_{n,h,\cdot}[L_n]$, which is the last
statement of Theorem \ref{th:centering}.

\begin{proof}[Proof of part (iv) of Theorem \ref{th:centering}]
The proof follows the argument of the proof of part (ii) of Theorem
\ref{th:CLT+concentration}.  Suppose that $\int_\Omega\omega_0^2\,
\probd[\dd\omega]>0$, so that $w(h)>0$ for all $h>h_c$ according to
Lemma \ref{lem:wh_positive}, and denote by $W_{n,h}$ the variance of
$\Ex_{n,h,\cdot}[L_n]$ with respect to the law $\probd$:
\begin{equation*}
  W_{n,h}:=\Exd\bigg[\Big(\Ex_{n,h,\cdot}[L_n]-\Exd\big[\Ex_{n,h,\cdot}[L_n]\big]\Big)^{\!2}\bigg]\,.
\end{equation*}
Given a compact set $H\subset(h_c,+\infty)$, part (iii) of Theorem
\ref{th:centering} tells us that
\begin{equation}
\adjustlimits\lim_{n\uparrow\infty}\sup_{h\in H}\bigg|\frac{W_{n,h}}{n}-w(h)\bigg|=0\,,
\label{eq:stand_dev_EL}
\end{equation}
while Lemma \ref{lem:cumulants_EL} assures us that there exists a
constant $c>0$ such that $|{}^{r\!}\mathcal{K}_{n,h}|\le c^r(r!)^3 n$
for all $r,n\in\N$ and $h\in H$, ${}^{r\!}\mathcal{K}_{n,h}$ being the
cumulant of order $r$ of $\Ex_{n,h,\cdot}[L_n]$. Thus, noting that
${}^{r\!}\mathcal{K}_{n,h}/\sqrt{W_{n,h}^r}$ is the cumulant of order
$r\ge 2$ of the centered, unit-variance random variable
$(\Ex_{n,h,\cdot}[L_n]-\Exd[\Ex_{n,h,\cdot}[L_n]])/\sqrt{W_{n,h}}$, we
have the following \textit{Statulevi\v{c}ius condition}: for every
$n\in\N$, $h\in H$, and integer $r\ge 3$
\begin{equation*}
  \Bigg|\frac{{}^{r\!}\mathcal{K}_{n,h}}{\sqrt{W_{n,h}^r}}\Bigg|\le (r!)^3
  \Bigg(\max\bigg\{\frac{c}{\sqrt{W_{n,h,}}},\frac{c^3n}{\sqrt{W_{n,h}^3}}\bigg\}\Bigg)^{\!\!r-2}\,.
\end{equation*}
We observe that $\inf_{h\in H}w(h)=w(h_o)>0$ for some $h_o\in H$, as
$w$ is continuous throughout $(h_c,+\infty)$, and that $\inf_{h\in
  H}W_{n,h}\ge w(h_o)n/4$ for all sufficiently large
$n$ by \eqref{eq:stand_dev_EL}.  In this way, \cite[Theorem
  2.4]{doring2022} yields for all $n\in\N$ and $h\in H$
\begin{align}
  \nonumber
 &\sup_{u\in\Rl}\,\Bigg|\probd\bigg[\frac{\Ex_{n,h,\cdot}[L_n]-\Exd[\Ex_{n,h,\cdot}[L_n]]}{\sqrt{W_{n,h}}}\le u\bigg]-
  \frac{1}{\sqrt{2\pi}}\int_{-\infty}^u\ee^{-\frac{1}{2}z^2}\dd z\Bigg|\\
  \nonumber
  &\qquad\qquad\qquad\qquad\qquad\qquad\qquad\qquad
  \le 145  \Bigg(\max\bigg\{\frac{c}{\sqrt{W_{n,h}}},\frac{c^3n}{\sqrt{W_{n,h}^3}}\bigg\}\Bigg)^{\!\!\frac{1}{5}}\,.
\end{align}
Similarly to the proof of part (ii) of Theorem
\ref{th:CLT+concentration}, replacing $u$ with
$\sqrt{nw(h)/W_{n,h}}\,u$ in the left-hand side of the above
bound we obtain for all sufficiently large $n$ and $h\in H$
\begin{align}
  \nonumber
 &\sup_{u\in\Rl}\,\Bigg|\probd\bigg[\frac{\Ex_{n,h,\cdot}[L_n]-\Exd[\Ex_{n,h,\cdot}[L_n]]}{\sqrt{nw(h)}}\le u\bigg]-
  \frac{1}{\sqrt{2\pi}}\int_{-\infty}^u\ee^{-\frac{1}{2}z^2}\dd z\Bigg|\\
  \nonumber
  &\qquad\qquad\qquad\le 145\Bigg(\max\bigg\{\frac{c}{\sqrt{W_{n,h}}},\frac{c^3n}{\sqrt{W_{n,h}^3}}\bigg\}\Bigg)^{\!\!\frac{1}{5}}
  +\frac{|W_{n,h}-nw(h)|}{nw(h)}\,,
\end{align}
so
\begin{align}
  \nonumber
 &\sup_{h\in H}\,\sup_{u\in\Rl}\,\Bigg|\probd\bigg[\frac{\Ex_{n,h,\cdot}[L_n]-\Exd[\Ex_{n,h,\cdot}[L_n]]}{\sqrt{nw(h)}}\le u\bigg]-
  \frac{1}{\sqrt{2\pi}}\int_{-\infty}^u\ee^{-\frac{1}{2}z^2}\dd z\Bigg|\\
  \nonumber
  &\qquad\le \frac{145}{n^{\frac{1}{10}}}\Bigg(\max\bigg\{\frac{2c}{\sqrt{w(h_o)}},\frac{8c^3}{\sqrt{w(h_o)^3}}\bigg\}\Bigg)^{\!\!\frac{1}{5}}
  +\frac{1}{w(h_o)}\sup_{h\in H}\bigg|\frac{W_{n,h}}{n}-w(h)\bigg|\,.
\end{align}
This proves part (iv) of Theorem \ref{th:centering} thanks to
\eqref{eq:stand_dev_EL}.
\end{proof}

\subsection{Quenched fluctuations}

The concentration bound in part (i) of Theorem \ref{th:centering}
allows us to provide a rough estimate of the rate by which
$\Ex_{n,h,\omega}[L_n]/n$ approaches $\rho(h):=\partial_hf(h)$ for
typical $\omega$ as $n$ goes to infinity, as stated by Proposition
\ref{prop:fluctuations_EL}.

\begin{proof}[Proof of Proposition \ref{prop:fluctuations_EL}]
We show that there exists $\Omega_o\in\mathcal{F}$ with
$\probd[\Omega_o]=1$ such that for all $\omega\in\Omega_o$ and
$H\subset(h_c,+\infty)$ compact
\begin{equation}
  \label{prop:fluctuations_EL_0}
    \limsup_{n\uparrow\infty}\,\sup_{h\in H}\frac{|\Ex_{n,h,\omega}[L_n]-\Exd[\Ex_{n,h,\cdot}[L_n]]|}{\sqrt{n\log n}}<+\infty\,.
  \end{equation}
Then, the proposition follows from part (ii) of Theorem
\ref{th:centering}.

For $s\in\N$, put $h_s:=-s$ if $h_c=-\infty$ and $h_s:=h_c+1/s$ if
$h_c>-\infty$, so that $h_s>h_c$ and
$\lim_{s\uparrow\infty}h_s=h_c$. According to part (i) of Theorem
\ref{th:centering} there exists a constant $\kappa_s>0$ such that for
every $n\in\N$ and $u\ge 0$
 \begin{equation*}
\sup_{h\in [h_s,+\infty)}\probd\bigg[\Big|\Ex_{n,h,\cdot}[L_n]-\Exd\big[\Ex_{n,h,\cdot}[L_n]\big]\Big|>u\bigg]\le 2\ee^{-\frac{\kappa_s u^2}{n+u^{5/3}}}\,.
 \end{equation*}
Let $\Lambda_{n,s}$ be the random variable that maps $\omega\in\Omega$
to
\begin{equation*}
  \Lambda_{n,s}(\omega):=\max_{i\in\{1,\ldots,n^3\}}
  \Bigg\{\bigg|\Ex_{n,h_s+i/n^2,\omega}[L_n]-\Exd\Big[\Ex_{n,h_s+i/n^2,\cdot}[L_n]\Big]\bigg|\Bigg\}\,.
\end{equation*}
Defining $\lambda_s:=10/\kappa_s$ for brevity, for all sufficiently
large $n$ we have $(\lambda_sn\log n)^\frac{5}{6}\le n$ and
\begin{align}
  \nonumber
  &\probd\Big[\Lambda_{n,s}>\sqrt{\lambda_sn\log n}\Big]\\
  \nonumber
  &\qquad\le\sum_{i=1}^{n^3}\probd\Bigg[\bigg|\Ex_{n,h_s+i/n^2,\cdot}[L_n]-\Exd\Big[\Ex_{n,h_s+i/n^2,\cdot}[L_n]\Big]\bigg|>\sqrt{\lambda_sn\log n}\Bigg]\\
  \nonumber
  &\qquad\le 2n^3\ee^{-\frac{\kappa_s \lambda_sn\log n}{n+(\lambda_sn\log n)^{5/6}}}\le 2n^3\ee^{-\frac{\kappa_s\lambda_s}{2}\log n}=\frac{2}{n^2}\,.
 \end{align}
Thus, the Borel--Cantelli lemma ensures us that there exists
$\Omega_o\in\mathcal{F}$ with $\probd[\Omega_o]=1$ such that
\begin{equation}
  \limsup_{n\uparrow\infty}\frac{\Lambda_{n,s}(\omega)}{\sqrt{\lambda_sn\log n}}\le 1
  \label{prop:fluctuations_EL_123}
\end{equation}
for any $\omega\in\Omega_o$ and $s\in\N$. This implies
(\ref{prop:fluctuations_EL_0}). In fact, given $\omega\in\Omega_o$ and
$H\subset(h_c,+\infty)$ compact there exists $s\in\N$ such that
$[h_s,h_s+n]\supset H$ for all sufficiently large $n$. In turn, for
all these large $n$ and every $h\in H$ there exists
$i\in\{1,\ldots,n^3\}$ such that $h_s+(i-1)/n^2\le h\le
h_s+i/n^2$. Recalling that $L_n\le n$, it follows that
\begin{equation*}
  \Ex_{n,h,\omega}[L_n]=\frac{\Ex[L_n\ee^{\sum_{a=1}^n(h+\omega_a)X_a}X_n]}{\Ex[\ee^{\sum_{a=1}^n(h+\omega_a)X_a}X_n]}\\
  \le\ee^{\frac{1}{n}}\,\Ex_{n,h_s+i/n^2,\omega}[L_n]\,.
\end{equation*}
Using that $\ee^\zeta\le 1+\zeta\ee^\zeta$ for $\zeta\ge 0$ and,
again, that $L_n\le n$ we find
\begin{equation}
  \label{prop:fluctuations_EL_1}
  \Ex_{n,h,\omega}[L_n]\le\Ex_{n,h_s+i/n^2,\omega}[L_n]+\ee\le\Ex_{n,h_s+i/n^2,\omega}[L_n]+3\,.
  \end{equation}
Similarly, we have
\begin{equation}
  \Ex_{n,h,\omega}[L_n]\ge\ee^{-\frac{1}{n}}\,\Ex_{n,h_s+i/n^2,\omega}[L_n]\ge\Ex_{n,h_s+i/n^2,\omega}[L_n]-1\,.
  \label{prop:fluctuations_EL_2}
\end{equation}
Bounds (\ref{prop:fluctuations_EL_1}) and
(\ref{prop:fluctuations_EL_2}) show that
\begin{align}
  \nonumber
  &\Big|\Ex_{n,h,\omega}[L_n]-\Exd\big[\Ex_{n,h,\cdot}[L_n]\big]\Big|\\
  \nonumber
  &\qquad\le\bigg|\Ex_{n,h_s+i/n^2,\omega}[L_n]-\Exd\Big[\Ex_{n,h_s+i/n^2,\cdot}[L_n]\Big]\bigg|+4\le \Lambda_{n,s}(\omega)+4\,,
\end{align}
and the arbitrariness of $h$ and $n$ gives
\begin{equation*}
\limsup_{n\uparrow\infty}\,\sup_{h\in H}\frac{|\Ex_{n,h,\omega}[L_n]-\Exd[\Ex_{n,h,\cdot}[L_n]]|}{\sqrt{\lambda_sn\log n}}\le 1
\end{equation*}
thanks to (\ref{prop:fluctuations_EL_123}).
\end{proof}

\begin{acks}[Acknowledgments]
 We thank Quentin Berger and Hubert Lacoin for insightful discussions.
\end{acks}

\begin{funding}
G.\ G.\ acknowledges the support of the Cariparo Foundation.
M.\ Z.\ is a member of INdAM-GNAMPA (Gruppo Nazionale per l'Analisi Matematica, la Probabilit\`a e le loro Applicazioni).
\end{funding}


\begin{thebibliography}{}

\bibitem{cf:AW} M. Aizenman and J. Wehr, {Rounding effects of quenched randomness on first-order phase transitions}, Comm. Math. Phys. {\bf 130} (1990), 489-528.

\bibitem{cf:AZ96}
S.~Albeverio and X.~Y.~Zhou, 
Free energy and some sample path properties of a random walk with random potential,
J.\ Stat.\ Phys.\ \textbf{83} (1996), 573--622.

\bibitem{cf:AZ14}
  K.~S.~Alexander and N.~Zygouras,
  Path properties of the disordered pinning model in the delocalized regime,
  Ann.\ Appl.\ Prob.\  \textbf{24} (2014), 599--615.
  
\bibitem{bingham1989} N.~H.\ Bingham, C.~M.\ Goldie, and J.~L.\ Teugels,
  \textit{Regular Variation} (Cambridge University Press, Cambridge, 1989).
  
\bibitem{cf:notesQB}
Q.~Berger,
\textit{Interfaces et Polym\`eres Al\'eatoires},  \href{https://perso.lpsm.paris/~bergerq/documents/Interfaces_Polymeres.pdf}{lecture notes} (2023).  

%\bibitem{campanino1979} M.\ Campanino, D.\ Capocaccia, and B.\ Tirozzi,
%  The local central limit theorem for Gibbs random fields,
%  Commun.\ Math.\ Phys.\ \textbf{70} (1979), 125--132.

\bibitem{caravenna2013} F.\ Caravenna and F.\ den Hollander,
  A general smoothing inequality for disordered polymers,
Electron.\ Commun.\ Probab.\ \textbf{18} (2013), 1--15.

 
\bibitem{cf:CGT12} F.~Caravenna, G.~Giacomin and F.~L.~Toninelli,
Copolymers at selective interfaces: settled issues and open problems,
Springer Proc.\ Math.\ \textbf{11} (2012), 289--311. 

%\bibitem{caravenna2013} F.\ Caravenna and F.\ den Hollander,
%  A general smoothing inequality for disordered polymers,
%Electron.\ Commun.\ Probab.\ \textbf{18} (2013), 1--15.

\bibitem{cf:CCP2019}
D.~Cheliotis,  Y.~Chino and J.~Poisat, Julien,
The random pinning model with correlated disorder given by a renewal set,
Ann. H. Lebesgue {\bf 2} (2019), 281-329.

%\bibitem{cf:DingHuandMaia}
%J.~Ding,  F.~Huang and J.~Maia,  \textit{Phase transitions in low-dimensional long-range random field Ising models}, arXiv:2412.19281
	


\bibitem{cf:DingXia} J.~Ding and J.~Xia, 
{Exponential decay of correlations in the two-dimensional random field Ising model},
Invent. Math.   {\bf 224} (2021),  999-1045.


%\bibitem{cf:DT77} R.~L.~Dobrushin and B.~Tirozzi, 
%The central limit theorem and the problem of equivalence of ensembles,
%Comm. Math. Phys. \textbf{54} (1977),  173--192.
  
\bibitem{doring2022} H.\ D\"oring, S.\ Jansen, and K.\ Schubert,
  The method of cumulants for the normal approximation,
  Probab.\ Surv.\ \textbf{19} (2022), 185--270.
  
\bibitem{cf:taming95}
 H.~von Dreifus, A.~Klein, and J.~Fernando Perez,
Taming Griffiths' singularities: infinite differentiability of quenched correlation functions,
Comm. Math. Phys. \textbf{170} (1995), 21--39.
  
\bibitem{giacomin2007} G.\ Giacomin,
  \textit{Random Polymer Models}
  (Imperial College Press, World Scientific, 2007).
  
\bibitem{cf:G-SF} G.\ Giacomin,
  \textit{Disorder and critical phenomena through basic probability models},
  Lecture Notes in Mathematics \textbf{2025}
  (Springer, Heidelberg, 2011).
  
\bibitem{cf:GG-MPAG2022} G.\ Giacomin and R.~L.\ Greenblatt,
The zeros of the partition function of the pinning model,
Math.\ Phys.\ Anal.\ Geom.\ \textbf{25} (2022), 16.
  
\bibitem{giacomin2020} G.\ Giacomin and B.\ Havret,
  Localization, big-jump regime and the effect of disorder for a class of generalized pinning models,
  J.\ Stat.\ Phys.\ \textbf{181} (2020), 2015--2049.
  
   
\bibitem{giacomin2006_1} G.\ Giacomin and F.~L.\ Toninelli,
The localized phase of disordered copolymers with adsorption, 
Lat.\ Am.\ J.\ Probab.\ \textbf{1} (2006), 149--180.

\bibitem{giacomin2006_2} G.\ Giacomin and F.L.\ Toninelli,
  Smoothing effect of quenched disorder on polymer depinning transitions,
  Commun.\ Math.\ Phys.\ \textbf{266} (2006), 1--16.


\bibitem{cf:GTirrel2009}
G. Giacomin and F. L. Toninelli,
On the irrelevant disorder regime of pinning models,
Ann. Probab. {\bf 37} (2009), 1841-1875. 

\bibitem{giacominzamparo} G.\ Giacomin and M.\ Zamparo,
  The effect of disorder on the big-jump phenomenon in the pinning model,
  in preparation.

%\bibitem{giacomin2006_2} G.\ Giacomin and F.L.\ Toninelli,
%  Smoothing effect of quenched disorder on polymer depinning transitions,
%  Commun.\ Math.\ Phys.\ \textbf{266} (2006), 1--16.

%\bibitem{giacomin2006_3} G.\ Giacomin and F.L.\ Toninelli,
%  Smoothing of depinning transitions for directed polymers with quenched disorder,
%  Phys.\ Rev.\ Lett.\ \textbf{96} (2006), 070602.
  
  \bibitem{cf:dH}
    F.\ den Hollander,
    \textit{Random polymers}, Lecture Notes in Mathematics \textbf{1974}
    (Springer-Verlag, Berlin, 2009).

\bibitem{feller1943} W.\ Feller,
The general form of the so-called law of the iterated logarithm,
Trans.\ Amer.\ Math.\ Soc.\ \textbf{54} (1943), 373-402.

    
%\bibitem{janson1988} S.\ Janson,
%  Normal convergence by higher semiinvariants with applications to sums of dependent random variables and random graphs,
%  Ann.\ Probab.\ \textbf{16} (1988), 305--312.
  
%\bibitem{martin1973} A.\ Martin-L\"of,
%  Mixing properties, differentiability of the free energy and the central limit theorem for a pure phase in the Ising model at low temperature,
%  Commun.\ Math.\ Phys.\ \textbf{32} 75--92 (1973)

\bibitem{cf:KM2003} Y.\ Kafri and D.\ Mukamel,
  Griffiths singularities in unbinding of strongly disordered polymers,
  Phys.\ Rev.\ Lett.\ \textbf{91} (2003), 055502.

\bibitem{cf:LS2017}
H.~Lacoin and J.~Sohier, 
Disorder relevance without Harris criterion: the case of pinning model with $\gamma$-stable environment,
Electron.\ J.\ Probab.\ \textbf{22} (2017), 50.

\bibitem{maurer2021} A.\ Maurer and M.\ Pontil,
Concentration inequalities under sub-gaussian and sub-exponential conditions,
Adv.\ Neural Inf.\ Process.\ \textbf{34} (2021), 7588--97.

%\bibitem{moricz1982}  F.A.\ M\'oricz, R.J.\ Serfling, and W.F.\ Stout,
%  Moment and probability bounds with quasi-superadditive structure for the maximum partial sum,
%Ann.\ Probab.\ \textbf{10} (1982), 1032--1040.
  
%\bibitem{mcdiarmid1989} C.J.H.\ McDiarmid,
%  On the method of bounded differences, in \textit{Surveys in Combinatorics, 1989}, J. Siemons ed.,
%(Cambridge University Press, Cambridge, 1989), pp. 148--188

\bibitem{peccati2011} G.\ Peccati and M.~S.\ Taqqu,
  \textit{Wiener Chaos: Moments, Cumulants and Diagrams - A survey with computer implementation}
  (Springer, Milan, 2011).
  
 % \bibitem{cf:Poisat}
  %J.~Poisat, On quenched and annealed critical curves of random pinning model with finite range correlations, Ann. Inst. Henri Poincar\'e Probab. Stat. {\bf 49} (2013), 456-482.



%\bibitem{petrov1975} V.V.\ Petrov,
%  \textit{Sums of independent random variables}
%  (Springer-Verlag, Berlin, 1975).

\bibitem{shlosman1986} S.~B.\ Shlosman,
  Signs of the Ising model Ursell functions,
Commun.\ Math.\ Phys.\ \textbf{102} (1986), 679--686. 
  
\bibitem{sylvester1975} G.~S.\ Sylvester,
  Representations and inequalities for Ising model Ursell functions,
  Commun.\ Math.\ Phys.\ \textbf{42} (1975), 209--220.
  
 \bibitem{cf:Ton-AAP08}
 F.L.~Toninelli, 
Disordered pinning models and copolymers: beyond annealed bounds,
Ann.\ Appl.\ Probab.\ \textbf{18} (2008),  1569--1587. 
  

  \bibitem{cf:Velenik} Y.~Velenik, 
Localization and delocalization of random interfaces,  
Probab. Surv. \textbf{3} (2006), 112--169. 

 \bibitem{cf:VBB}
A.~Vezzani, E.~Barkai and R.~Burioni, {Single-big-jump principle in physical modeling},
Phys. Rev. E {\bf 100} (2019), 012108

%\bibitem{zamparo2019} M.\ Zamparo,
%  Large deviations in renewal models of statistical mechanics,
%  J.\ Phys.\ A: Math.\ Theor.\ \textbf{52} (2019), 495004.



\bibitem{zamparo2021} M.\ Zamparo,
  Critical fluctuations in renewal models of statistical mechanics,
J.\ Math.\ Phys.\ \textbf{62} (2021), 113301.

\bibitem{zamparo2022} M.\ Zamparo,
  Renewal model for dependent binary sequences,
  J.\ Stat.\ Phys.\ \textbf{187} (2022), 5 (34 pp.). 
  
  \bibitem{cf:watbled} F.\ Watbled, Concentration inequalities for disordered models,
  Lat.\ Am.\ J.\ Probab.\ \textbf{9} (2012), 129--140.

 \end{thebibliography}
\end{document}